\documentclass{amsart}
\usepackage{mathrsfs}
\usepackage{amssymb,amsmath}
\usepackage{graphicx}
\usepackage{comment}
\usepackage{tikz}
\usepackage{tikz-cd}
\usepackage[all,cmtip]{xy}
\usepackage[pdftex]{hyperref}
\usepackage{flowchart}
\usetikzlibrary{matrix,calc,arrows}

\newcommand{\bC}{{\mathbb C}}
\newcommand{\bP}{{\mathbb P}}
\newcommand{\bR}{{\mathbb R}}
\newcommand{\bZ}{{\mathbb Z}}
\newcommand{\bN}{{\mathbb N}}

\newcommand{\bF}{{\mathbb F}}
\newcommand{\bQ}{{\mathbb Q}}

\newcommand{\bRZ}{\bR/2 \pi \bZ}

\newcommand{\fF}{\pi}

\newcommand{\cF}{\mathcal F}

\newcommand{\cA}{\mathcal A}
\newcommand{\cB}{\mathcal B}
\newcommand{\cC}{\mathcal C}
\newcommand{\cD}{\mathcal D}
\newcommand{\cE}{\mathcal E}

\newcommand{\cX}{\mathcal X}

\newcommand{\cN}{\mathcal N}
\newcommand{\cM}{\mathcal M}
\newcommand{\cO}{\mathcal O}

\newcommand{\cU}{\mathcal U}
\newcommand{\cV}{\mathcal V}
\newcommand{\cL}{\mathcal L}
\newcommand{\cR}{\mathcal R}
\newcommand{\cT}{\mathcal T}
\newcommand{\cH}{\mathcal H}
\newcommand{\cJ}{\mathcal J}

\newcommand{\scrF}{\mathscr F}

\newcommand{\scrN}{\mathscr N}

\newcommand{\scrP}{\mathscr P}

\newcommand{\fd}{\mathfrak{d}}
\newcommand{\Mbar}{\overline{\cM}}
\newcommand{\Rbar}{\overline{\cR}}

 \newcommand{\Aut}{\operatorname{Aut}}
\newcommand{\Ob}{\operatorname{Ob}}

\newcommand{\id}{\operatorname{id}}

\newcommand{\Ch}{\operatorname{Ch}}
\newcommand{\Floer}{\operatorname{CF}}
\newcommand{\Hom}{\operatorname{Hom}}

\newcommand{\sqset}{\mathbf{scSet}}
\newcommand{\Top}{\mathbf{Top}}

\newcommand{\Color}{\operatorname{Col}}

\DeclareMathOperator*{\hocolim}{\operatorname{hocolim}}

\DeclareMathOperator{\Sec}{Sec}
\DeclareMathOperator{\inj}{inj}
\def\co{\colon\thinspace}

\newcommand{\pairofpants}{    \draw (180:.5 and .1) arc   (180:360:.5 and .1);
    \draw[gray,dashed] (0:.5 and .1) arc   (0:180:.5 and .1);
    \begin{scope}[shift={(1,2)}]
      \draw (0,0) ellipse  (.5 and .1);
    \end{scope}
     \begin{scope}[shift={(-1,2)}]
       \draw (0,0) ellipse  (.5 and .1);
     \end{scope}
     \draw (.5,0) .. controls (.5,1) and (1.5,1) .. (1.5,2);
     \draw (-.5,0) .. controls (-.5,1) and (-1.5,1) .. (-1.5,2);
     \draw (.5,2) .. controls (.5,1) and (-.5,1) .. (-.5,2);
 }

 \newcommand{\threeinputs}{ \draw (180:.5 and .1) arc   (180:360:.5 and .1);
         \draw[gray,dashed] (0:.5 and .1) arc   (0:180:.5 and .1);
          \begin{scope}[shift={(0,3)}]
      \draw (0,0) ellipse  (.5 and .1);
    \end{scope}
    \begin{scope}[shift={(2,3)}]
      \draw (0,0) ellipse  (.5 and .1);
    \end{scope}
     \begin{scope}[shift={(-2,3)}]
       \draw (0,0) ellipse  (.5 and .1);
     \end{scope}
     \draw (.5,0) .. controls (.5,1) and (2.5,2) .. (2.5,3);
     \draw (-.5,0) .. controls (-.5,1) and (-2.5,2) .. (-2.5,3);
     \draw (-.5,3) .. controls (-.5,2) and (-1.5,2) .. (-1.5,3);
\draw (.5,3) .. controls (.5,2) and (1.5,2) .. (1.5,3);  }

\numberwithin{equation}{section}

\newtheorem{thm}{Theorem}[section]
\newtheorem{cor}[thm]{Corollary}
\newtheorem{lem}[thm]{Lemma}
\newtheorem{lemma}[thm]{Lemma}
\newtheorem{prop}[thm]{Proposition}
\newtheorem{defin}[thm]{Definition}
\newtheorem{def-lem}[thm]{Definition-Lemma}

\theoremstyle{remark}

\newtheorem{rem}[thm]{Remark}

\newtheorem{example}[thm]{Example}

\author{Mohammed Abouzaid, Yoel Groman, Umut Varolgunes}
\title{Framed $E_2$ structures in Floer theory}
\begin{document}
\maketitle
\begin{abstract}
 We resolve the long-standing problem of constructing the action of the
operad of framed (stable) genus-$0$ curves on Hamiltonian Floer
theory; this operad is equivalent to the framed $E_2$ operad. We
formulate the construction in the following general context: we
associate to each compact subset of a closed symplectic manifold a new
chain-level model for symplectic cohomology with support, which we show
carries an action of a model for the chains on the moduli space of
framed genus $0$ curves. This construction turns out to be strictly
functorial with respect to inclusions of subsets, and the action of
the symplectomorphism group. In the general context, we appeal to
virtual fundamental chain methods to construct the operations over
fields of characteristic $0$, and we give a separate account, over
arbitrary rings, in the special settings where Floer's classical
transversality approach can be applied. We perform all constructions
over the Novikov ring, so that the algebraic structures we produce are
compatible with the quantitative information that is contained in
Floer theory. Over fields of characteristic $0$, our construction can
be combined with results in the theory of operads to produce explicit
operations encoding the structure of a homotopy $BV$ algebra. In an
appendix, we explain how to extend the results of the paper from the
class of closed symplectic manifolds to geometrically bounded ones.
\end{abstract}
\setcounter{tocdepth}{1}
\tableofcontents
\section{Introduction}

\subsection{Hamiltonian Floer theory}
\label{sec:oper-hamilt-floer}

In his study of the Arnol'd conjecture \cite{Floer1989b}, Floer associated a cohomology group to each non-degenerate Hamiltonian $H \co M \times S^1 \to \bR$ on a closed symplectic manifold, based on the gradient flow of the action functional on the free loop space of $M$. Such gradient flow lines correspond to cylinders satisfying a deformation of the holomorphic curve equation. The fact that one can study analogous equations on more general Riemann surfaces was first observed by Donaldson, leading to the construction of an associative product on Floer cohomology associated to pairs of pants, which was later shown by Piunikhin--Salamon--Schwartz \cite{PiunikhinSalamonSchwarz1996} to yield a ring that is isomorphic to the quantum cohomology ring for those manifolds satisfying the property that the class of the symplectic form is a positive multiple of the first Chern class. Separately, and in the technically different context of exact symplectic manifolds with contact boundary, Viterbo \cite{Viterbo1999} introduced a circle action in Floer theory, which takes the form of a degree $-1$ operator
\begin{equation}
  \Delta \co SH^*(M)  \to  SH^{*-1}(M), 
\end{equation}
on a variant of Hamiltonian Floer cohomology, called \emph{symplectic cohomology}, which goes back to Hofer and Floer's work \cite{FloerHofer1994} on the symplectic topology of open subsets of $\bC^n$.

In this paper, we consider a version of Floer cohomology \cite{Seidel2012b,Groman2015,Varolgunes2018,Venkatesh2018}, which we call \emph{symplectic cohomology with prescribed support}, in a change from the previous terminology, which vastly generalises both of these frameworks, but we shall temporarily suppress this point.

The pair of pants product makes sense in the context of symplectic cohomology as well and, together with the operator $\Delta$ introduced by Viterbo, is known  to satisfy the relation
\begin{multline} \label{eq:BV-relation}
  \Delta(xyz) = \Delta(xy) z  + (-1)^{|x|} x \Delta(y z ) + (-1)^{(|x|+1)|y|} z \Delta(xy) \\
  - \Delta(x) yz -  (-1)^{|x|} x \Delta(y) z -  (-1)^{|x| + |y|} x y \Delta(z),
\end{multline}
which implies that symplectic cohomology forms a \emph{Batalin-Vilkovisky algebra}; this is a folklore result, whose proof appears, for example, in \cite[\S 2.5]{Abouzaid2013}.

The geometry giving rise to Equation \eqref{eq:BV-relation} can be expressed in terms of the homology of the moduli space $f \Mbar^{\bR}_{0,4}$ of \emph{framed stable genus $0$ curves,} where the notion of framing corresponds to a choice of tangent ray at each marked point; the left hand side corresponds to the class in the first homology of  $f \Mbar^{\bR}_{0,4}$ associated to rotating the tangent ray at a specific marked point which is distinguished as output, and the right hand side expresses this class as a sum of the classes associated to rotating each input and those associated to breaking the domain into two components (each with three marked points) glued along the node, and rotating the tangent ray on one side of the node.

This geometric picture suggests that the correct chain-level structure that gives rise to the Batalin-Vilkovisky structure on symplectic cohomology is that of an algebra over the operad formed by the moduli spaces $f \Mbar^{\bR}_{0,k+1}$ (the case of $k=1$ is exceptional, and we set it to be equal to the circle $S^1 $). The operad structure arises from the concatenation of Riemann surfaces with marked points to nodal Riemann surfaces, which induces a map of chain complexes
\begin{equation} \label{eq:concatenation_moduli_all_marked_points}
 C_*( f \Mbar^{\bR}_{0,k_1+1}) \otimes \cdots \otimes  C_*( f \Mbar^{\bR}_{0,k_n+1})  \otimes   C_*( f \Mbar^{\bR}_{0,n+1})  \to  C_*( f \Mbar^{\bR}_{0,\sum k_i +1}),
 \end{equation}
 whereas the algebra structure on symplectic cochains is a collection of operations
\begin{equation} \label{eq:operad_action_framed_moduli}
 \underbrace{SC^*(M) \otimes \cdots \otimes SC^*(M)}_{k}  \otimes  C_*( f \Mbar^{\bR}_{0,k+1}) \to SC^*(M),  
\end{equation}
for some model $C_*( f \Mbar^{\bR}_{0,k+1}) $ of the singular chains on $f \Mbar^{\bR}_{0,k+1}$, and some chain complex $SC^*(M)$ whose homology is symplectic cohomology, satisfying the following properties (we suggest \cite{LeinsterBook} as a reference for operads and algebras over operads):
\begin{enumerate}
\item Invariance under the action of the symmetric group on $k$-letters, acting by permuting the first $k$ marked point of elements of $f \Mbar^{\bR}_{0,k+1} $, and the copies of $SC^*(M)  $ in the domain of \eqref{eq:operad_action_framed_moduli}.
\item Compatibility with the operations associated to concatenation of the moduli spaces of framed curves (discussed in Appendix \ref{sec:stab-ksv-moduli}), in the sense that the map
  \begin{equation}
    \begin{tikzcd}
SC^*(M)^{\otimes \sum_{i=1}^{n} k_i}  \otimes  C_*( f \Mbar^{\bR}_{0,k_1+1}) \otimes \cdots \otimes  C_*( f \Mbar^{\bR}_{0,k_n+1}) \otimes    C_*( f \Mbar^{\bR}_{0,n+1})  \ar[d] \\ 
SC^*(M)^{\otimes \sum_{i=1}^{n} k_i} \otimes C_*( f \Mbar^{\bR}_{0,\sum_{i=1}^{n} k_i +1})   
    \end{tikzcd}
  \end{equation}
  obtained either by first separately applying the maps in Equation \eqref{eq:operad_action_framed_moduli} for  $k=k_i$, then applying it for $k=n$, agrees with the map obtained by first applying Equation \eqref{eq:concatenation_moduli_all_marked_points}, followed by Equation \eqref{eq:operad_action_framed_moduli} for the sum $k = \sum k_i $. 
\end{enumerate}


The main result of the paper establishes the existence of such a structure, which we moreover show to satisfy the following fundamental properties:
\begin{itemize}
\item Independence of auxiliary choices.
\item Strict functoriality under restriction maps.
  \item Compatibility with quantitative structures.
\end{itemize}
We shall give a precise formulation of our main result in Theorem \ref{thm:main_thm} below, which we precede by a necessary overview of the notion of support for symplectic cohomology. Afterwards, we shall discuss our strategy for the proof of these results, as well as the relationship between the operad formed by the moduli spaces $f \Mbar^{\bR}_{0,k+1}$, and the framed $E_2$ operad mentioned in the title, which is more extensively studied in the literature.

\begin{rem} \label{sec:PROP-structure}
It will be apparent from our methods that one can construct a model for symplectic cochains, satisfying the above list of properties, and carrying an action of the operad formed by the chains of the union of the moduli spaces of framed Riemann surfaces $f \Mbar^{\bR}_{g,k+1} $ of arbitrary genus (more precisely, one has to shift the chains by a function of the genus to account for the degree of the corresponding operations). We leave the details of such an extension to the reader largely because it is orthogonal to the interesting operations in higher genus, which require one to consider gluing at multiple points (the operadic structure only allows operations with one output, which corresponds to gluing at one point). We expect that such an extension would require a more significant use of methods of homotopical algebra than the present paper.
\end{rem}

\begin{rem}
There is a natural analogy between the chain structures we are constructing, and those which appear in Lagrangian Floer theory, leading to the question of why one cannot construct the operations in Equation \eqref{eq:operad_action_framed_moduli} by a procedure which follows the existing steps in that context. To explain the essential difficulty, recall that, notwithstanding the technicalities in resolving questions of anomaly and obstruction which are required to define the Floer cohomology groups of a Lagrangian $L$, it is by now well-established  \cite{FukayaOhOhtaOno2009,Seidel2008a} that one obtains an $A_\infty$ structure on the Lagrangian Floer chain complex, which can be written as a consistent collection of operations
\begin{equation} \label{eq:operad_action_real_moduli}
 \Floer^*(L) \otimes \cdots \otimes \Floer^*(L) \otimes  C_*( \Rbar_{0,k+1}) \to \Floer^*(L),  
\end{equation}
where $\Rbar_{0,k+1} $ is the moduli space of stable discs with $k+1$ marked points.

The fundamental difference between Equations \eqref{eq:operad_action_framed_moduli} and \eqref{eq:operad_action_real_moduli} is that the moduli spaces $\Rbar_{0,k+1} $ have a particularly simple topology: they can be realised as polytopes (the \emph{Stasheff associahedra}), and the operadic structure maps which control the consistency of the operations are given by inclusions of products of these polytopes as boundary faces. This leads to the algebraic structure being controlled by relatively simple combinatorics: there is a collection of operation indexed by the natural numbers, one for each moduli space $\Rbar_{0,k+1}$, satisfying the familiar $A_\infty$ relation, which is the way that Equation \eqref{eq:operad_action_real_moduli} is implemented in the literature, with the model for $C_*( \Rbar_{0,k+1})$ given by the cellular chains of the polytope structure.

Because the geometry of the moduli spaces $f \Mbar^{\bR}_{0,k+1} $ is much more complicated such an approach is not adequate for Equation \eqref{eq:operad_action_real_moduli}. This is already apparent for the case of operations with only one input, which we recall are given by a copy of $S^1$, with composition map $S^1 \times S^1 \to S^1$ given by the usual multiplication. Evidently, this map is not given by a cellular inclusion. In fact, in the standard model for symplectic cochains, the chain level structure corresponding to the circle action is given \cite{Seidel2008b} by a sequence of operations $\Delta^i$, indexed by the natural numbers, satisfying $d \Delta^n = \sum \Delta^{i} \Delta^{n-i}$. 
\end{rem}

\subsection{Support conditions for symplectic cohomology}
\label{sec:relat-sympl-cohom}

A standard construction associates to each compact subset $K$ of a space $M$ the indicator (characteristic) function $H_K$ which is $0$ on $K$ and is infinite away from it. This construction is functorial with respect to inclusions in the sense that, whenever $K$ is a subset of $K'$, we have a pointwise inequality
\begin{equation}
  H_{K'} \leq H_{K}.
\end{equation}

When $M$ is a symplectic manifold (which we now assume to be closed for simplicity), symplectic cohomology with support $K$ can be thought of as a lift of the above construction to cohomology groups: since Floer cohomology is not defined for discontinuous functions (nor those which take infinite value), one considers instead a sequence of (non-degenerate) Hamiltonians $H^i$ converging to $H_K$, to which one associates the Floer cochain groups $\Floer^*(H^i,J^i)$ for an auxilliary sequence of almost complex complex structures $J^i$. It is crucial at this stage to be careful with the choice of coefficients: a modern interpretation of Floer's invariance result \cite{Floer1989b} is that the isomorphism type of the Floer cohomology groups does not depend on the choice of Hamiltonian when the coefficient ring is the Novikov field\footnote{Because we allow the ground ring $\Bbbk$ to be an arbitrary ring, rather than a field, using the term Novikov field in this context is an abuse of terminology. Similarly, the Novikov ring is not strictly speaking a valuation ring, even though we shall use the term. }, whose elements are series $\sum_{i=0}^{\infty} a_i T^{\lambda_i}$ with $a_i$ lying in a chosen ground ring $\Bbbk$, and $\lambda_i$ real numbers with the property that $\lim_i \lambda_i = +\infty$. In order to retain dynamical information about the functions $H^i$, one works instead with the smaller Novikov ring whose elements consists of series for which the exponents $\lambda_i$ are non-negative. The category of modules inherits a natural completion functor associated to the $T$-adic filtration, which is defined as the inverse limit of the tensor product with quotients of the Novikov ring by powers of $T$, which is essential to the following:
\begin{defin} \label{def:relative_SH}
  The \emph{symplectic cohomology} $SH^*_M(K)$ with support a compact subset $K$ of a closed symplectic manifold is the homology of the completion of the (homotopy) direct limit of the chain Floer complexes of a monotone sequence of Hamiltonians converging to the indicator functor $H_K$:
  \begin{equation}
  SC^*_M(K) \equiv \widehat{\hocolim_{i} \Floer^*(H^i,J^i)}.  
\end{equation}
\end{defin}
In the above definition, the fact that we assume the sequence $\{H^i\}$ of Hamiltonians to be monotone is essential in realizing the maps in the direct limit as maps of Floer complexes over the Novikov ring, which is what allows us to define the symplectic cochains supported on $K$ by completion. Geometrically, this is a consequence of the fact that monotone continuation maps always have non-negative (topological) energy, which is a property that fails for general continuation maps.

As we shall discuss in Section \ref{sec:strictly-funct-coch} below, the specific model for the homotopy colimit used in \cite{Varolgunes2018} is the mapping telescope, which is a complete chain complex receiving a map from each Floer group $\Floer^*(H^i,J^i)$ in the chosen sequence, together with a homotopy in each triangle
    \begin{equation} \label{eq:homotopy_commutative_small-model}
      \begin{tikzcd}
        \Floer^*(H^i,J^i)  \ar[r] \ar[dr] &    \Floer^*(H^{i+1},J^{i+1}) \ar[d] \\
        &  SC^*_M(K),
      \end{tikzcd}
    \end{equation}
    where the horizontal map is the continuation map.  In fact the underlying homotopy type of the mapping telescope can be characterised by a universal property associated to this data.    This more abstract point of view will be useful to understand our construction.

Symplectic cohomology supported on $K$ is independent of the choice of approximating Hamiltonians, recovers ordinary homology when $M=K$, and is functorial under inclusions in the sense that there is a restriction map
\begin{equation}
  SH^*_M(K') \to SH^*_M(K)  
\end{equation}
whenever $K$ is a subset of $K'$, which is strictly compatible for triple inclusions. In addition, it satisfies a remarkable Mayer-Vietoris property, for a class of coverings which include those that are arise from a covering of the base of an coisotropic fibration. This property is crucial for recent applications both to symplectic topology \cite{Dicksteinetal2021} and to mirror symmetry \cite{GromanVarolgunes2021}. The last reference includes an extension of the definition of symplectic cohomology with support given by compact subsets to the case where the ambient symplectic manifold $M$ is \emph{geometrically bounded,} incorporating all the classes of open symplectic manifolds for which Floer theory is expected to be defined.

\subsection{Statement of results}
\label{sec:stat-main-result}

We now state the main result of this paper, which is proved in Section \ref{sec:floer-algebra}: 
\begin{thm}
  \label{thm:main_thm}
  Associated to any compact subset $K$ of a closed symplectic manifold is a complete  torsion free chain complex $SC^*_{M,f\Mbar^{\bR}_0}(K)$ over the Novikov ring, called the \emph{operadic symplectic cochains with support $K$}, which is equipped with the following structures:
  \begin{enumerate}
  \item An action of the operad whose $k$th chain complex, for $2 \leq k$, is given by the (symmetric normalised) cubical chains on $f \Mbar^{\bR}_{0,k+1}$, and whose first chain complex is the cubical chains of the circle.
  \item A restriction map for each inclusion $K \subset K'$ of compact subsets
    \begin{equation}
          SC^*_{M,f\Mbar^{\bR}_0}(K') \to  SC^*_{M,f\Mbar^{\bR}_0}(K),
        \end{equation}
        which is compatible with the operadic action, so that the composition of the restriction maps associated to a pair of inclusions $K \subset K' \subset K''$ strictly agrees with the restriction map for $K \subset K''$.
         \item An action of the symplectomophism group of $M$, i.e. an isomorphism
    \begin{equation} \label{eq:action_symplectomorphism}
          SC^*_{M,f\Mbar^{\bR}_0}(K) \cong   SC^*_{M,f\Mbar^{\bR}_0}(\psi(K))  
        \end{equation}
        for each symplectomorphism $\psi$ of $M$, which is compatible with composition, preserves the operadic action, and commutes with restriction maps.
      \end{enumerate}
    \end{thm}
    \begin{rem}
      In Appendix \ref{sec:app-diss}, we extend the above result (as well as the other main results of the paper), to the setting of geometrically bounded manifolds, which includes in particular the class of Liouville manifolds on which most of the literature on symplectic cohomology is formulated. We opt to segregate the discussion of geometrically bounded manifolds because it adds an additional level of technical complexity to a paper that is already quite technical. 
    \end{rem}

\begin{rem}
Theorem \ref{thm:main_thm} may appear suspiciously strong to experts, who might expect to see the compatibility property in its statement to involve higher homotopies, rather than being strict as asserted. These homotopies arise from the need to interpolate between various choices of data for defining the Floer complexes, as well as operations on them. An essential point of our approach is that all possible choices of Floer-theoretic data are incorporated in the definition of the operadic symplectic cochains with support a compact set $K$; in this way, the restriction maps associated to inclusions are straightforward to define, because they correspond to the fact that the space of data that arise in the definition of the operadic symplectic cochains for a compact set in $M$ is a strict subset of the data that arise in the definition for a subset thereof. Similarly, the action of the symplectomorphism group arises directly from its action on the space of data defining the operadic symplectic cochains. We shall presently give a more technically precise description of these ideas in Sections \ref{sec:strictly-funct-coch} and \ref{sec:strat-proof-theor} below.
\end{rem}

In the statement of the theorem, we are implicitly working with a Novikov ring whose coefficients are a field of characteristic $0$, e.g. the rational numbers $\Bbbk=\mathbb{Q}$, and referring to $\bZ/2$-graded chain complexes. The proof of the above result thus relies on techniques of virtual counts, specifically on the results of \cite{Abouzaid2022}. The present paper separately contains a complete proof that is independent of any theory of virtual counts, for Floer complexes with coefficient ring the Novikov ring over the integers $\bZ$, under the assumption that the ambient symplectic manifold $M$ is exact, Calabi-Yau, or monotone. Restricting to the subcomplex formed by contractible orbits one can work more generally on symplectic manifolds that are spherically Calabi-Yau or monotone, i.e. for which the first Chern class and symplectic class, evaluated on $\pi_2(M)$, are non-negatively proportional to each. Under the standard extra assumptions our results give chain complexes with finer gradings as well.

\begin{rem} \label{rem:semi-positivity}
  In the literature on Hamiltonian Floer theory \cite[Definition 6.4.1]{McDuffSalamon2012}, the ad hoc notion of semi-positivity is introduced as a condition under which the Floer complex is defined integrally, for a generic choice of almost complex structure. However, in this context, the standard methods do not associate higher homotopies to (generic) families of paths of almost complex structures. One is thus led to choose a generic almost complex structure, and to change the definition of the chain complex $SC^*_{M,f\Mbar^{\bR}_0}(K)$ so that all pseudoholomorphic curves are defined with respect to this fixed almost complex structure, and the Hamiltonian data are chosen generically to achieve transversality. In this way, the first two parts of Theorem \ref{thm:main_thm} can be lifted to the integral Novikov ring in the semi-positive case. We do not know, however, how to prove the last part, regarding invariance under the action of the symplectomophism group, using such methods.

  We expect that Fukaya and Ono's proposal \cite{FukayaOno2001} for extracting integral counts from moduli spaces of pseudo-holomorphic curves whose locus of non-trivial isotropy virtually has both  strictly positive codimension and a stable complex normal bundle, will lead to a construction of the desired integral lift in general. A version of this proposal was realised by Bai and Xu \cite{BaiXu2022} and Rezchikov \cite{Rezchikov2022} in a context which is sufficient to conclude the well-definedness of the Floer complexes, but which is not currently sufficient to define operations on it.
\end{rem}

As stated, Theorem \ref{thm:main_thm} makes no reference to Floer theory. In Section \ref{sec:univ-prop-oper} below, we will explain the way in which all the structures are determined, in a universal way, from Floer-theoretic operations, but for concreteness, it is useful to separate the following result, regarding the chain complex underlying the operadic symplectic cochains:
    \begin{thm}\label{thm:linear_comparison_from_CF}
      For each Hamiltonian $H$ which is negative on $K$, there is a map
  \begin{equation} \label{eq:map_Ham_Floer_Cochains-SC}
      \Floer^*(H,J) \to   SC^*_{M,f\Mbar^{\bR}_0}(K),  
  \end{equation}
and for each monotone continuation map, there is a homotopy in the diagram
    \begin{equation} \label{eq:homotopy_commutative_continuation}
      \begin{tikzcd}
        \Floer^*(H_0,J_0)  \ar[r] \ar[dr] &    \Floer^*(H_1,J_1) \ar[d] \\
        & SC^*_{M,f\Mbar^{\bR}_0}(K).
      \end{tikzcd}
    \end{equation}
    \end{thm}

Note that the conclusions of Theorems \ref{thm:main_thm} and \ref{thm:linear_comparison_from_CF} are satisfied by many examples of rather trivial nature which do not record any deep information about symplectic topology. For example, by the constant assignment of a fixed algebra over the chains of the moduli spaces $f \Mbar^{\bR}_{0,k+1}$ to each compact subset of $K$, with a trivial morphism from each Hamiltonian Floer group. A more interesting example arises from the fact that one can equip the ordinary cochains of each topological space with the structure of an algebra over the $E_\infty$ operad, functorially with respect to all maps. Via the homotopically unique map from the operad $f \Mbar^{\bR}_{0,k+1}$ to the $E_\infty$ operad (induced by the fact that the latter is terminal), this gives rise to another example satisfying the conclusions of  Theorems \ref{thm:main_thm} and \ref{thm:linear_comparison_from_CF}. 

The following result, proved in Section \ref{sec:comp-two-models-1}, provides a comparison with the known constructions (and computations) of symplectic cohomology with support:

\begin{thm} \label{thm:comparison}
The map from Hamiltonian Floer cochains to the operadic symplectic cochains with support $K$ induces a quasi-isomorphism
    \begin{equation} \label{eq:quasi-isomorphism_big-small}
          SC^*_{M}(K) \cong   SC^*_{M,f\Mbar^{\bR}_0}(K)
        \end{equation}
        from the homotopy colimit model of the symplectic cochains with support $K$. This map is a homotopy equivalence whenever  the base ring $\Bbbk$ is a field, and it is compatible with restrictions maps, in the sense that the following diagram commutes up to prescribed homotopy:
        \begin{equation}
          \begin{tikzcd}
 SC^*_M(K')\ar[r] \ar[d]  & SC^*_M(K) \ar[d] \\
             SC^*_{M,f\Mbar^{\bR}_0}(K') \ar[r] &  SC^*_{M,f\Mbar^{\bR}_0}(K).
           \end{tikzcd}
        \end{equation}
      \end{thm}
      \begin{rem}
        We expect that the stronger conclusion that Equation \eqref{eq:quasi-isomorphism_big-small} is a homotopy equivalence holds in general, but the proof that we provide relies on abstract properties of the category of chain complexes over rings of finite global dimension, which do not seem to apply to the integral Novikov ring. 
      \end{rem}

\subsection{Framed $E_2$ structures}
\label{sec:fram-e_2-struct}

The space of operations $f \Mbar^{\bR}_{0,k+1}$ that arises in Theorem \ref{thm:main_thm} is known to be homotopy equivalent to the space of (disjoint) embeddings of $k$ discs in the unit disc $D^2$: the key point is that, for each genus $0$ curve with one marked point (the output) which is equipped with a choice of tangent ray, there is a contractible choice of identifications of the domain with $\bC \bP^1$, mapping the marked point to $\infty$, the tangent ray to the positive real axis, and all other marked points to the complement of the unit disc. In order to extract the desired homotopy equivalence from this construction, one uses the further contractible choice of a sufficiently small positive real number for each other marked point (input), which extends the remaining choices of tangent rays to disjoint embedded discs. 

The fact that this construction is compatible with the stable compactification, as well as with the operadic structure maps is encoded by the following result, due to Kimura, Stasheff, and Voronov \cite[Section 3.4 and 3.7]{KimuraStasheffVoronov1995}, who use the notation $\scrN$ for $f\Mbar^{\bR}_0$ and $\scrF$ for the framed $E_2$ operad: 
\begin{prop}\label{prop:framed_E_2_zig_zag}
  There is an operad $\scrP$, which is equipped with operad maps
  \begin{equation}
    E_2^{fr} \leftarrow    \scrP \to f\Mbar^{\bR}_0
  \end{equation}
  that are level-wise homotopy equivalences. \qed
\end{prop}

This result immediately implies that the homology of the framed $E_2$ operad acts on symplectic cohomology with support any compact set, and that this action is compatible both with restriction maps and with the action of symplectomorphism groups. The homology of this operad was shown by Getzler \cite{Getzler1994} to be generated by two operations, an associative and commutative product and a degree $1$-operator squaring to $0$, subject only to the relation encoded by Equation \eqref{eq:BV-relation}. We conclude:
\begin{cor}
  The symplectic cohomology group $ SH^*_M(K)$ is equipped a natural $BV$-algebra structure, which is preserved by restriction maps, and by the isomorphism $ SH^*_M(K) \cong SH^*_{M}(\psi K)$ associated to a symplectomorphism $\psi$ of $M$. \qed
\end{cor}
Restricting to the product, this construction recovers the product  constructed in \cite{TonkonogVarolgunes2020}.

In order to formulate an explicit chain level structure on the telescope model of the symplectic cochains, we restrict to characteristic $0$, in which case there is a replacement of the framed $E_2$ operad, called the $BV_\infty$ operad, consisting of explicit operations \cite[Theorem 20]{Galvez-Carrilloetal2012} for which a homotopy transfer result is known \cite[Theorem 33]{Galvez-Carrilloetal2012} (the cited result is formulated for a characteristic $0$ field but an inspection of the proof shows that it suffices to work over a commutative $\bQ$-algebra). This allows us to conclude:

\begin{cor}
 Assuming that the ground ring $\Bbbk$ contains the rational numbers, the telescope model $SC^*_M(K)$ of symplectic cohomology with support $K$ can be equipped with the structure of a $BV_\infty$ algebra, lifting the $BV$ structure on homology. \qed
\end{cor}

The list of operations on $BV_\infty$ algebras include those of an $L_\infty$ algebra, and we expect that these operations will recover those introduced by Siegel \cite{Siegel2019},  without referring specifically to symplectic cohomology with support, and which he showed define new symplectic capacities.

\begin{rem}
The explicit description of $BV_\infty$ as a cofibrant replacement of the   framed $E_2$ operad in characteristic $0$ strongly relies on the formality of the latter, i.e. on the existence of a homotopy equivalence between the rational homology of the   framed $E_2$ operad (which is the $BV$ operad by Getzler's result) and its rational chains. This was established in \cite{GiansiracusaSalvatore2010,Severa2010}, building on Tamarkin's result establishing the formality of the ordinary $E_2$ operad \cite{Tamarkin2003}. Such a result is known to fail integrally as proved by Salvatore for the non-symmetric part of the operations in characteristic $2$ \cite{Salvatory2019}, and by Cirici and Horel for the symmetric part in general \cite[Remark 6.9]{CiriciHorel2018}.
\end{rem}

\subsection{A strictly functorial cochain model for $SH^*_M(K)$ }
\label{sec:strictly-funct-coch}

As a final preparatory step to explaining the proof of our main results, we consider a toy problem, namely the construction of a model $SC^*_{M, \star}(K) $ for the symplectic cochains with support $K$, which is strictly functorial under inclusions, i.e. so that the restriction maps for a triple of inclusions $K \subset K' \subset K''$ give rise to a commutative diagram
\begin{equation} \label{eq:commutative_triangule_middle_model}
  \begin{tikzcd}
    SC^*_{M, \star}(K'')  \ar[r] \ar[dr] &  SC^*_{M, \star}(K')  \ar[d] \\
    & SC^*_{M, \star}(K). 
  \end{tikzcd}
\end{equation}

We start by noting that the choices made in the definition of the existing model for $ SH^*_{M}(K)$ are (i) a sequence $H^i$ of Hamiltonians converging to the indicator function of $K$, (ii) almost complex structures $J^i$ used to define the Floer complexes $\Floer^*(H^i,J^i)$, and (iii) continuation equations defining chain maps
\begin{equation}
 \Floer^*(H^i,J^i) \to    \Floer^*(H^{i+1},J^{i+1}).
\end{equation}
From these data, the complex $ SC^*_{M}(K)$ is then defined in \cite{Varolgunes2018} as the completion of the (total complex) given by the mapping telescope
\begin{equation}
  \begin{tikzcd}
    \Floer^*(H^0,J^0) & \Floer^*(H^1,J^1) & \cdots \\
    \Floer^*(H^{0},J^{0}) \ar[u,"="] \ar[ur] &  \Floer^*(H^{1},J^{1}) \ar[u,"="] \ar[ur] & \cdots
  \end{tikzcd}
\end{equation}
following the method originating in \cite{AbouzaidSeidel2010}.

The advantage of the mapping telescope definition is that it is essentially as small as possible, so that the construction of explicit algebraic operations on it requires the least effort. But the construction of restriction maps require making an interpolation between the choices made for $K$ and $K'$, and there is no reason to expect that the interpolations for a triple of inclusions agree.

We resolve this as follows: first, we interpret the construction of the sequence of Floer complexes $ \Floer^*(H^i,J^i) $ and of the continuation maps between them as a functor
\begin{equation} \label{eq:Floer_functor_sequence}
CF \co  \bN \to \Ch  
\end{equation}
from the natural numbers (thought of as a category with an arrow from $i$ to $j$ if and only if $i \leq j$) to the category of chain compexes over the Novikov ring. Next, we sacrifice the small size of the mapping telescope for the larger model of the homotopy colimit given by the bar construction:
\begin{equation}
    \hocolim \Floer^*(H^i, J^i) \cong B(\bN, CF).
  \end{equation}
  Finally, we replace the domain category $\bN$ by the category $\cF_{K,\star}(1)$ of all pairs $(H,J)$ for which Floer theory is defined, so that the Hamiltonian $H$ is negative on $K$, and morphisms given by monotone continuation maps between them. It is important in this last stage to remember that there is a natural notion of a family of continuation maps, and we thus have to consider $\cF_{K,\star}(1)$ as an enriched category, which turns out to have the particularly simple feature that the space of morphisms between objects are either empty or contractible. We make the technical choice of considering only families of continuation maps parametrised by cubes, thus modelling homotopy theory using (symmetric) cubical sets as discussed in Appendix \ref{sec:cubical-sets}. There are alternative methods, such as grappling with the definition of a topology on the (infinite) dimensional space of continuation maps (allowing for breaking), or using simplicial sets, or even going all the way to formulate our construction using quasi-categories.

  The essential point at this stage is that the functor of Floer cochains extends to an enriched functor
  \begin{equation}
    \Floer^* \co  C_* \cF_{K,\star} \to  \Ch.
  \end{equation}
  This means that we assign to each pair $(H,J)$ its associated Floer complex, and to each $n$-cube of continuation maps a degree $-n$ map of Floer cochains, which is compatible with restriction to boundary strata (and vanishes for degenerate cubes). The homotopy colimit of the corresponding functor is the dashed arrow in the diagram of differential graded categories
    \begin{equation} \label{eq:left-Kan-linear-part}
    \begin{tikzcd}
      C_* \cF_{K,\star} \ar[r,"\Floer^*"] \ar[d, "\pi"]  & \Ch. \\
      \star \ar[ur,dashed]
    \end{tikzcd}
  \end{equation}
  where the vertical map is the projection map to a point, and we have a prescribed (homotopy) natural transformation between $\Floer^*$ and the composition of this arrow with the projection. This natural transformation is the structure map appearing in the formulation of the universal property of the (homotopy) colimit, namely the existence of a map
  \begin{equation}
    \Floer^*(H,J) \to \hocolim_{C_* \cF_{K,\star}}    \Floer^*
  \end{equation}
for each object $(H,J)$ of $\cF_{K,\star}$ that commutes up to prescribed homotopy with the action of morphisms in the category $ C_* \cF_{K,\star} $.

We find it convenient to use a specific model of the homotopy colimit given by the bar construction (see \cite{HollenderVogt1997}), and we denote its completion by:
  \begin{equation}
     SC^*_{M, \star}(K) \equiv \widehat{B( \Floer^*,  \cF_{K,\star}(1)) }  .
  \end{equation}

  With this definition at hand, establishing the existence of a commutative Diagram \eqref{eq:commutative_triangule_middle_model} is straightforward: an inclusion $K \subset K'$ induces an inclusion of categories
  \begin{equation}
        \cF_{K', \star}(1)  \to \cF_{K,\star}(1) ,
      \end{equation}
      because the only condition, that is not a global condition independent of $K$, which is imposed on the objects and morphisms of $\cF_{K,\star}(1)$ is the requirement that the Hamiltonian $H$ be non-negative on $K$, and this condition is inherited under inclusions. Given a nested triple $K \subset K' \subset K''$, the construction yields a nested inclusion of categories $\cF_{K'',\star}(1) \subset \cF_{K',\star}(1) \subset \cF_{K,\star}(1)$, so that the functoriality of the bar construction yields the desired commutative triangle.

      The last thing to check is that this construction defines a complex which is homotopy equivalent to the usual symplectic cochains with support $K$. The essential point in this case is that the elements of the sequence $\{H^i\}_{i =0}^{\infty}$ eventually dominate any Hamiltonian $H$ on $M$ which is strictly negative on $K$, hence that the space of morphisms from any object of $\cF_{K,\star}(1)$ to the chosen sequence eventually become contractible (this is Proposition \ref{lmFrgtHtpyEq}). A standard comparison result for homotopy colimits (c.f. Section \ref{sec-hv}) then implies:
      \begin{prop} \label{prop:compare_old_new_hocolim_models}
        The map
      \begin{equation}
    SC^*_{M}(K) \to  
   SC^*_{M, \star}(K)
              \end{equation}
              induced by the inclusion of the sequence $ \{(H^i, J^i)\}_{i =0}^{\infty}$ as a subcategory of $  \cF_{K,\star}(1) $ is a quasi-isomorphism. Assuming that $\Bbbk$ is a field, it is also a homotopy equivalence. \qed
      \end{prop}

              \subsection{Strategy for the proof of Theorem \ref{thm:main_thm}}
\label{sec:strat-proof-theor}

We now explain how to adapt the strategy outlined in the previous section in order to obtain a model which is equipped with the structure of an algebra over the operad of chains on $f \Mbar^{\bR}_{0,k+1}$. As indicated earlier, this means that our definition should include all possible choices not just for defining the Floer complexes, but also for constructing multiplicative operations on it.

A convenient formalism for recording all this data is that of a multicategory, which simultaneously extends the notions of category and operad: as in a category one is given a collection of objects, but in addition to morphisms between such objects, multimorphisms which have multiple inputs and a single output are included. Such multimorphisms correspond to the higher spaces of an operad, and indeed an operad is exactly a multicategory with one object.

Multicategories arise naturally in Floer theory because the product map on Floer cochains
\begin{equation} \label{eq:Floer-product}
    \Floer^*(H^1, J^1) \otimes     \Floer^*(H^2, J^2) \to     \Floer^*(H^0, J^0) 
  \end{equation}
  is defined by considering equations of Floer type on a pair of pants with $3$ marked points, whose restriction to neighbourhoods of the three marked points respectively agree, under a local biholomorphism of a half-cylinder with the pair of pants, with the Floer equations associated to the pair $(H^i, J^i) $. There is no canonical choice for such a map, and our goal therefore is to identify a space  $\cF((H^1, J^1),(H^2, J^2);(H^0, J^0) )  $ (associated to $K=\emptyset$, which is the universal case for our construction) of Floer data with two inputs and one output, so that we have a map
  \begin{multline}
    \Floer^*(H^1, J^1) \otimes     \Floer^*(H^2, J^2)   \otimes  C_*   \cF((H^1, J^1),(H^2, J^2);(H^0, J^0) )  \\ \to     \Floer^*(H^0, J^0),
  \end{multline}
  as well as composition maps associated to changing the inputs and outputs. More generally, we need to define spaces $\cF((H^1, J^1), \cdots, (H^k, J^k);(H^0, J^0) )$ for each input sequence of objects of $\cF$, and (multi)-composition maps realising the structure of a multicategory, which we formally recall in Appendix \ref{sec:multicategories}.

  The essential difficulty in this task is to ensure that the notion of Floer data that we use  (i) has  a well-defined notion of composition, (ii) induces maps of Floer complexes over the Novikov ring, and (iii) satisfies the property that, for each input sequence $\left( (H^1, J^1), \cdots, (H^k, J^k) \right) $ and output data $(H^0, J^0) $, the projection map from the space of multimorphisms to the moduli space $f\Mbar^{\bR}_{0,k+1}$  is a homotopy equivalence whenever the function $H^0$ is sufficiently close to the indicator function of $K$.
  \begin{rem}
    The importance of the third condition above may not be immediately apparent to the reader, but it is essential in proving that our construction yields a model for the symplectic cochains. The following analogy may be helpful: consider fibre bundles $X_1 \to B$ and $X_2 \to B$ with fibres $F_1$ and $F_2$, and assume that we have a map $X_1 \to X_2$ which also a fibre bundle, through which the first fibre bundle projection factors. From these data, we obtain a map on $B$ between the local systems with fibres $F_1$ and $F_2$, and we would like a natural condition that implies that this map is an isomorphism of local systems. Such a condition is provided by the assumption that the map $X_1 \to X_2$ induces a homotopy equivalence $F_1 \to F_2$.

  In our setting, the data of $X_1 \to B$ and $X_2 \to B$ are ultimately used to construct the model of symplectic cochains defined in this paper, and the standard one. The space $X_1$ consists of the data of multimorphisms, and the map to $X_2$ only remembers their domain and their output. By requiring that the forgetful map be a homotopy equivalence, we shall be able to conclude that the map between the two models is an isomorphism.
  \end{rem}
  The most general way to achieving (i) and (ii) is to define Hamiltonian data on a framed Riemann surface 
  to consist of a $1$-form $\mathfrak{H}$ valued in the space of Hamiltonians on the target symplectic manifold, with prescribed restriction to a neighbourhood of the punctures. In this context, the positive energy condition that is required in order for operations to be defined over the Novikov ring is the non-linear equation 
\begin{equation}\label{eqMonotonicityPoisson}
d\mathfrak{H}+\{\mathfrak{H},\mathfrak{H}\}\geq 0.
\end{equation}
Unfortunately, we have been unable to prove that this choice satisfies the third condition above.
  
Instead, we  introduce a notion of \emph{split-monotone Floer data} on a framed Riemann surface $\Sigma$; this consists of a closed $1$-form on $\Sigma$, and a function $H$ on $\Sigma \times M$, subject to several conditions, of which the most important is the requirement that, for each point $x \in M$, the wedge product of the differential of $H(x)$ with $\alpha$ is a non-negative $2$-form on $\Sigma$
  \begin{equation} \label{eq:monotonicity_H_alpha}
    dH(x) \wedge \alpha \geq 0;
  \end{equation}
this condition is the specialisation of Inequality \eqref{eqMonotonicityPoisson} to the case $\mathfrak{H} = H \otimes \alpha $.  When $\Sigma$ is a cylinder $\bR \times S^1$ with coordinates $(s,t)$, and $\alpha = dt$, it is clear that we recover the condition of monotonicity $\frac{\partial H}{\partial s} \geq 0$ for continuation maps, which underlies Definition \ref{def:relative_SH}. 

An additional feature worth mentioning is that our $1$-forms $\alpha$ are required to be of the form $\alpha=w_pdt$ near each puncture $p$ for some positive real weight $w_p$  (subject to the constraint in Equation \eqref{eq:weight_bound}). The need for real weights is clarified in Figure \ref{fig:broken_continuation}. Namely, for contractibility to hold we need to interpolate between different assignments of weights for given asymptotic data (there are other approaches where contractibility is achieved entirely using continuation maps, c.f. \cite{AbouzaidSeidel2010}).

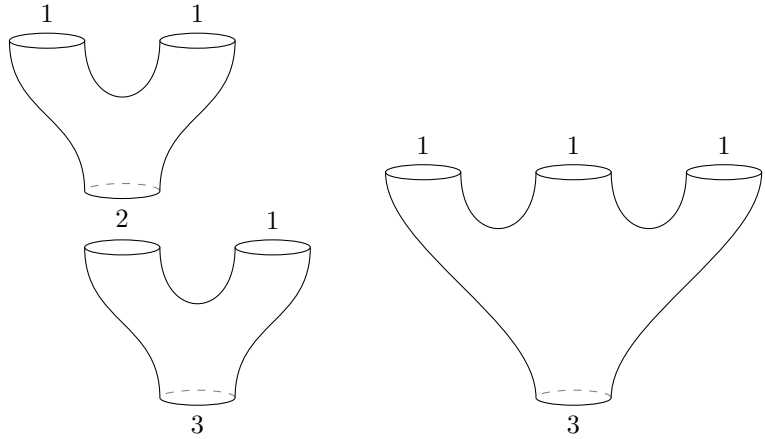
\begin{figure}[h]\label{fig:broken_continuation}.
  \centering
  \begin{tikzpicture}
    \node[label=above:{$1$}] at (1,2) {};
    \node[label=above:{$1$}] at (-1,2) {};
    \node[label=below:{$2$}] at (0,0) {}; 
    \pairofpants
    \begin{scope}[shift={(1,-2.75)}]
      \pairofpants
      \node[label=below:{$3$}] at (0,0) {};
       \node[label=above:{$1$}] at (1,2) {};
     \end{scope}

     \begin{scope}[shift={(6,-2.75)}]
        \threeinputs
        \node[label=below:{$3$}] at (0,0) {};
        \node[label=above:{$1$}] at (0,3) {};
        \node[label=above:{$1$}] at (-2,3) {};
        \node[label=above:{$1$}] at (2,3) {};
     \end{scope}
  \end{tikzpicture}
  \caption{Breaking of curves with integral weights}
  \label{fig:integral_weights}
\end{figure}

  A split Floer datum on $\Sigma$ defines an element of the space of multimorphisms with inputs $\left( (H^1, J^1), \cdots, (H^k, J^k) \right) $ and output $(H^0, J^0) $, when there is a choice of cylindrical ends at the $i$th puncture so that $ H \otimes \alpha $ and $J$ pull back to $H_i \otimes dt$ and to $J_i$. The topology on this space is straightforward to derive from the topology on the moduli space of framed Riemann surfaces and the $C^\infty$ topologies on function spaces.

  Unfortunately, this construction is not closed under compositions. For example, we would need to have a (continuous) composition map
  \begin{multline}
    \cF((H^1, J^1),(H^2, J^2);(H^0, J^0)) \otimes \cF((H^{0}, J^{0}) ;({H'}^{0}, {J'}^{0})) \\
    \to      \cF((H^1, J^1),(H^2, J^2);({H'}^{0}, {J'}^{0})),
  \end{multline}
which is compatible with the actions on Floer cochains, and more generally we need composition maps of multimorphisms with other multimorphisms. If we restrict attention to Floer data on (smooth) Riemann surfaces, this is impossible, because the composition of these operations is geometrically associated to the pre-stable Riemann surface obtained by attaching a cylinder to a pair of pants at one end. The definition of the multicategory $\cF$ thus involves split Floer data on pre-stable Riemann surfaces. This makes a straightfoward definition of a topology more tricky, and we choose to bypass this, as discussed earlier, by working with cubical sets as our model for homotopy types.

Continuing along the outline of the simpler problem discussed in Section \ref{sec:strictly-funct-coch}, we have by now explained the construction of an algebraic object that encodes all possible choices for constructing genus $0$ operations, with one output, in Hamiltonian Floer theory. The next step is to apply the usual Floer theoretic procedure to associate to each $n$-cube in $\cF((H^1, J^1),(H^2, J^2);(H^0, J^0) ) $ an operation of degree $n$ in Equation \eqref{eq:Floer-product}, and more generally a chain map
  \begin{multline}
   \Floer^*(H^1, J^1) \otimes  \cdots \otimes   \Floer^*(H^k, J^k) \\     \otimes C_*  \cF((H^1, J^1), \ldots, (H^k, J^k);(H^0, J^0) ) \to     \Floer^*(H^0, J^0).
  \end{multline}
  The compatibility of these maps with (multi)-compositions in $\cF$ amount to asserting that they assemble to a differential graded multi-functor
  \begin{equation} \label{eq:Floer-multifunctor}
    C_* \cF \to \Ch,
  \end{equation}
  whose target is the category of chain complexes equipped with its monoidal structure given by the tensor product, which we consider as multicategory whose objects are $\bZ_2$-graded chain complexes and whose multimorphisms with source a sequence $(C_1, \ldots, C_k)$ of chain complexes and target a chain complex $C_0$, given by the chain complex of maps
  \begin{equation}
    C_1 \otimes \cdots \otimes C_k \to C_0.    
  \end{equation}
At this stage, we recall that each compact subset $K$ of $M$ determines the subset of objects of $\cF$ consisting of those pairs $(H,J)$ for which the Hamiltonian $H$ is strictly negative on $K$. We write $\cF_K \subset \cF$ for the full multicategory on these objects, and abuse notation by writing $\Floer^*$ for the restriction of the Floer functor to this multicategory.

  It now remains to extract, from Equation \eqref{eq:Floer-multifunctor}, and the projection map from $\cF_K$ to $f\Mbar^{\bR}_0$, a total complex which carries an action of the chain operad associated to $f\Mbar^{\bR}_0 $. In order to do this, we need one final piece of abstraction. In the diagram of multicategories
  \begin{equation}
    \begin{tikzcd}
      C_* \cF_K \ar[r,"\Floer^*"] \ar[d, "\pi"]  & \Ch, \\
      C_* f\Mbar^{\bR}_0 \ar[ur,dashed]
    \end{tikzcd}
  \end{equation}
  we shall consider a diagonal arrow that we refer to as the \emph{operadic (homotopy) left Kan extension of $\Floer^*$ along $\pi$.} This is a natural generalisation, in two different ways, of the notion of a homotopy colimit which arose in Diagram \eqref{eq:left-Kan-linear-part}: (i) we pass from categories to multicategories, and (ii) we work over the operad $C_* f\Mbar^{\bR}_0$ rather than over the point. We note at this stage that, from the point of view of operads as multicategories with one object, the diagonal arrow distinguishes a chain complex (the image of this unique object), together with an action of the operad of chains on $ f\Mbar^{\bR}_0$.

  Instead of characterising the Kan extension by its universal property discussed in Section \ref{sec:univ-prop-oper} below, we choose to work with a specific model, analogous to the bar construction, which we describe explicitly in Section \ref{sec:floer-algebra} and denote by $L\pi_* \Floer^*$  and whose completion defines the operadic symplectic cochains with support $K$ that we refer to in Theorem \ref{thm:main_thm}.

The structural results in Theorem \ref{thm:main_thm}, regarding restriction and the action of the symplectomorphism group follow rather directly from the good functorial properties of the model which we have chosen for the left Kan extension. On the other hand, our construction of the comparison map in Equation \eqref{eq:quasi-isomorphism_big-small} requires some explicit computations, because it is a comparison between a categorical and an operadic Kan extension.

\subsection{The universal property of the operadic model}
\label{sec:univ-prop-oper}

While we use an explicit construction of the Floer algebra $SC^*_{M,f\Mbar^{\bR}_0}(K)$ in this paper, we complete this introduction by briefly indicating the universal property that it enjoys. We have already used the fact that an operad such as $ C_* f\Mbar^{\bR}_0$ is the same thing as a multicategory with a single object while an $ C_* f\Mbar^{\bR}_0$-algebra is the same thing as a multifunctor from $ C_* f\Mbar^{\bR}_0$ to $ \Ch$. Thus we take $SC^*_{M,f\Mbar^{\bR}_0}(K)$ to denote not merely a chain complex with operations but a multifunctor. We then denote by   $SC^*_{M,f\Mbar^{\bR}_0}(K)\circ{\fF}$ the composition with the forgetful map.  

The essential point at this stage is that there is a homotopy natural transformation of multifunctors
\begin{equation}\label{eq-hom-nat-trans}
  \alpha: CF^* \Rightarrow SC^*_{M,f\Mbar^{\bR}_0}(K)\circ{\fF} \end{equation}
whose component at the object $(H,J)\in\cF$ is the inclusion map $\Floer^*(H,J)\hookrightarrow SC^*_{M,f\Mbar^{\bR}_0}(K)$. It would take us too far afield to formally define such a notion, but one way to formulate it as the data of a collection of homotopies for each operadic structure map, together with higher homotopies associated to compositions thereof. Alternatively, the reader can find a discussion in the quasi-categorical setting in  \cite[Lemma 2.16]{AyalaFrancisTanaka2017},

The homotopy universal property is a follows: \emph{given any $f\Mbar^{\bR}_0$ algebra $D$ together with a homotopy natural transformation $\alpha': \Floer\Rightarrow D\circ{\fF}$ there exists, uniquely up to contractible choice, a homotopy natural trasnformation $\beta:SC^*_{M,f\Mbar^{\bR}_0}(K)\to D$ such that $\alpha'=\beta\circ\alpha$. }

Theorem \ref{thm:linear_comparison_from_CF} can be interpreted as the first non-trivial datum extracted from the natural transformation in Equation \eqref{eq-hom-nat-trans}, at the level of morphisms. In order to formulate the analogous data for multimorphisms, fix a sequence $\{H_i\}_{i=0}^{k}$ of non-degenerate Hamiltonians which are strictly negative on $K$, and a sequence $\{J_0\}$ of $S^1$-families of almost complex structures.  Consider a family of framed genus-$0$ Riemann surfaces parametrised by an manifold (more generally, a cycle in $f\Mbar^{\bR}_{0,k+1}$), equipped with split Floer data as described above. 
Such data determine a map
\begin{equation}
 \Floer^*(H_1, J_1) \otimes \cdots \otimes \Floer^*(H_k, J_k)  \to   \Floer^*(H_k, J_k)
\end{equation}
obtained from the moduli spaces of virtual dimension $0$ solutions of the family of the associated Cauchy-Riemann equation. We omit the proof of the following result.

\begin{prop}

  There is a prescribed homotopy in the diagram
  \begin{equation}
    \begin{tikzcd}
  \Floer^*(H_1, J_1) \otimes \cdots \otimes \Floer^*(H_k, J_k)   \ar[r] \ar[d] &  \Floer^*(H_k, J_k) \ar[d] \\
  SC^*_{M,f\Mbar^{\bR}_0}(K)  \otimes \cdots \otimes SC^*_{M,f\Mbar^{\bR}_0}(K) \ar[r] & SC^*_{M,f\Mbar^{\bR}_0}(K),   
    \end{tikzcd}
   \end{equation}
in which the bottom horizontal map is defined applying the operadic structure map to the cycle in $f\Mbar^{\bR}_{0,k+1} $ associated to the chosen family of Floer data.
\end{prop}

We propose to the reader that the above result should be interpreted as follows: given a sequence of elements of $SH^*_M(K) $ arising as the image of cycles in the Floer cohomology of Hamiltonians $H_i$, the image of these elements under any operation parametrised by a cycle in the operad $f\Mbar^{\bR}_{0}$ can be computed by first lifting this cycle to the multicategory $\cF_K$, then applying the associated Hamiltonian Floer-theoretic operation, and finally composing with the map from Floer cohomology to symplectic cohomology with support $K$. The main subtlety with this viewpoint is that, because of the chain-level completion, an arbitrary element of the symplectic cohomology group with support $K$ may not arise as the image of an element of Floer cohomology for any choice of Hamiltonian. This is one reason for formulating the above result at the chain-level.

\subsection{Outline}
\label{sec:outline}
In Section \ref{sec:hamilt-index-categ}, we construct the canonical \emph{Hamiltonian indexing multicategory} $\mathcal{H}$ of a symplectic manifold. The objects of $\mathcal{H}$ are all non-degenerate Hamiltonians and for a compact $K\subset M$ we  consider the full sub-multicategory $\mathcal{H}_K$ with the additional condition on objects that they be negative on $K$. The multimorphisms are very roughly given by families of Hamiltonians parametrised by a genus-$0$ Riemann surface, satisfying a \emph{monotonicity condition}, i.e. that they are non-decreasing along some oriented singular one dimensional foliation on the surface (which is not a priori fixed and is part of the data), see Definition \ref{def:pre-multimorphism-H}. In particular, recording the biholomorphism type of the surface defines a forgetful multifunctor \begin{equation} \mathcal{H}\to f\Mbar^{\bR}_0(n). \end{equation} In order to be able to encode the homotopy type of multimorphism sets, we construct $\mathcal{H}$ as a multicategory enhanced over symmetric cubical sets, see Definition \ref{dfCuMultMor}. The main result in this section (referred to as \emph{contractibility}) is Proposition \ref{lmFrgtHtpyEq} which says that if $H_0$ is sufficiently larger than $H_1,\ldots, H_n$, the symmetric cubical set $\mathcal{H}(H_1,\ldots, H_n; H_0)$ is homotopy equivalent to $f\Mbar^{\bR}_0(n)$; without such an assumption, we have little control over the homotopy type of the multimorphism spaces, which could for example be empty.

The multicategory $\mathcal{H}$ (or $\mathcal{H}_K$) is used as an intermediate step to introduce, in Section \ref{sec:floer-functor}, the \emph{Floer data indexing multicategory} $\mathcal{F},$ whose objects are pairs consisting of a Hamiltonian and an almost complex structure, and whose multimorphisms also include the data of families of almost complex structures. There is a canonical forgetful multifunctor $\mathcal{F}\to \mathcal{H}.$ Since the monotonicity condition does not involve the almost complex structures, the construction of $\mathcal{F}$ does not present any new significant difficulties, neither does the analogous contractibility result. We appeal to the results of \cite{Abouzaid2022} to obtain a dg-functor called the \emph{Floer functor}: \begin{equation*}
    \Floer^* \co  C_* \cF \to  \Ch,
  \end{equation*} which extends the assignment of the Hamiltonian Floer complex to each object $(H,J)$ of $\cF$, by assigning a degree $-n$ map to each $n$ dimensional cubical chain in a multimorphism space. While the construction the Floer functor in full generality uses virtual techniques, we also explain how it is constructed in the non-negatively monotone case using genericity arguments. The Floer functor might be of some independent interest as it is in some sense the universal diagram of Hamiltonian Floer complexes defined using split-monotone Floer data.

In Section \ref{sec:floer-algebra}, we construct a canonical dashed arrow   \begin{equation} 
    \begin{tikzcd}
      C_* \cF_{K} \ar[r,"\Floer^*"] \ar[d, "\pi"]  & \Ch, \\
      C_* f\Mbar^{\bR}_0 \ar[ur,dashed]
    \end{tikzcd}
  \end{equation} using an \emph{operadic bar construction}. This gives the proof of Theorem \ref{thm:main_thm}. The diagram does not commute, but there is also a canonical homotopy natural transformation from the Floer functor to the composition of the other two maps. We explicitly describe only a very small part of this structure which constitutes our proof of Theorem \ref{thm:linear_comparison_from_CF}.
  
  Section \ref{sec:comp-two-models-1} is devoted to a proof of Theorem \ref{thm:comparison}. We first prove that our operadic bar construction is chain homotopy equivalent to a categorical bar construction. To pass to categories we use the \emph{PROP functor} from multicategories to categories. Up to a difference in how the symmetric group actions enter in the construction, this passage is almost formal. On the other hand, it does take a considerable amount of work for us to settle this difference. This part concludes with Section \ref{sec:proof-that-map}.  Then we show that we can pass to much smaller models (such as the telescope model) using standard cofinality arguments and contractibility.\\

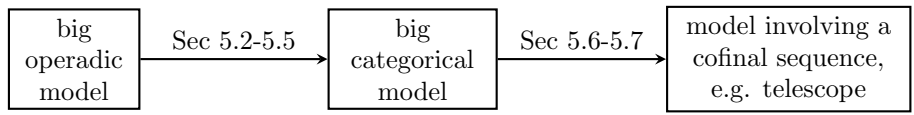
\begin{tikzpicture}[>=stealth, thick]

\node (A) at (0,0) [draw, process, text width=1.5cm, align=flush center] 
{big operadic model};

\node (B) at (4.5,0) [draw, text width=2cm, process, minimum height=0.5cm, align=flush center] 
{big categorical model};

\node (C) at (9.5,0) [draw, process,text width=3cm, minimum height=0.5cm, align=flush center] 
{model involving a cofinal sequence, e.g. telescope};

\draw[->] (A) to [above,"Sec 5.2-5.5"] (B);
\draw[->] (B) to [above,"Sec 5.6-5.7"] (C);

\end{tikzpicture}

\subsection*{Acknowledgements} The first author would like to thank Andrew Blumberg for helpful discussions about several aspects of this project. We would also like to thank Benoit Fresse, Alexander Polishchuk, Leonid Positelski, Tomer Schlank, and Bruno Valette for helpful comments and pointers to the literature about various questions on commutative algebra and on the homotopy theory of operads. Finally, we thank the referees for their extensive comments. They have helped to improve the paper in many ways. 

M.A. was supported by an NSF Standard Grant (DMS-2103805), the Simons Collaboration on Homological Mirror Symmetry, a Simons Fellowship award, and the Poincaré visiting professorship at Stanford University.
Y.G. was supported by the ISF (grant no. 2445/20) and the BSF (grant no. 2020310).
U.V. was supported by the T\"{U}B\.{I}TAK 2236 (CoCirc2) programme with a grant numbered 121C034.

\section{The Hamiltonian Indexing Multicategory}
\label{sec:hamilt-index-categ}

We fix a closed symplectic manifold $M$. The purpose of this section is to define a multicategory $\cH$, whose objects are Hamiltonians $H \in C^\infty(\bR/\bZ\times M)$ with non-degenerate $1$-periodic orbits, and whose multimorphisms are spaces (more precisely, symmetric cubical sets as described in Appendix \ref{sec:cubical-sets}) of Riemann surfaces with additional data which we shall presently specify. This multicategory will carry a forgetful map to the singular cubical chains (see Example \ref{example-symmetric-singular}) of  the framed KSV operad $f\Mbar^{\bR}_0$. The reader who finds this section overwhelming on first reading may benefit from first consulting Section \ref{sec:strictly-funct-coch} which discusses a simpler setting, as well as Section \ref{sec:strat-proof-theor}, from which the role of the constructions of this section in the overall strategy should become clearer.

\subsection{Multimorphisms of dimension $0$}

The essential point of our approach is the construction of a set of multimorphisms associated to each sequence $(H^1,\dots,H^n) $ of Hamiltonian inputs, and each choice $H^0$ of output. This will correspond to the $0$-cubes of the symmetric cubical set of multimorphisms associated to these data. All higher cubes will later turn out to be expressible in term of maps from the cube to this set.

Since the data that determine a multimorphism will involve a choice of Riemann surface, and we would like to identify the data supported on biholomorphic Riemann surfaces, it is convenient, as is familiar from many similar Floer-theoretic constructions, to define the desired space as a space of equivalence classes of a larger set which we now introduce (our conventions for trees and Riemann surfaces are prescribed in Appendix \ref{sec:trees-riem-surf}):

\begin{defin} \label{def:pre-multimorphism-H}
A \emph{pre-multimorphism}  with input a sequence of Hamiltonians $(H^1,\dots,H^n)$ and output a Hamiltonian $H^0$ consists of the following data:
\begin{enumerate}
\item A pre-stable rational curve
    $\Sigma$
    with $n$ inputs and one output (in particular, equipped with a cylindrical end  $\epsilon_p^{\pm}$ on every puncture $p$ of each component). 
    \item A labeling of each edge $e$ of the tree underlying $\Sigma$ by an element 
    $H^e\in\cH$.
    The $i$-th input edge, for $i>0$, is labeled by $H^i$ and the output by $H^0$. We denote by $T$ the given tree together with the labeling of the edges by Hamiltonians. We refer to $T$ as a Hamiltonian labeled tree.

  \item For each vertex $v\in T$ a choice of a pair $(H^v,\alpha_v)$ where $H^v:\Sigma_v\times M\to\bR$ is a smooth function, and $\alpha_v$ is a closed $1$-form on $\Sigma_v$.

\end{enumerate}

These are required to satisfy the following conditions:
\begin{enumerate}
\item On a bivalent vertex $v$, we have $H^{e_{out}}\neq H^{e_{in}}$.     
\item For all $x\in M$ we have the monotonicity inequality
    \begin{equation} \label{eq:monotonicity_multimorphism}
    dH^v_x\wedge \alpha_v\geq 0.
  \end{equation}
\item For any vertex $v$ and any puncture $p$ of $\Sigma_v$, there is a positive real number $w_p$ called the weight, such that we  have 
  \begin{align}
\label{eqTransInvDat0}
     \epsilon^{\pm,*}_{p}\alpha_v& =w_pdt \\
\label{eqTransInvDat}
        \epsilon_{p}^{\pm,*}(H^v\alpha_v) & =H^{e_p}dt.        
  \end{align}
  \item For each internal edge of $T$ with endpoints $v_-$ and $v_+$, the weights of $\alpha_{v_-}$ and $\alpha_{v_+}$ at the punctures associated to $e$ agree.

  \item For any vertex $v$, if $p_{out}$ is the positive puncture of $\Sigma_v$ and is $p$ a negative puncture of $\Sigma_v$: \begin{equation}\label{eq:weight_bound}\frac{w_{p}}{w_{p_{out}}}\geq \frac{1}{2^{|E_{in}(v)|}-1}.\end{equation}
  \end{enumerate}
\end{defin}
In Equation \eqref{eq:monotonicity_multimorphism}, we use the notation
\begin{equation}
 H^v_x(\cdot)=H^v(x,\cdot):\Sigma_v\to\bR.   
\end{equation}

Observe that by closedness of the $1$-forms $\alpha_v$, the output weight is the sum of the input weights at every vertex. The essential point in this definition is that, while one can define operations in Floer theory using Hamiltonian data of more general type than the split data that we choose (for example, one may consider a $1$-form valued in the space of Hamiltonians), imposing the analogue of the monotonicity constraints in Equation \eqref{eq:monotonicity_multimorphism} in such a general context results in a space whose homotopy type seems difficult to describe (even allowing for varying the choice of output Hamiltonian).

\begin{rem}
  The condition $H^{e_{out}}\neq H^{e_{in}}$ can be omitted at the cost of allowing arbitrarily long compositions of the constant continuation map with itself. Imposing it thus corresponds, in a certain way, to working with a reduced version of the theory. 

  Similarly, Inequality \eqref{eq:weight_bound} is imposed for convenience to obtain compactness of the space of allowable weights for fixed inputs and outputs. It could be omitted at the cost of changing the definitions so that the homotopy type of the spaces $\cH(H^1,\dots,H^n;H^0)$ is accurately reflected by the homotopy type of a sequence of exhausting subsets.

  The remaining conditions are unavoidable consequences of needing to ensure that, when we choose almost complex structure in Section \ref{sec:almost-compl-struct}, the resulting moduli spaces give rise to operations from the Floer complexes of the input Hamiltonian to that of the output Hamiltonian.  
\end{rem}

 \begin{defin} Two pre-multimorphisms are \emph{equivalent} if there is an isomorphism of the underlying pre-stable rational curves which intertwines the choices of cylindrical ends, the data of $H^e$, and finally the data of $\{(H^v,\alpha_v)\}$ up to the equivalence relation
 \begin{equation}
 \{(e^{r}H^v,e^{-r}\alpha_v)\}\sim \{(H^v,\alpha_v)\}\quad\forall r\in\bR.
\end{equation}\end{defin}

This leads to the following notion:
\begin{defin} 
  A \emph{multimorphism of dimension $0$} is an equivalence class of pre-multimorphisms. We denote by $\cH_0(H^1,\dots,H^n;H^0)$ the set of all multimorphisms of dimension $0$. For  a Hamiltonian labeled tree $T$ with $n$ inputs and one output we denote by $\cH_{0,T}(H^1,\dots,H^n;H^0)$ the set of multimorphisms defined on pre-stable rational curves modeled on $T$.
\end{defin}

\begin{rem}\label{remSmoohStructure}
For each fixed $T$, the set $\cH_{0,T}(H^1,\dots,H^k;H^0)$ has a natural smooth structure as an infinite dimensional Fr\'echet manifold. Namely, when  when $T$ consists of a single vertex, it a fiber bundle over the manifold of smooth framed curves with $k+1$ marked points; the fiber is analyzed in detail in Lemma \ref{lmForSmHoEq}. For general $T$ we may write a description as a fibre products of variants of these bundles in which the Hamiltonian is not fixed at some of the punctures. We shall avoid introducing a topology on the space  $\cH_{0}(H^1,\dots,H^k;H^0)$ of \textit{all} multimorphisms of dimension $0$, and instead will construct a cubical set of which these are the $0$-cubes.
\end{rem}

Part of the structure of a multicategory is an action, by relabelling, on multimorphisms. Concretely, the means that, given a permutation $\rho$ of the sequence $(1, \ldots, k)$, we need an isomorphism
\begin{equation}\label{eqHSymAc}
\cH_{0}(H^1,\dots,H^k;H^0)\to \cH_{0}(H^{\rho(1)},\dots,H^{\rho(k)};H^0),
\end{equation}
satisfying the property that the composition of the maps associated to permutations $\rho_1$ and $\rho_2$ to agree with the map associated to $\rho_1 \circ \rho_2$. This is in fact given by an elementary relabelling procedure on the set of pre-multimorphisms: we assign to $\Sigma$ the pre-stable rational framed curve $\rho_* \Sigma$ with the same underlying curve and framing, but with the input labels permuted by $\rho$. Since the sequence of inputs is also permuted by $\rho$, the rest of the data in the right hand side of Equation \eqref{eqHSymAc} (of labelling of trees, and choices of $1$-form, Hamiltonians, weights, and cylindrical ends), are canonically determined by this choice, and the imposed properties are preserved.

For the next definition, we consider sequences $\vec{H}^1\in\cH^{k_1}$ and $\vec{H}^2\in\cH^{k_2}$, where the notation indicates the Cartesian product on the object sets, and write
\begin{equation} \label{eq:multi-composition-Ham}
  \vec{H}_1 \circ_i \vec{H}_2  
\end{equation}
for the sequence with $k_1 + k_2 - 1$ elements obtained by replacing the $i$th component of $H_2$ with the vector $\vec{H}^1$. 

Given two  Hamiltonian labelled trees $T_1$ and $T_2$ such that $i$th input label of $T_2$ is the same as the output label of $T_1$, we can construct a new Hamiltonian labelled tree $T_1 \circ_i T_2$  by taking the disjoint union of $T_1$ and $T_2$ and identifying the output edge of $T_1$ with the $i$th input edge of $T_2.$ 

This operation extends to a multicomposition operation on multimorphisms. 
\begin{defin} \label{def:mult-dimens-0} 
  The multicomposition
\begin{equation}
\circ_i:\cH_0(\vec{H}_1;H_2^i)\times \cH_0(\vec{H}_2;H)\to\cH_0(\vec{H}_1 \circ_i \vec{H}_2 ;H)
\end{equation} takes $(\fd_1,\fd_2)$ to the multimorphism defined by attaching the output of $\fd_1$ to the $i$th input of $\fd_2$.
\end{defin}
The main point to check is the compatibility condition  in Definition \ref{def:pre-multimorphism-H} between the two data at the edge of $T_1 \circ_i T_2$ along which the trees are attached. This is ensured by choosing representatives of $\fd_1$ and $\fd_2$ such that the output weight of $\fd_1$ agrees with the $i$th input weight of $\fd_2$. To avoid any confusion, we reiterate that this definition does not involve any gluing of Riemann surfaces (or Hamiltonian data), but uses only attaching pre-stable Riemann surfaces as in Definition \ref{def-pre-stable-rational-curve} (carrying, in addition, Hamiltonian data). 

\subsection{Higher cubes of multimorphisms}
\label{sec:high-cubes-mult}

Informally speaking, we shall define positive dimensional cubes of multimorphisms as cube families of multimorphisms of dimension $0$ which are obtained by gluing near the boundary strata. We think of this as a replacement for the more naive strategy of defining a smooth structure on the set of multimorphisms of dimension $0$ and considering the smooth singular cubes of this target. This alternative strategy appears to be technically much more complicated because the space of choices is infinite dimensional, and there seems to be no natural way to equip it with a manifold structure that is consistent with the breaking of Riemann surfaces.

We now introduce some notation for the gluing operation that will be used in the definition of higher dimensional cubes. We consider an element $\fd\in\cH_0(H^1,\dots,H^n;H^0)$ with underlying Hamiltonian labeled tree $T$, and refer the reader to Definition \ref{def-gluing-general} for the precise way in which we formulate gluing of Riemann surfaces:

\begin{defin} The gluing of $\fd$ along parameters $\vec{r}\in [0,1]^{E_{int}(T)}$ is the multimorphism obtained by gluing the underlying framed pre-stable curves, equipped with the restriction of the data $(H^{v_-},\alpha^{v_-})$ and $(H^{v_+},\alpha_{v_+})$ for each $e\in E_{int}(T)$ connecting vertices $v_-$ and $v_+$ with $r_e>0$. \end{defin}

Note that the compatibility condition of weights along the two sides of the node, which we imposed in Definition \ref{def:pre-multimorphism-H},  ensures that the Hamiltonian data on the glued Riemann surface are well-defined after gluing (along the boundary identification in Equation \eqref{eq:gluing_along_boundary}). In addition, the equality $2^{n-1}\cdot 2^{m-1}=2^{(n+m-1)-1}$ shows that Condition \eqref{eq:weight_bound} is satisfied.

Given $\fd\in \cH_0(\vec{H},H)$, which we stress can be broken, we call the triple $(T,\fd',\vec{r})$ of a Hamiltonian labeled tree $T$, $\fd'\in \cH_{0,T}(\vec{H},H)$ and  $\vec{r}:{|E_{int}(T)|}\to [0,1]$ such that \begin{equation}
\fd=\Gamma_{\vec{r}}(\fd') 
\end{equation}\emph{a gluing decomposition} of $\fd$. A given element $\fd$ of $ \cH_0(\vec{H},H)$ may have more than one gluing decomposition, but the choice is fixed by choosing the tree labelling the decomposition. We shall use the following weaker result, in which one fixes the source of the gluing map:

\begin{lemma} Any two gluing decompositions of $\fd$ with the same underlying tree $T$ and 
Hamiltonian datum $\fd'\in \cH_{0,T}(\vec{H},H) $ have equal gluing parameters. 
\end{lemma}
\begin{proof}
  We must show that the gluing parameters can be reconstructed from $\fd$. We shall not do so directly, but rather, we show that, for each edge $e$ of $T$, a strictly monotonic function $p_e$ of the gluing parameter can be thus reconstructed. For the discussion below, we shall use the fact that such an edge determines a (free) homotopy class of circles in $\Sigma$.

We define $p_e(\vec{r})$ as the infimum 
of the product of quantities $(p_1, \ldots, p_k)$ in $(0,1)$ such that there are disjoint holomorphic embeddings $E:(\log(p_j),0)\times \bRZ\to \Sigma_v$, for some component $\Sigma_v$ of the target, in the specified homotopy class, satisfying the following properties
\begin{align}
  E^* \alpha_v & = c dt \textrm{ for some }    c>0 \\
  E^*\left(  H^{v}_{x} \alpha_v \right) & =  H^e dt,
\end{align}
assuming that such an embeddings exists. Note that, since the gluing annulus associated to $e$ satisfies the above property, the only way for such an embedding to fail to exist is if the gluing parameter equals $1$ (and there is no way to extend it holomorphically so that the stated properties hold), in which case  we set $p_e(\vec{r})$ to be $1$. The other extreme case, with $p_e(\vec{r})=0$,  corresponds to the case in which $e$ is not collapsed.
Let us call the collection of maximal embeddings $E_e$.

Note that, since $\alpha$ does not vanish on $E_e$, its restriction to it is a foliation, which by the requirement that $\alpha$ pulls back to $dt$, has closed leaves. The condition that $H^{e} \leq H^{e'}$ for any edge $e'$ that succeeds $e'$ along the arc to the output, with strict inequality at some point, implies that the sets $E_e$ and $E_{e'}$ are disjoint whenever $e \neq e'$.

The monotonicity condition further implies that $E_e$ is contained in the image, under gluing, of the Riemann surfaces $\Sigma_{v_-}$ and $\Sigma_{v_+}$ which are associated to the endpoints of $e$. This implies that  $p_e(\vec{r})$  depends only on the gluing parameter $r_e$. Since $E_e$ contains the gluing annulus associated to $e$, 
we conclude that  $ p_e(\vec{r})$  is a monotonic function of $r_e$ because the gluing region, which shrinks when the gluing parameter increases, while the remaining regions do not change.

\end{proof}

The next definition describes the local models for the cubes of multimorphisms that we will consider. \begin{defin}\label{dfLocModMult}
  A \emph{local model for a codimension $k$ corner of a family of multimorphisms} with input $\vec{H}$ and output $H$ consists of the following data:

\begin{enumerate}
\item (Domain of maximal breaking) a smooth manifold $C$, 
\item  (Collar neighbourhood) an open neighbourhood  $U$ of the origin in $[0,1]^k$ (in the case $k=0$ we write $U=\{0\}$).
\item (Choice of broken curves) A Hamiltonian labeled tree $T$, and a smooth map
  \begin{equation} \label{eq:broken_curve_family}
        b \co C\times U \to\cH_{0,T}(\vec{H},H),
  \end{equation}
where the target is smooth as in Remark \ref{remSmoohStructure}.
\item (Gluing data) A smooth map
  \begin{equation} \label{eq:gluing_data}
   g:C\times U\to[0,1)^{|E_{int}(T)|}     
  \end{equation}
whose components vanish identically on $C\times\{0\}$.  Moreover, for each  face $\sigma$ of $U$, any component $g_i$ of $g$ either vanishes identically or is nowhere $0$ on the interior of $C\times\sigma$.
\end{enumerate}
\end{defin}

We are now ready to define the cubes of multimorphisms in $\cH$.  For the definition we fix once and for all a positive number $\epsilon_0 < 1/2$. For each $n$ denote by $F_n$ the set of faces of the $n$-cube $[0,1]^n$ (we include the top stratum among them, so that  $F_n$ has $3^n$ elements). For each $f\in F_n$, denote by $f^o$, the interior of $f$.  For each $f\in F_n$, let $W_f$ be the image in the $n$-cube of the canonical affine embedding of 
\begin{equation} \label{eq:decomposition_W_f}
f^o\times[0,\epsilon_0)^{\mathrm{codim}(f)}.
\end{equation}
Denote by  $\cV_n$ the open cover of $[0,1]^n$  given by $\cV_n=\{W_f\}_{f\in F_n}$. Note that there are natural identifications $F_n=(F_1)^n$ and $\cV_n=(\cV_1)^n$ where an $n$-tuple $(W_{f_1},\dots,W_{f_n})\in(\cV_1)^n$ is identified with the Cartesian product $(W_{f_1}\times \dots\times W_{f_n})=W_{f_1\times \dots\times f_n}$. Thus the cover $\cV_n$ is compatible with intersection with faces of the $n$-cube and is equivariant with respect to the action by transposition of the coordinates on $F_n$. 

\begin{defin}\label{dfCuMultMor}
  An $n$-cube $\fd$ in $\cH(\vec{H},H)$ consists of a family of Hamiltonian data given by a choice $(T_f, b_f,g_f)_{f\in F_n}$ of a local model for each face $f$ of the cube, with domain the set $W_f$, equipped with the decomposition from Equation \eqref{eq:decomposition_W_f} so that the following property holds: given an inclusion $f_0\subset f_1$, let $E$ be the set of edges of $T_{f_0}$ that are collapsed under the map $T_{f_0}\to  T_{f_1}$. Let $g^E_{f_0}:V\times U\to[0,1)^{E}$ be the composition of $g_{f_0}$ with the projection $[0,1)^{|E_{int}(T)|}\to [0,1)^{E}$. Then we have
\begin{equation}\label{eqGluingCompatibility}
b_{f_1}|_{W_{f_0}}=\Gamma_{g^E_{f_0}}(b_{f_0}).
\end{equation}
\end{defin}
To gain some intuition for this definition, the reader may want to  note that from this data we get, for each inclusion $f_1\subset f_2$ of faces, a surjective map $\sigma_{12}:V(T_{f_1})\to V(T_{f_2})$. If $v_1$ and $v_2$ are adjacent vertices of $T_{f_1}$, then whenever the corresponding components are glued, we have $\sigma_{12}(v_1)=\sigma_{12}(v_2)$. Alternatively,  if the gluing parameter of the edge connecting them vanishes, these vertices remain adjacent in $T_{f_2}$.

An $n$-cube determines a map of sets
  \begin{equation}\label{eqUnderlyingMap}
        b:[0,1]^n\to\cH_0(\vec{H},H),
      \end{equation}
      which is given by the formula
\begin{equation}\label{eqLocModMult}
b(p)=\Gamma_{g_f(x,y)}(b_f(x,y)). 
\end{equation}
whenever $p=(x,y)\in V_f\times U_f=W_f$. The compatibility condition in Equation \eqref{eqGluingCompatibility} implies that this expression is well-defined. The collection of local models $(T_f, b_f,g_f)_{f\in F_n}$ is called a \emph{gluing atlas}. We may have some distinct $n$-cubes whose underlying map $b$ is the same if $b$ contains broken rational curves.

\begin{rem}
  It is tempting to try to simplify Definition \ref{dfLocModMult} by making the family of Riemann surfaces depend only on the space $C$ (which in the case of interest corresponds to the interior of a face), and the gluing parameters only on the factor $U$ (which corresponds to its normal direction). This would require us to restrict the class of breakings that are allowed to take place at a corner. For example, Figure \ref{fig:breaking_requiring_general_gluig} shows a situation where the Riemann surfaces break twice in the corner of a square, once along one of the adjacent edges, and do not break along the other one. There is no gluing parameter associated to the normal direction of the edge along which no breaking take place, but the family of Riemann surfaces in a neighbourhood of the corner must depend on two parameters. Using the notation from the definition, this forces us to allow $b$ to depend on a tubular neighbourhood of the edge. An analogous argument, involving an edge of a square, labelled by a tree with two internal edges, so that the adjacent edges are labelled by non-isomorphic trees, shows that the parameter $g$ also must in general depend on the entire tubular neighbourhood. 
  
\begin{figure}[h]
  \centering
  \begin{tikzpicture}
     
     \filldraw (0,0) circle (2pt);
     \draw[thick] (0,4) -- (0,0) -- (4,0);
     \begin{scope}[shift={(-1.25,-1.25)},scale=.5]
            \draw (-2,2) -- (0,0) -- (2,2);
            \filldraw (0,0) circle (3pt);
            \draw (0,0) -- (0,-1);
       \filldraw (1.25,1.25) circle (3pt);
       \draw (1.25,1.25) -- (.75,2);
        \filldraw (-1.25,1.25) circle (3pt);
       \draw (-1.25,1.25) -- (-.75,2);
     \end{scope}
      \begin{scope}[shift={(-1.25,2)},scale=.5]

       \draw (-2,2) -- (0,0) -- (2,2);
 
       \filldraw (0,0) circle (3pt);

       \draw (0,0) -- (0,-1);
       \draw (0,0) -- (.75,2);
               \draw (0,0) -- (-.75,2);
     \end{scope}
      \begin{scope}[shift={(2,-1.25)},scale=.5]
    \draw (-2,2) -- (0,0) -- (2,2);

       \filldraw (0,0) circle (3pt);
 
       \draw (0,0) -- (0,-1);
       \draw (0,0) -- (.75,2);
            \filldraw (-1.25,1.25) circle (3pt);
            \draw (-1.25,1.25) -- (-.75,2);
     \end{scope}
      \begin{scope}[shift={(2,2)},scale=.5]

       \draw (-2,2) -- (0,0) -- (2,2);

       \filldraw (0,0) circle (3pt);

       \draw (0,0) -- (0,-1);
       \draw (0,0) -- (.75,2);
               \draw (0,0) -- (-.75,2);
     \end{scope}
  \end{tikzpicture}
  \caption{An assignment of trees to a strata near a corner which requires the gluing atlas to depend on neighbourhood of the strata. }
  \label{fig:breaking_requiring_general_gluig}
\end{figure}
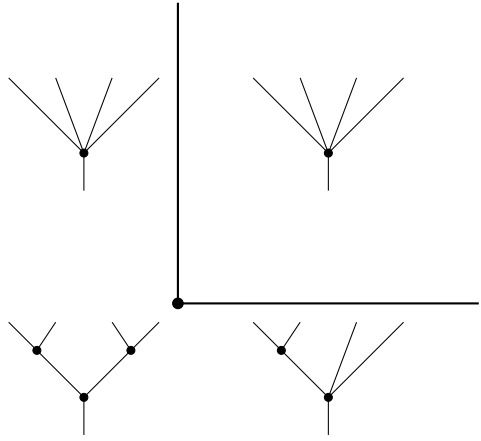

One could imagine putting a stronger constraint on the set of allowed cubes in order to avoid this, but this will result in the resulting cubical set failing to satisfy good homotopical properties (in particular, the Kan property). Concretely, as soon as one formulates a definition of $1$-cubes in which the edges appearing in Figure \ref{fig:breaking_requiring_general_gluig} are allowed, the Kan property would require that these edges can appear as a corner of a $2$-cube, and the only possible combinatorial data underlying such a $2$-cube is the one shown. While we do not explicitly use the Kan property, our proof of Lemma \ref{lmsmincHoEq} uses a construction analogous to what would be required to establish it, and would fail for this reason if we restricted the class of allowed cubes to avoid this issue. 

\end{rem}

Next, we note some basic consequences of the definition which will later be used:

\begin{lemma}
  The following properties hold for each $n$-cube $(T_f, b_f,g_f)_{f\in F_n}$ in $\cH(\vec{H},H)$:
\begin{enumerate}
\item If for $x\in f^o$, $b(x)$ is smooth, then $b_f(x)$ is also smooth and equal to $b(x).$
\item If $x\in f^o$, then $b(x)\in \cH_{0,T_f}(\vec{H},H).$
\end{enumerate}
Given another $n$-cube $(b', (T_f', b_f',g_f')_{f\in F_n})$, we have:
\begin{enumerate}
\item If $b=b'$, then $T_f=T_f'$ for all $f\in F_n$.
\item If $b=b'$ and $b_f=b_f'$ for all $f\in F_n$, then $g_f=g_f'$ for all $f\in F_n$.
\end{enumerate} \qed
\end{lemma}

The next result justifies the terminology of $n$-cube that we have been using.

\begin{lem}
  The collections of sets $\{\cH_n(\vec{H},H)\}$ are the underlying sets of $n$-cubes of a symmetric cubical set.
\end{lem}
\begin{proof}
We mention what happens to the part of the data consisting of the map $b:[0,1]^n\to\cH_0(\vec{H},H)$  and omit writing the natural operations on the gluing atlas.   In order to shorten the notation, denote by $\cH_n$ the set of $n$-cubes in $\cH(\vec{H},H)$. 

In the following $\iota^{\pm}_{n,i}(t):[0,1]^n\to [0,1]^{n+1}$
is the standard embedding as the face $x_i=1/2\pm1/2$,  $\pi_{n+1,i}:[0,1]^{n+1}\to[0,1]^n$ is the projection forgetting the $i$th coordinate, and $\tau_{n,i}:[0,1]^n\to [0,1]^n$ is the map which transposes the $i$th and $(i+1)$th coordinates.

Define the face maps
\begin{equation}
d^\pm_{n,i}:\cH_n(\vec{H},H)\to\cH_{n-1}(\vec{H},H)
\end{equation}
by $b\mapsto b\circ\iota^\pm_{n,i}$. Similarly, degeneracy maps 
\begin{equation}
\sigma_i:\cH_n(\vec{H},H)\to\cH_{n+1}(\vec{H},H)
\end{equation}
are defined by composition $b\mapsto b\circ \pi_i$. Finally, transposition maps 
\begin{equation}
p_{n,i}:\cH_{n}(\vec{H},H)\to\cH_n(\vec{H},H)
\end{equation}
are defined by $b\mapsto b\circ \tau_{n,i}$. 
It is readily verified that the face, degeneracy, and transposition maps are well defined and satisfy Equations \eqref{eqCubSet1}--\eqref{eqCubSet6}, yielding a symmetric cubical set.  
\end{proof}

The next definition is the generalisation of Definition \ref{def:mult-dimens-0}:
\begin{defin}
For each integer $i$ between $1$ and the length of a sequence $\vec{H}^2$ of Hamiltonians, define the multicomposition map 
\begin{equation}\label{eqHMultComp}
\circ_i:\cH_{n_1}(\vec{H}^1;H^{2,i})\times \cH_{n_2}(\vec{H}^2;H)\to\cH_{n_1+n_2}(  \vec{H}_1 \circ_i \vec{H}_2  ;H)
\end{equation}
to be the map determined by assigning to a pair $(\fd_1,\fd_2)$ the product $n_1 + n_2$-cube given by mapping a point $(x,y)$ to the composition $\fd_2(y)\circ_i\fd_1(x)$, and by taking the product gluing atlas.

\end{defin}

It is straightforward to check that Equation \eqref{eqHMultComp} is equivariant with respect to the product action by $ \Sigma_{n_1} \times \Sigma_{n_2}$ on the left, and the restriction of the action by $\Sigma_n$ on the right, since both of them simply act by permuting the coordinates of the corresponding cubes. Similarly, boundaries and degeneracies are defined in the same way on the two sides in terms of inclusions and projections of cubes. We conclude:
\begin{lem} 
  The map given in Equation \eqref{eqHMultComp} determines a map of symmetric cubical sets:
  \begin{equation} \label{eq:H_multicomposition_cubical_sets}
       \circ_i:\cH_{\bullet}(\vec{H}^1;H^{2,i})\otimes \cH_{\bullet}(\vec{H}^2;H)\to\cH_{\bullet}(  \vec{H}_1 \circ_i \vec{H}_2  ;H) .
  \end{equation} \qed
\end{lem}

\begin{rem}
  We warn the reader that the set of $n$-cubes of the tensor product $\cH_{\bullet}(\vec{H}^1;H^{2,i})\otimes \cH_{\bullet}(\vec{H}^2;H) $, which is described in detail in Appendix \ref{sec:monoidal-structure}, is not given by the union over $n_1 + n_2 = n$ of the left hand sides of Equation \eqref{eqHMultComp}, i.e. by pairs $(\fd_1,\fd_2)$ of $n_1$ and $n_2$ cubes.  The problem is that only the product $\Sigma_{n_1} \times \Sigma_{n_2}$ acts on the set of such pairs, whereas an action of $\Sigma_n$ is required as part of a symmetric cube structure. 

  Nonetheless, Equation \eqref{eqHMultComp} does determine, as asserted, a unique $\Sigma_n$-equivariant map from the $n$-cubes of $\cH_{\bullet}(\vec{H}^1;H^{2,i})\otimes \cH_{\bullet}(\vec{H}^2;H) $, to the $n$-cubes of the target. The reason is that such cubes are in fact given by a pair $(\fd_1,\fd_2)$ as above, together with the additional datum of an element $\sigma$ of $\Sigma_n$, modulo the action of $\Sigma_{n_1} \times \Sigma_{n_2} $  (see Equation \eqref{dfSymMonStr}). Denoting by
\begin{equation}
    \tau_{\sigma}:[0,1]^n\to[0,1]^n
  \end{equation}
  the permutation of the $n$-cube corresponding to $\sigma$, the assignment of Equation \eqref{eqHMultComp} canonically extends to the map of symmetric cubical sets which takes  $[\sigma,(\fd_1,\fd_2)]$ to the $(n_1+n_2)$-cube mapping the point $\tau_{\sigma}(x,y)$ to the composition $\fd_2(y)\circ_i\fd_1(x)$.
\end{rem}

At this stage, one can easily check that the compositions we have just defined satisfy the  associativity relations from Equations \eqref{eqMultComp1}--\eqref{eqMultComp3}, and that the permutation action given as in Equation \eqref{eqHSymAc} by relabelling the underlying trees and Riemann surfaces, satisfies Equations \eqref{eqSymAc1} and \eqref{eqSymAc2}.

To complete the construction of the multicategory of Hamiltonians, recall that we imposed, in Definition \ref{def:pre-multimorphism-H}, the condition that the input and output of each multimorphism with domain a cylinder be different. This implies that our construction so far has the property that morphisms from a Hamiltonian to itself are empty.

\begin{defin}
The multicategory $\cH$ of Hamiltonian data is the following symmetric multicategory enriched in symmetric cubical sets.  
\begin{itemize}
    \item The objects are the elements of $\cH$. \item The multimorphisms are obtained by considering the symmetric cubical sets $\cH_\bullet(\vec{H},H)$, and  formally adding units.   
    \item The multicompositions given by Equation \eqref{eq:H_multicomposition_cubical_sets}, with symmetric group action given by Equation \eqref{eqHSymAc}.
    \end{itemize}
\end{defin}

\subsection{Multimorphisms and the forgetful map}
Denote by $f\Mbar^{\bR}_{0}(k)_{\bullet}$ the symmetric cubical set of maps from cubes to the $k$th space $f\Mbar^{\bR}_{0}(k)$ of the Kimura-Stasheff-Voronov operad, as defined in Section \ref{sec:stab-ksv-moduli}. Given any sequence of Hamiltonians $\vec{H}\in\cH^k$, there is a natural map
\begin{equation}
\fF:\cH_\bullet(\vec{H},H)\to f\Mbar^{\bR}_{0}(k)_{\bullet}
\end{equation}
which forgets the Hamiltonian data and collapses the unstable components of the underlying framed curve. We refer the reader to Appendix \ref{sec:stab-ksv-moduli} for a detailed discussion of this stabilization process. It is easy to see that the collection of maps $\pi$ as $k$ varies assembles into a forgetful functor from the multicategory $\cH$ to the KSV operad. We declare the unit of $\cH(H,H)$ to map to the unit in $S^1=f\Mbar^{\bR}_{0}(1)_{0}$.

An essential difficulty in our construction is the fact that the map $\pi$ is not in general a homotopy equivalence. The following definition identifies a condition which will ensure this:
\begin{defin}\label{defHamOrdering}
Let $\vec{H}=(H^1,\dots,H^k)\in\cH^k$ and let $H^0\in\cH$. We say that $H^0>\vec{H}$ if for every $x \in M$, and every $ 1\leq i\leq k$, we have 
\begin{equation}
\min_{t\in \bR/\bZ}H^0_t(x)>2^{k-1} \max_{t\in\bR/\bZ} H^i_t(x). 
\end{equation}  
\end{defin}
For the next result, we consider a smooth punctured Riemann surface $\Sigma$ with $k$ inputs and $1$ output:
\begin{lemma}\label{lmExMonH}
If $H^0>\vec{H}$, then for any closed $1$-form $\alpha$ on $\Sigma$ which agrees with $w_idt$ near the $i$th input and with $dt$ near the output and such that $w_i\geq 2^{1-k}$, there exists a function
\begin{equation}
 H_\Sigma:\Sigma\times M\to\bR 
\end{equation}
so  that $H_{\Sigma}\alpha$ pulls back to $H^idt$ with respect to the cylindrical end for the $i$th input, and $H_{\Sigma}\alpha$ pulls back to $H^0dt$ with respect to the cylindrical end near the output, and so that the following condition holds for each $x\in M$:
\begin{equation} \label{eq:monotonicity_HSigma}
dH_{\Sigma}(x)\wedge\alpha\geq 0.
\end{equation}
\end{lemma}
\begin{proof}
  Extend each negative cylindrical end to an embedding of $ (-\infty,\delta] \times S^1$ for a small constant $\delta$, and the positive cylindrical end to an embedding of $[-\delta, \infty) \times S^1$, so that the images remain disjoint, and so that to pullback of $\alpha$ agrees with $w_i dt$. We call the images of these larger domains, the \emph{extended cylindrical ends}.

  Let $H: M\to\bR$ be a smooth function so that for any $x\in M$, and for all $ 1\leq i\leq k$,  we have 
\begin{equation}
\min_{t\in \bR/\bZ}H^0_t(x)> H(x)>2^{k-1} \max_{t\in\bR/\bZ} H^i_t(x)\geq\max_{t\in\bR/\bZ}\frac1{w_i}H^i_t(x).
\end{equation}  
We define the function $H_{\Sigma}$ piecewise as follows:
\begin{enumerate}
\item In the image of the cylindrical ends, $H_{\Sigma} $ agrees with $w_i H^i$.
\item Away from the images of the extended cylindrical ends, $H_{\Sigma}\equiv H$.
\item In the annuli $[0,\delta] \times S^1$ (or $[-\delta, 0] \times S^1$), we define $H_{\Sigma} $ by linearly interpolating, along the radial direction, between $w_iH^i$ and $H$, taking $w_0=1$.
\end{enumerate}
 The monotonicity condition in Equation \eqref{eq:monotonicity_HSigma} then follows because the resulting function increases along the radial coordinates near each end, and is domain independent away from them. 

\end{proof}
The following is the main result of this section. 
\begin{prop}\label{lmFrgtHtpyEq}
If $H^0>\vec{H}=(H_1,\dots, H_k)$, then the  map
\begin{equation}
\fF:\cH_{\bullet}(\vec{H},H^0)\to f\Mbar^{\bR}_{0}(k)_{\bullet}
\end{equation}
is a homotopy equivalence of symmetric cubical sets.
\end{prop}
The proof will given at the end of this section, after a few preparatory lemmas. Denote by $f\cM_{0,k+1}$ the interior of $f\overline{\cM}_{0,k+1}^{\bR}$, which can be alternately described as the set of framed smooth curves, which is an $(S^1)^{k+1}$ bundle over the interior $\cM_{0,k+1}$ of Deligne-Mumford space. Denote by $\cH^{sm}_{0}(\vec{H},H^0)$ the topological space of $0$-dimensional multimorphisms whose underlying rational curve is smooth. 
\begin{lemma}\label{lmForSmHoEq}
The map
\begin{equation}
\fF:\cH^{sm}_{0}(\vec{H},H^0)\to f\cM_{0,k+1}
\end{equation}
is a homotopy equivalence of topological spaces.
\end{lemma}

\begin{proof}  Denote by $\mathcal{D}$ the space of Riemann surfaces with $k$ inputs and $1$ output. We consider the space  $\mathcal{A}$ consisting of an element of $\mathcal{D}$ and a closed one form on the underlying Riemann surface satisfying the condition in Equation \eqref{eqTransInvDat0} for some choice of weights satisfying Condition \eqref{eq:weight_bound} with the output weight equal to $1$, up to biholomorphisms of Riemann surfaces which intertwine the rest of the data.

Note that we can find a unique element of the form $({\Sigma},e^rH_\Sigma,e^{-r}\alpha,\{\epsilon^{\pm}_p\})$ in any equivalence class $[({\Sigma},H_\Sigma,\alpha,\{\epsilon^{\pm}_p\})]\in \cH^{sm}_{0}(\vec{H},H^0)$ such that the output weight is $1$. Let us assume that the representatives that we use below satisfy this property.

The map $\fF$ factors through the projection map from $\mathcal{A}$ to $\mathcal{D} $ as follows:
\begin{equation} \label{eq:factor-forgetful-map}
[({\Sigma},H_\Sigma,\alpha,\{\epsilon^{\pm}_p\})]\mapsto [({\Sigma},\alpha,\{\epsilon^{\pm}_p\})]\mapsto [({\Sigma},\{\epsilon^{\pm}_p\})]\mapsto [\Sigma] ,
\end{equation}
which first forgets the choice of Hamiltonians $H_{\sigma}$, then the $1$-form $\alpha_{\vec{w}}$, and finally the choices of cylindrical ends. It is easy to see that all of these maps are fiber bundles.

We shall show that each of these maps have contractible fibres. The fibres of the first two maps are convex (with respect to the linear structures on the set of Hamiltonians and $1$-forms), so it suffices to show that they are non-empty: for the first map, this is the content of Lemma \ref{lmExMonH}, while for the middle map, this is a consequence of the de-Rham isomorphism which implies we can find a $1$-form $\overline{\alpha}$ with residue $w_i$ on the $i$th input and residue $1$ on the output as long as $w_1+\ldots+w_n=1$. The $1$-form $\overline{\alpha}$ may not be cylindrical with respect to our chosen ends, but this is easily fixed by adding an exact $1$-form. We also used the fact that $w_i\geq 2^{1-n}$ is a convex condition and that there are solutions of $w_1+\ldots+w_n=1$ satisfying this property for all $1\leq i\leq n$, e.g $w_i=1/n$.

It remains to show that the forgetful map from $\cD$ to the moduli space of framed curves has contractible fibres. Note that  given a cylindrical end $\epsilon:(-\infty,0]\times\bRZ  \to\Sigma$ and a non-negative real number $r$, we can obtain another cylindrical end by composing the translation map $(-\infty,0]\times\bRZ\to  (-\infty,-r]\times\bRZ$ with the restriction of $\epsilon$ to $(-\infty,-r]\times\bRZ$.  

Applying this restriction and translation operation, we find that the space of cylindrical ends deformation retracts onto the subset of those whose images lie in the interior of the image of a fixed cylindrical end. We now work on one puncture at a time. The space of cylindrical ends with prescribed tangent ray (which map into the image of a fixed cylindrical end) can be identified with the space of holomorphic embeddings $f$ from the closed unit disk  into the open unit disk preserving the origin and such that $f'(0)\in\bR_+$. By a similar argument we can replace the target fixed disk with $\mathbb{C}$.

By scaling down (multiply with $r$ where $0<r<1$), restricting and scaling up (multiply with $r^{-1}$) we define a contracting flow on the space of such embeddings. The fixed points of this flow are linear embeddings. Using the asymptotic condition these are just the ones given by real scalar multiplication, which is evidently forms a contractible set, proving the desired result.

\end{proof}

For the next result, we denote by $\Box^{sm}_{\bullet} M$ the smooth singular cubes of a differentiable manifold $M$.

\begin{lemma}\label{lmsmincHoEq}
 
  The inclusion in  \begin{equation} \label{eq:include_smooth_gluable_cubes}
   \iota_*: \Box^{sm}_{\bullet}(\cH^{sm}_{0}(\vec{H},H^0) )  \hookrightarrow \cH_{\bullet}(\vec{H},H^0).
 \end{equation}
  is a homotopy equivalence of symmetric cubical sets.
\end{lemma}
\begin{proof}
Abbreviate  $Y_{\bullet}=\Box^{sm}_{\bullet}(\cH^{sm}_{0}(\vec{H},H^0) )$ and $X_{\bullet}=\cH_{\bullet}(\vec{H},H^0)$. We show that there exists a deformation retraction of $X_{\bullet}$ onto $Y_{\bullet}$. That is, we construct a map of symmetric cubical sets $\rho_{\bullet}: X_{\bullet}\to PX_{\bullet}$ satisfying
\begin{equation}\label{eqCubDefRet}
d^-\circ \rho_{\bullet}=\id,\quad d^+\circ \rho_{\bullet}(X^r_{\bullet})\subset Y_{\bullet},\quad \rho_{\bullet}\circ\iota_{\bullet}=P(\iota_{\bullet})\circ s_{*,1}.
\end{equation}
It is easy to see that a deformation retraction of symmetric cubical sets is a homotopy equivalence.

Given a cube $\fd$ we construct $\rho_{\bullet}(\fd)$. For a zero cube $\fd$ we consider the $1$-cube $\rho_{\bullet}(\fd)=(b_f,g_f)_{f\in F_1} \in X_1$ defined by 
\begin{equation}
b:=\{   t\mapsto
\Gamma_{t/2}(\fd) \}\in X_1^r
\end{equation}
And the gluing atlas $(b_f,g_f)_{f\in F_1}$ is the obvious one. Note that, whenever $\fd$ lies in $\iota(Y_{\bullet})$, the glued curve $\Gamma_{t/2}(\fd)$ does not depend on $t$, so we have
\begin{equation}
  \rho_{\bullet}(\fd) = s_{0,1}(\fd) \quad \textrm{ if } \fd\in \iota(Y_{\bullet}).  
\end{equation}

To define $\rho$ for higher dimensional cubes recall the cover $\cV_n$ introduced right after  Definition \ref{dfLocModMult}. Fix once and for all a partition of unity $\{\psi_{f,1}\}$ subordinate to the cover $\cV_1$ of the unit interval. This induces a corresponding  partition of unity $\{\psi_{f,n}\}$ subordinate to the cover $\cV_n$ by taking products.

It is straightforward to check that the partition of unity $\{\psi_{f,n}\}$  is compatible with restriction to faces and equivariant with respect to transposition. That is, whenever $f$ is contained in a codimension $1$ face $f'$ we have 
\begin{equation}\label{eqPOUComp}
\psi_{f,n}|_{f'}=\psi_{f,n-1},
\end{equation}
and for any $\sigma\in\Sigma_n$ we have
\begin{equation}
\psi_{\sigma(f),n}=\psi_{f,n}\circ p_{\sigma}.
\end{equation}

We now proceed to define the retraction on an $n$-cube $\fd$.  We  first define  the underlying map $\rho(b)(t,x)$ of $\rho(\fd)$ pointwise for $(t,x)\in[0,1]\times [0,1]^n$. Let $f_0$ be the smallest face  for which $x\in W_{f_0}$. 
Then, by definition,
 \begin{equation}
b(x)=\Gamma_{g_{f_0}(x)}(b_{f_0}(x)).
 \end{equation}
 
 Define a function $\tilde{g}:[0,1]\times W_{f_0}\to [0,1)^{E(T_{f_0})}$ as follows. For each $e\in T_{f_0}$ let $S_{e,f_0}$ be the set of faces $f'\supseteq f_0$ for which $e$ is not collapsed under the map $T_{f_0}\to T_{f'}$. Let 
 \begin{multline}\label{eqTildG}
 \tilde{g}_{f_0,e}(t,x):=\left(\sum_{f'\in S_{e,f_0}}\psi_{f'}(x)\right)(t/2(1-g_{f_0,e}(x))+(1-t)g_{f_0,e}(x))\\ +\left(\sum_{f'\not\in S_{e,f_0}}\psi_{f'}(x)\right)g_{f_0,e}(x).
 \end{multline}

For $x$ in the $n$-cube, let  $f(x)$ be the smallest face $f_0$ for which $x\in W_{f_0}$. 
Define 
\begin{equation}\label{eqGlRetF}
\rho(b)(t,x)=\Gamma_{\tilde{g}_{f(x)}(t,x)}(b_f(x)).
\end{equation}

It remains to construct a gluing atlas. To prepare the ground for this, define for each face $f$ of the $n$-cube a function $\hat{g}_f:[0,1]\times W_f\to [0,1)^{E(T_{f})}$ by
\begin{equation}
\hat{g}_{f,e}(t,x):=\tilde{g}_{f(x),e}(t,x)
\end{equation}

We turn to define a function $\hat{b}_f(t,x):[0,1]\times W_f\to  \cH_{0,T_f}(\vec{H},H)$. For any $f_0\subset f$ let $E^{f_0,f}$ be the set of edges of $T_{f_0}$ that are collapsed under the map $T_{f_0}\to  T_{f}$. 
Let $\hat{g}^{f_0,f}:[0,1]\times W_{f_0}\to[0,1)^{E^{f_0,f}}$ be the composition of $\hat{g}_{f_0}$ with the projection $[0,1)^{E_{int}(T_{f_0})}\to [0,1)^{E^{f_0,f}}$. Define
\begin{equation}
\hat{b}_{f}(t,x)=\Gamma_{\hat{g}^{f(x),f}(t,x)}(b_{f(x)}(x)).
\end{equation}

Before proceeding, we prove that $\hat{g}_f$ and $\hat{b}_f$ are smooth functions. This needs to be verified for points where $f(x)$ changes. That is, points on the boundaries of $W_{f_0}$ for $f_0\subset f$ (see Figure \ref{fig:gluing_nearby_face}). The change in $\hat{g}_f$ upon crossing the boundary of $W_{f_0}$ at a point $x$ in the boundary  amounts to replacing the expression $\psi_{f_1}(x)g_{f,e}(x)$ in Equation \eqref{eqTildG} by the expression $\psi_{f_1}(x)(t/2(1-g_{f_0,e}(x))+(1-t)g_{f_0,e}(x))$ for each face $f_1$ that contains the face $f_0$ but not the the face $f(x)$. Since $\psi_{f_1}$ vanishes identically near the boundary of $W_{f_0}$ the smoothness of $\hat{g}_f$ follows. The smoothness of $\hat{b}_f$ is similarly implied by Equation \eqref{eqGluingCompatibility} by the same vanishing.

We now return to the task of constructing the gluing atlas. For the faces $\{0\}\times f$ with $f\in F_n$,  we take the local model consisting of the data of the underlying set $\hat{W}_f=V_f\times [0,1)\times U_f$, the underlying tree $\hat{T}_f=T_f$, the broken curve map $\hat{b}_f$, and  gluing function $\hat{g}_f$ just defined.  It is clear by construction and by the smoothness of $\hat{g}_f$ that this satisfies the requirements of a local model for a corner and that $b(x)|_{W_f}=\Gamma_{\hat{g}_f(x)}(\hat{b}_f(x))$. 

For all the other faces $f$ we have that $T_f$ consists of a single vertex. Thus  the local model is necessarily the restriction of $\rho^*(b)$ to the neighborhood of the face. For this to be a local model for each such face, it suffices to verify  that the map on the right hand side of Equation \eqref{eqGlRetF}, restricted to the complement of the face $t=0$, is smooth as a map to $\cH^{sm}$, the space of smooth Hamiltonian data on the punctured sphere.  This follows from Equation \eqref{eqGlRetF} together with the fact that the functions $\hat{g}$ are smooth.

\begin{figure}[h]
  \centering
  \begin{tikzpicture}
     
    \filldraw[opacity=.25,red] (0,0) -- (0,4) -- (2,4) -- (2,0) -- cycle;
    \filldraw[opacity=.25,blue] (0,0) -- (4,0) -- (4,2) -- (0,2) -- cycle;
     \filldraw (0,0) circle (2pt);
     \draw[thick] (0,4) -- (0,0) -- (4,0);
     \node[label=below left:{$f_0$}] at (0,0) {};
     \node at (1,3) {$W_{f}$};
     \node at (1,1) {$W_{f_0}$};
       \node[label=above:{$x$},circle,fill,inner sep=1.5pt] at (1,2) {};
     \begin{scope}[shift={(-2,-1)}]
        \node[label=above:{$T_{f_0}$}] at (0,0) {};
       \draw (-1,1) -- (0,0) -- (1,1);
       \filldraw (-.5,.5) circle (3pt);
       \filldraw (0,0) circle (3pt);
       \filldraw (0,-.5) circle (3pt);
       \draw (0,0) -- (0,-1);
     \end{scope}
      \begin{scope}[shift={(-2,2)}]
        \node[label=above:{$T_{f}$}] at (0,0) {};
       \draw (-1,1) -- (0,0) -- (1,1);
    
       \filldraw (0,0) circle (3pt);
     
       \draw (0,0) -- (0,-1);
     \end{scope}
  \end{tikzpicture}
  \caption{Trees associated to the boundary of a cube}
  \label{fig:gluing_nearby_face}
\end{figure}
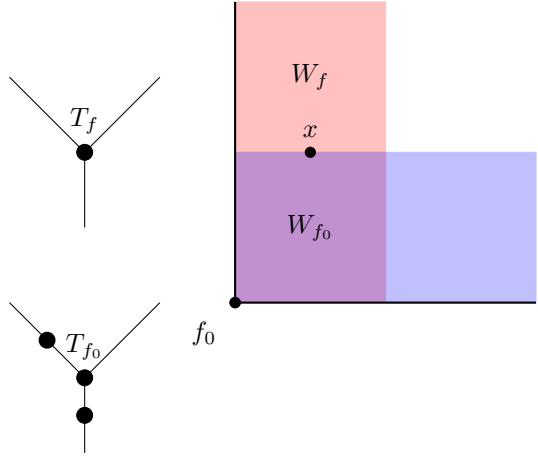

The gluing axiom and compatibility axioms are automatic by construction. We have thus constructed for each $n$-cube $\fd$ in $X_n$ an $n+1$-cube $\rho_n(\fd)\in PX_{n}=X_{n+1}$.

We now show that this map $\rho_{\bullet}:X_{\bullet}\to PX_{\bullet}$  is a  map of cubical sets. That is, it satisfies Equations \eqref{eqDeg1Mor} and \eqref{eqDeg1Mor2}: the compatibility of Equation \eqref{eqGlRetF} with taking face maps is a consequence of Equation \eqref{eqTildG}, together with the compatibility of the maps $\psi_{f,n}$ with restriction to faces which is provided by Equation \eqref{eqPOUComp}.  For compatibility with degeneracy maps, one verifies that whenever the data on $\fd$ is independent of a particular coordinate then so is the data defined by Equations \eqref{eqTildG} and \eqref{eqGlRetF}. Finally, note that Equations \eqref{eqTildG} and \eqref{eqGlRetF} are equivariant with respect to transpositions involving coordinates other than the one corresponding to the direction of the homotopy, proving the compatibility of the homotopy with symetries of the cube. 

Finally, we verify that $\rho_{\bullet}$ indeed satisfies Condition \eqref{eqCubDefRet}. Examining Equation \eqref{eqTildG} we have $\rho_{\bullet}(\fd)(0,x)=\fd(x)$, $\rho_{\bullet}(\fd)(1,x)\in Y_{\bullet}$ and, for any $\fd\in Y_{\bullet}$ we have $\rho_{\bullet}(\fd)(t,x)=\fd(x)$ for all $t\in[0,1]$. This proves the claim.

\end{proof}

We now complete this section with the proof of its main result:
\begin{proof}[Proof of Proposition \ref{lmFrgtHtpyEq}]
Consider the diagram
\begin{equation}
  \begin{tikzcd}
    \Box^{}_{\bullet}(\cH^{sm}_{0}(\vec{H},H^0))\ar[d] &\Box^{sm}_{\bullet}(\cH^{sm}_{0}(\vec{H},H^0))\ar[d]\ar[r] \ar[l] &\cH_{\bullet}(\vec{H},H^0)\ar[d,"\fF"]\\
        \Box^{}_{\bullet}(f\cM_{0,k+1}) &\Box^{sm}_{\bullet}(f\cM_{0,k+1})\ar[l] \ar[r]&f\Mbar^{\bR}_{0}(k)_{\bullet}.
  \end{tikzcd}
\end{equation}
That $\fF$ is a homotopy equivalence will follow from the fact that all other arrows in the diagram are homotopy equivalences. Indeed, the upper right horizontal arrow is a homotopy equivalence by Lemma \ref{lmsmincHoEq}. The left vertical one is a homotopy equivalence by Lemma \ref{lmForSmHoEq}, and the fact that a homotopy equivalence of topological spaces induces a homotopy of the associated symmetric cubical sets. The smooth approximation argument for families of functions and $1$-forms parametrised by a cube then implies that the two horizontal maps on the left are homotopy equivalences, hence so is the middle vertical map (the use of such approximation arguments to show that the inclusion of smooth chains into continuous chains is a homotopy equivalence goes all the way back to Eilenberg \cite{Eilenberg1947}).

The bottom right horizontal map factors as 
\begin{equation}
\Box^{sm}_{\bullet}(f\cM_{0,k+1})\to\Box_{\bullet}(f\cM_{0,k+1})\to \Box_{\bullet}(f\Mbar^{\bR}_{0}(k)) \equiv  f\Mbar^{\bR}_{0}(k)_{\bullet}.
\end{equation}
The map on the right of the last equation is induced from the inclusion of the interior of $f\Mbar^{\bR}_0(k)$. Since $f\Mbar^{\bR}_0(k)$ is a manifold with corners, this inclusion is a homotopy equivalence. The map on the left is a homotopy equivalence, again by the cubical analogue of \cite{Eilenberg1947}. 

\end{proof}

\section{The Floer functor}
\label{sec:floer-functor}

The purpose of this section is twofold. For arbitrary symplectic manifolds, we construct a multicategory $\cF$, lying over $\cH$, consisting of Floer data, i.e. Hamiltonian data with chosen almost complex structures, and prove:
\begin{prop}
  \label{prop:regular_multimorphisms_correct}
The projection map of the non-unital multicategories associated to $\cF $ and $\cH$ is a homotopy equivalence.
\end{prop}
The above result implies that the analogue of Proposition \ref{lmFrgtHtpyEq} holds for the category $\cF$, which we will use essentially in our later arguments.
\begin{rem}
  The reason for which the above statement is formulated for the underlying non-unital categories is because our construction will not incorporate any morphism between  $(J,H)$ and $(J',H')$, whenever $H = H'$, except for identities which we formally include.   In \cite[Part 3]{Abouzaid2022}, a larger category $\cF$ is constructed, which includes such morphisms, but the proof that the projection map remains a homotopy equivalence in this case is slightly more complicated.
\end{rem}

As in the introduction, let $\Bbbk$ be a commutative ring and denote by $\Lambda_{\geq 0}$ the Novikov ring
\[
\Lambda_{\geq 0}:=\left\{\sum_{i=0}^{\infty} a_iT^{\lambda_i}|a_i\in\Bbbk,\lambda_i\in\mathbb{R}_{\geq 0}\text{ strictly increasing and }\lim_{i\to\infty}\lambda_i=\infty\right\}.
\]
As in the strategy discussed around Equation \eqref{eq:Floer-multifunctor}, passing to the associated differential graded multicategory, we can now directly appeal to the first author's work on virtual fundamental chains:
\begin{prop}[Propositions 9.10 and 12.9 of \cite{Abouzaid2022}]
  If $\Bbbk$ is a characteristic $0$ field, there is a multi-functor
  \begin{equation}
        \Floer : C_* \cF \to \Ch_{\Lambda_{\geq 0}}
      \end{equation}
      with target the category of $\bZ/2$-graded chain complexes, considered as a multicategory by the tensor product of chain complexes, which assigns to every pair $(H,J)$ a chain complex whose underlying $\Lambda_{\geq 0}$ module is freely generated by rank-$1$ free modules associated to the time-$1$ Hamiltonian orbits of $H$, and whose differential is defined by a virtual count of solutions to Floer's equation. \qed
\end{prop}
We refer to the cited reference for the proof of the above result, though the reader will find, in Section \ref{sec:almost-compl-struct} and \ref{sec:moduli_spaces} below the basic geometric constructions of the moduli spaces which are involved.

Our second goal is to give a construction that is independent of the theory of virtual counts in the situation where virtual techniques can be avoided without an assumption on the base ring $\Bbbk$ (see Remark \ref{rem:semi-positivity} for a discussion of why we do not consider the semi-positive case):
\begin{prop} \label{prop:regular_moduli_for_monotone}
  If  $c_1(M)$ and $[\omega] $ are proportional on $\pi_2(M)$, with non-negative proportionality constant, then there is a multicategory $ \cF^{reg} \subset \cF$, so that (i) every object of $\cH$ admits a lift to $  \cF^{reg}$, (ii) for each input sequence of objects $\overrightarrow{(J,H)}$ and output $(J_0, H_0)$, the associated inclusion of multimorphism spaces
  \begin{equation}
\cF^{reg} (  \overrightarrow{(J,H)}, (J_0, H_0)) \to \cF (  \overrightarrow{(J,H)}, (J_0, H_0))   
  \end{equation}
  is a homotopy equivalence, and (iii) Floer theory defines a cubically enriched multifunctor
   \begin{equation}
        \Floer : C_* \cF^{reg} \to \Ch_{\Lambda_{\geq 0}}
      \end{equation}
      with target the category of $\bZ/2$-graded chain complexes over the Novikov ring. The differential is defined by a geometric count of solutions to Floer's equation.\end{prop}
\begin{rem}
In the special case where $c_1(M)$ vanishes, a trivialisation of the canonical bundle of $M$ determines a lift of the Floer multi-functor to the category of $\bZ$-graded complexes. As discussed in \cite{Abouzaid2013}, the condition that $c_1(M)$ be $2$-torsion is sufficient to define a lift of the Floer complexes to $\bZ$-graded chain complexes, but this lift is not in general compatible with operations.
\end{rem}

\subsection{Almost complex structures}
\label{sec:almost-compl-struct}

Let $\cJ$ denote the space of $S^1$-families of $\omega$-compatible almost complex structures on $M$. The set of objects of the multicategory $\cF$ is the product of $\cH$ with $\cJ$, i.e. it consists of a non-degenerate Hamiltonian together with a compatible family of almost complex structures.

Let $\overrightarrow{(H,J)}$ be a sequence $(H^1,J^1),(H^2, J^2), \ldots, (H^k,J^k)$ of objects of $\cF$. 
\begin{defin}\label{def:pre-multimorphism-F}
  A \emph{pre-multimorphism} in $\cF $ from $\overrightarrow{(H,J)}$  to $(H^0,J^0)$ consists of a pre-multimorphism from $\vec{H}$ to $H^0$ in the sense of Definition \ref{def:pre-multimorphism-H}, together with
  \begin{enumerate}
  \item a lift of the label $H^e$ of each edge of the underlying tree $T$ to a pair $(H^e, J^e)$,
    \item a map from $\Sigma$ to the space of tame almost complex structures on $M$, whose pullback under the strip-like end associated to every edge $e$ adjacent to a vertex $v$ agrees with $J^e$.
  \end{enumerate}
\end{defin}
As before, we define two pre-multimorphisms to be equivalent if there is a biholomorphism between the underlying pre-stable curves which intertwines both the Hamitonian and almost complex data. We write $\cF_{0}(\overrightarrow{(H,J)};(H^0,J^0))$ for the set of equivalence classes of such data and $\cF_{0,T}(\overrightarrow{(H,J)};(H^0,J^0)) $ for those with underlying labelled tree $T$. The $C^\infty$ topology on the space of almost complex structures equips this space with a natural topology for which the following result holds:
\begin{lem}
  The projection map
  \begin{equation}
   \cF_{0,T}(\overrightarrow{(H,J)};(H^0,J^0))  \to  \cH_{0,T}(H^1,\dots,H^k;H^0)
  \end{equation}
   is a fibration with contractible fibres. \qed
\end{lem}
As discussed in the proof of Proposition \ref{lmFrgtHtpyEq}, this implies that the resulting map of smooth singular cubes is a homotopy equivalence. This will only take us partly towards the desired result, because, as in Section \ref{sec:high-cubes-mult}, we shall not topologise the full space $\cF_{0}(\overrightarrow{(H,J)};(H^0,J^0)) $, but rather define higher $n$-cubes to be those which are smooth on each stratum, and are equipped with prescribed gluing data in neighbourhoods of all strata.

To proceed, we thus extend the gluing construction from Section \ref{sec:high-cubes-mult}: given  an element $\cF_{0,T}(\overrightarrow{(H,J)};(H^0,J^0)) $ and a choice of gluing parameters $\vec{r}\in [0,1)^{|E_{int}(T)|}$ indexed by the interior edges of $T$, we obtain a glued datum
\begin{equation}
 \Gamma_{\vec{r}}(\fd) \in \cF_{0}(\overrightarrow{(H,J)};(H^0,J^0))
\end{equation}
which lies in the stratum labelled by the tree obtained by collapsing every interior edge of $T$ with non-zero gluing parameter. In this way, we can define the notion of a \emph{Floer local model} by adding the complex structure data to Definition \ref{dfLocModMult}, i.e. this consists of an open manifold $V$, an open neighbourhood $U$ of the origin in the $k$-cube, and a tree $T$ labelled by Hamiltonians and almost complex structures, together with smooth maps
\begin{align}
    b \co V\times U & \to \cF_{0,T}(\overrightarrow{(H,J)},(H^0, J^0)) \\
   g \co V\times U & \to [0,1)^{|E_{int}(T)|},
\end{align}
where we require the gluing data to vanish on $V \times \{0\}$, and for each component to either identically vanish, or to be non-zero in the interior. The point is that this data determines a map
\begin{equation}
  \Gamma_{g}(b) \co  V \times U \to  \cF_{0}(\overrightarrow{(H,J)},(H^0, J^0)) 
\end{equation}
which is the Floer data associated to this local model.

\begin{defin}\label{dfCuMultMor-J-strict} 
  For each natural number $n$, the set $\cF_{n}(\overrightarrow{(H,J)},(H^0,J^0))$ consists of collections $(b_f,g_f)_{f\in F_n}$, indexed by the faces of the cube, of Floer local models for a corner.  We require that the Floer data obtained by gluing agree whenever they are defined, in the sense of Definition \ref{dfCuMultMor}.
\end{defin}

We omit the proof of the next result, which is a straightforward generalisation of the construction of the multicategory $\cH$:
\begin{prop}
  The collection of sets $\{\cF_{n}(\overrightarrow{(H,J)},(H^0,J^0))\}_{n=0}^{\infty}  $ admits face, degeneracy, and symmetry operations making them into the $n$-cubes of a symmetric cubical set $ \cF_{\bullet}(\overrightarrow{(H,J)},(H^0,J^0))$, which is equipped with a forgetful map to $ \cH_{\bullet}(\vec{H},H^0) $. 
\end{prop}

The next result is a consequence of the contractibility of the space of almost complex structures which are compatible relative to a given symplectic form:
\begin{lem} \label{lem:forget_J_acyclic-Kan}
  The maps from $ \cF_{\bullet}(\overrightarrow{(H,J)},(H^0,J^0))$ to $ \cH_{\bullet}(\vec{H},H^0) $ is an acyclic Kan fibration. 
\end{lem}
\begin{proof}[Sketch of proof]
  It suffices to show that, given a cube in $ \cH_{\bullet}(\vec{H},H^0)$ whose boundary is equipped with a lift to $ \cF_{\bullet}(\overrightarrow{(H,J)},(H^0,J^0))$, we may extend this lift to the entire cube. The gluing atlas of Hamiltonians determines a lift to a gluing atlas incorporating almost complex structures, which provides a lift to a neighbourhood of the boundary. Since the underlying topological type of the (pre-stable) Riemann surface in the interior of the cube is fixed, we may now directly use the contractibility of the space of tame almost complex structures.
\end{proof}

Combining Lemma \ref{lem:forget_J_acyclic-Kan} with Proposition \ref{lmFrgtHtpyEq} we conclude:
\begin{cor}
The forgetful map
\begin{equation}
\fF: \cF_{\bullet}(\overrightarrow{(H,J)},(H^0,J^0) ) \to f\Mbar^{\bR}_{0}(k)_{\bullet}
\end{equation}
is a homotopy equivalence. 
\end{cor}

Next, we extend the notation from Equation \eqref{eq:multi-composition-Ham}, and given  sequences $\overrightarrow{(H,J)}^1\in\cF^{k_1}$ and $\overrightarrow{(H,J)}^2\in \cF^{k_2}$, we write 
\begin{equation}
  \overrightarrow{(H,J)}^1 \circ_i \overrightarrow{(H,J)}^2  
\end{equation}
for the sequence with $k_1 + k_2 - 1$ pairs obtained by replacing the $i$th component of $\vec{H}^2$ with the vector $\vec{H}^1$. It is straightforward to lift the multicomposition maps from Equation \eqref{eqHMultComp} to maps
\begin{multline}\label{eqHMultComp2}
\circ_i:\cF_{n_1}(\overrightarrow{(H,J)}^1;(H,J)^{2,i})\otimes \cF_{n_2}(\overrightarrow{(H,J)}^2;(H,J)) \\ \to\cF_{n_1+n_2}(  \overrightarrow{(H,J)}_1 \circ_i \overrightarrow{(H,J)}_2  ;(H,J))
\end{multline}
which equips $\cF$ with the structure of a multicategory, which admits a forgetful multifunctor to $\cH$. 

\subsection{Moduli spaces}
\label{sec:moduli_spaces}

Associated to each cube $\fd$ in $ \cF_{n}(\overrightarrow{(H,J)},(H,J))$ is a topological space
\begin{equation}
  \label{eq:moduli_space}
  \Mbar(\fd)
\end{equation}
whose elements are equivalence classes of stable solutions to the Cauchy-Riemann equation $b_{\fd}(x)$ associated by $\fd$ to a point $x \in [0,1]^n$. Concretely, if $x$ lies in a neighbourhood of a stratum $f$, then $b_{\fd}(x)$, which is valued in $\cF_{0}(\overrightarrow{(H,J)},(H, J)) $, is obtained by gluing the data $b_f$ according the to the gluing parameters associated to this point by $x$. Part of the datum of an element of this set thus consists of a pre-stable Riemann surface $\Sigma_x$ with $k+1$-punctures. The solutions that we consider have domain pre-stable curves $\Sigma$ with $k+1$-punctures equipped with an embedding
\begin{equation}
\Sigma_x \to \Sigma
\end{equation}
whose complement is unstable, and which is compatible with the labelling of the punctures in the sense that the marked points labelled by $i$ in $\Sigma_x$ and $\Sigma$ are connected by a chains of components in the complement of $\Sigma_x$.  We partition the components of $\Sigma \setminus \Sigma_x$ into \emph{Floer cylinders}, which separate the marked point labelled by $i$ in $\Sigma_x$ from $\Sigma$, or the two sides of a node in $\Sigma_x$, and \emph{sphere bubbles}. We thus consider maps
\begin{equation}
  u \co \Sigma \to M
\end{equation}
satisfying the Cauchy-Riemann equation
\begin{equation}
  \left(du - X_{H_\Sigma} \otimes \alpha_{\Sigma}\right)^{0,1} = 0
\end{equation}
where the function $H_{\Sigma}$, the $1$-form $\alpha_\Sigma$, and the almost complex structure are determined as follows:
\begin{itemize}
\item On a component of $\Sigma_x$, we use the data determined by $\fd$
\item On a Floer cylinder, we use the data determined by the associated node of $\Sigma_x$ (i.e. with respect to the identification with $\bR \times S^1$ which is canonical up to translation, we set $H_{\Sigma} = H_e$ for the edge $e$ labelling this node, the $1$-form $\alpha_\Sigma$ to be given by $dt$, and the almost complex structure by $J_e$).
  \item On a sphere bubble, the function $H_{\Sigma}$ identically vanishes (as does $\alpha_\Sigma$), and the almost complex structure is given by the almost complex structure on $M$ associated by $\fd$ to the point in $\Sigma_x$ to which this bubble is attached (or to the corresponding points in $S^1$ if the bubble is attached to a Floer component).
  \end{itemize}
We associate to each solution of this equation the topological energy (c.f. \cite[Lemma 8.1.6]{McDuffSalamon2012})
\begin{align}
  E(u) & = \int u^* \omega + \tilde{u}^* dH \wedge \alpha \\
  & = \int |du - X_H \otimes \alpha|^2 d \mathrm{vol}_{\Sigma} + d H_{u(z)} \wedge \alpha.
\end{align}
  
The moduli space $  \Mbar(\fd) $ is then defined to be the space of such finite energy maps, with the property that the asymptotic conditions on the two sides of each node agree, modulo the equivalence relation which identifies two maps that are intertwined by a biholomorphism over $M$ preserving the Floer data. We have the following consequence of Gromov compactness:
\begin{lem}
 The energy functional is proper on $  \Mbar(\fd) $.
\end{lem}
\begin{proof}[Sketch of proof:]
  We must show that a sequence of solutions of bounded energy has a subsequence which admits limit. The projection to the base $[0,1]^n$ of the family admits such a subsequence. By construction, the domains of these subsequences converge in the Gromov sense (namely, they admit subdomains, whose complements are thin annuli, on which the conformal structure converges), as does the Cauchy-Riemann equation that we consider (by the gluing construction). The result is now a standard application of Gromov compactness for families of Riemann surfaces.
\end{proof}

\subsection{Regular Floer data}
\label{sec:regular_data}
We begin by formulating a notion of regularity for objects of $\cF$, in terms of the union  $\cM(J)$ of the (uncompactified) moduli spaces $\cM(J_t)$ of $J_t$-holomorphic spheres (for $t \in S^1$) as well as the (uncompactified) moduli space $\cM(J,H)$ of solutions to Floer's equation
\begin{equation}
 j \circ \left( du - X_{H_t} \otimes dt \right) =   \left( du - X_{H_t} \otimes dt \right) \circ J_t
\end{equation}
on the cylinder. We write $\cM_1(J)$ and $ \cM_1(J,H) $ for the corresponding moduli spaces with one marked point.
\begin{defin}
  A pair $(J,H)$ is \emph{regular} if the following conditions hold:
  \begin{enumerate}
  \item All elements of $\cM(J)$ which are represented by simple pseudo-holomorphic spheres of vanishing Chern number are regular.
  \item All element of $\cM(J,H)$ of virtual dimension strictly smaller than $2$ are regular.
  \item The restriction of the evaluation maps
    \begin{equation}
           \cM_1(J) \to S^1 \times  M \leftarrow \cM_1(J,H)
         \end{equation}
to the loci of simple spheres of vanishing Chern number and Floer cylinders of   virtual dimension strictly smaller than $2$ are transverse (and in particular are disjoint).
  \end{enumerate}
\end{defin}

In order to use the above notion, it is important to consider what it implies for the space $\Mbar(J,H)$ of \emph{stable} solutions to Floer's equation, in the context where the symplectic form and the first Chern class are non-negatively proportional:
\begin{lem} \label{lem:regular_implies_regular_stable}
If a pair $(J,H)$ is \emph{regular}, then all stable solutions to Floer's equations of virtual dimension strictly smaller than $2$ are regular (i.e. the associated linearised operator is surjective).
\end{lem}
\begin{proof}
  It suffices to show that these conditions imply that the \emph{stable} solutions to Floer's equation of virtual dimension strictly smaller than $2$ are given as follows: those of virtual dimension $0$ have domain a cylinder, and those of virtual dimension $1$ have domain either a cylinder or a pair of virtual dimension $0$ cylinders.

  To see this, it is convenient to temporarily formulate our discussion in terms of the \emph{Fredholm index} of a stable solution to Floer's equation, which is one more than the virtual dimension, and has the advantage that it is additive over the component, with each  sphere bubble contributing twice its Chern number. Since our assumptions include the requirement that the moduli spaces of solutions to Floer's equation which have strictly negative virtual dimension are empty, and there are no spheres of negative Chern number, the contribution of each component to the Fredholm index is non-negative, and only those spheres of trivial Chern class contribute trivially. This means that a stable solution of virtual dimension strictly less than $2$ cannot carry any sphere bubble of strictly positive Chern number, and that all its Floer cylinder components have virtual dimension either $0$ or $1$. Finally, we use our assumption that the cycles swept by cylinders and simple Chern $0$ spheres in $S^1 \times X$ are transverse (hence disjoint) to conclude that the remaining possibility (of a cylinder carrying a sphere bubble) is excluded, since every Chern $0$ sphere is a multiple cover of a simple one.

\end{proof}
\begin{rem}
  It may appear more natural to use the conclusion of the above Lemma as a definition, but that would make later constructions more complicated. The essential point is that one can prove that regularity for stable moduli spaces of virtual dimension strictly less than $2$ is equivalent to the regularity of the corresponding moduli spaces of smooth cylinders, with the additional condition that the cycle they sweep be disjoint from all pseudo-holomorphic spheres of Chern number $0$. The issue is that, in this characterisation, no reason for the disjunction is given, so it becomes difficult to work with it in families. 
\end{rem}

We now extend this notion to multimorphisms in $\cF$: the linearisation of the Cauchy-Riemann operator satisfied by each point in $ \Mbar(\fd) $ is a Fredholm map
\begin{equation} \label{eq:linearisation_CR}
C^{\infty}(\Sigma; u^* TM)/ \mathfrak{aut} \, \Sigma \to \Omega^{0,1}(\Sigma; u^* TM),
\end{equation}
where $ \mathfrak{aut}\, \Sigma$ is the tangent space of the space of biholomorphisms of $\Sigma$ which preserve the Floer data on $\Sigma$ (i.e. preserve the equivalence class of this point in $ \Mbar(\fd) $). Assuming that the corresponding point in $[0,1]^n$ lies in the interior of a face $[0,1]^k$, the deformation of the operator associated to moving the point within this stratum defines a map
\begin{equation} \label{eq:normal_deformation_CR}
 T [0,1]^k \to \Omega^{0,1}(\Sigma; u^* TM).
\end{equation}
We call the sum of the operators in Equation \ref{eq:linearisation_CR} and \ref{eq:normal_deformation_CR} the \emph{extended linearisation} operator.

In analogy with the situation for Floer cylinders, we associate to each cube $\fd$ a pair of evaluation maps
\begin{equation}
\Mbar_{1}(\fd) \to   X   \leftarrow \Mbar_{1}(J_{\fd})
\end{equation}
where the left hand side is the moduli space of solutions to Floer's equation, with $1$-marked point, and the right hand side is the moduli space of pseudo-holomorphic spheres, parametrised by the curves underlying $\fd$, also with one marked point.  
\begin{defin}
  The family of Floer data $\fd$ is \emph{regular} if the following properties hold:
  \begin{enumerate}
  \item  the extended linearisation operation is surjective for all elements of $\Mbar(J_{\fd})$ which are simple pseudo-holomorphic spheres of Chern class $0$,
  \item  the extended linearisation operator is surjective for all elements of $\Mbar(\fd) $, whose underlying curve does not contain any Floer cylinder or sphere bubble, and for which this index is strictly smaller than $2$, and
    \item the evaluation maps from these two parametrised spaces to the product of $X$ with the universal curve are transverse (and hence disjoint). 
  \end{enumerate}
\end{defin}
Repeating the argument of Lemma \ref{lem:regular_implies_regular_stable} for families, we have:
\begin{lem}
  If the family of Floer data $\fd$ are regular, then the extended linearisation operator is surjective for each element of $ \Mbar(\fd) $ whose virtual dimension is strictly less than $2$. \qed
\end{lem}

We note that regularity is closed under taking boundaries, degeneracies, and symmetries, as well as multicompositions (products). This leads to the following:
\begin{defin}
  The multicategory of regular Floer data $\cF^{\mathrm{reg}}$ is the sub-multicategory of $\cF$ with objects regular pairs $(J,H)$ and morphisms regular Floer data.
\end{defin}

\subsection{Contractibility of the space of regular data}

In the treatments of Floer theory using perturbations, one assumes that the Floer data at a corner of the parameter space, together with a choice of gluing parameter, determine the data in some neighbourhood. This is sufficient to achieve transversality in the monotone or exact setting because these assumptions ensure compactness of the moduli spaces of solutions to pseudo-holomorphic curve equations of prescribed dimension, so that a gluing construction, relying on the implicit function theorem, allows one to deduce regularity of all solutions in a neighbourhood from the regularity of solutions on the boundary sratum. In general, however, the moduli spaces of pseudo-holomorphic curves are not compact, as compactness only holds after bounding the energy. Without a refinement of the existing methods, we thus cannot expect the existence of a gluing parameter for which the glued solutions are regular.

However, the gluing atlas setup that we are using is more flexible than this naive approach since it allows us, as in  \cite{Abouzaid2010}, to perturb the data in a neighbourhood of every corner. Indeed, we allow the function $b_f$ to vary with respect to the normal direction, so that we have a dense open set of choices for which the subset of $ \Mbar(\fd) $ consisting of elements of bounded energy is regular. Exhausting $\bR$ as a countable union of bounded above subsets, we conclude that the set of choices for which regularity holds is a countable intersection of dense open sets. Since we only need to perturb the almost complex structure in order to achieve transversality, we conclude that every object of $\cH$ lifts to $\cF^{\mathrm{reg}}$. Going further, we have:

\begin{lem}
  The inclusion of the multicategory of regular Floer data $\cF^{\mathrm{reg}}$ in $\cF$ induces an equivalence of multimorphism spaces:
\end{lem}
\begin{proof}
We construct a deformation retraction: for each cube $\fd$ in $ \cF  $, we choose a regular cube $\fd'$ with the same underlying Hamiltonian data, so that $\fd = \fd'$ if $\fd$ is regular, and a homotopy between them, compatibly with face and degeneracy maps. The construction is inductive on the dimension of the cube, and relies essentially on genericity of regular data, so that the data $b'_f$ underluing $\fd'$ can be chosen to be an arbitrarily small perturbation of the data $b_f$ underlying $\fd$, yielding a canonical path between them up to contractible choice, and provides the desired deformation retraction.
\end{proof}


We complete this section with a proof of its main result:

\begin{proof}[Proof of Proposition \ref{prop:regular_moduli_for_monotone}]
The construction of the Floer functor now follows from the existence of coherent orientations in Floer theory. More precisely, Floer and Hofer \cite{FloerHofer1993} treated the case of the differential on the Floer complex. In our setting, we count each solution $u$ to Floer's equation by its (signed) count, weighted by a factor $T^{E(u)}$, where $E(u)$ is the topological energy of $u$, which is defined as the integral
\begin{align}
  E(u) & = \int u^* \omega + \tilde{u}^*dH\wedge dt \\
  & = \int |du - X_H \otimes dt|^2 ds\wedge dt
\end{align}
where $\tilde{u}$ is the graph of $u$. Explicitly, the differential then takes the form

\begin{equation}
  [\mathrm{o}_p] \mapsto \sum_{q} \sum_{u \in \Mbar^{0}(p,q)} (-1)^{\mathrm{sign}(u)} T^{E(u)} [\mathrm{o}_q],  
\end{equation}
where $[\mathrm{o}_p]$ and $[\mathrm{o}_q]$ are generators of the free modules associated by index theory to orbits $p$ and $q$,  and  $\Mbar^{0}(p,q)$ is the subset of $\Mbar(H,J) $ consisting of those (equivalence classes of) solution with asymptotic conditions given by the orbits $p$ and $q$, which in addition have the property that the virtual dimension vanishes.

The chain map associated to each $0$-cube and $1$-cube of the multimorphism spaces is defined in Symplectic Field theory \cite{Bourgeois-Mohnke2004}, and adapted to the Hamiltonian setting in \cite{Ritter2013}.  the signed contribution of each solution is weighted by $T^{E(u)}$, which the second formula together with Equation \eqref{eq:monotonicity_multimorphism} show lies in the Novikov ring (the first formula shows that this is a topological invariant, which ensures that we indeed obtain a chain map).

The operations associated to $1$-cubes appear in the above references as chain homotopies establishing that the resulting homology-level operations are independent of choice, and rely on the fact that the corresponding parameterised moduli spaces interval acquire coherent orientations, \emph{relative an orientation of the parameter space}. The same argument thus applies to associate to each $n$-cube $\delta$ a map of degree $-n$, obtained from the components of $\Mbar(\delta)$ of virtual dimension $0$ (which are regular by our assumptions). The fact that this map depends on a choice of orientation on the cube implies its compatibility with symmetries, and its compatibility with compositions follows by the same argument as in \cite{FloerHofer1993}. Finally, considering those moduli spaces $\Mbar(\delta)$ of virtual dimension $1$, we conclude the compatibility with face maps by observing that the boundary of these moduli spaces are either associated to the faces of $\delta$ or to breaking of Floer trajectories. 

\end{proof}

\section{The Floer algebra of a compact subset}
\label{sec:floer-algebra}

The purpose of this section is to prove Theorems \ref{thm:main_thm} and \ref{thm:linear_comparison_from_CF}: we thus associate to each compact subset $K$ of $M$ a chain complex $SC^*_{M,f\Mbar^{\bR}_0}(K)$, equip it with the structure of an algebra over the operad $C_*(f\Mbar^{\bR}_0)$ of symmetric normalised cubical chains (see Appendix \ref{sec:symm-norm-cubic}) on the moduli spaces of framed genus-$0$ stable Riemann surfaces, and prove the properties listed in the introduction. We particularly refer to Section \ref{sec:relat-sympl-cohom} for a review of symplectic cohomology with support constraints.

To formulate the construction, we write $  \cF_K $ for the full sub-multicategory of $\cF$ whose objects are pairs $(H, J)$ with $H$ negative on $K$.

\begin{rem}
  The reader who prefers to avoid virtual methods may substitute the multicategory of regular Floer data $\cF^{reg}$ for $\cF$ in the above paragraph under the assumption that $M$ satisfies the hypotheses of Proposition \ref{prop:regular_moduli_for_monotone}, and obtain the corresponding multicategory  $\cF^{reg}_{K}$ associated to a compact subset $K$ of $M$. None of the arguments that we will give in the remainder of the paper will depend on whether one uses $  \cF_K $  or  $\cF^{reg}_{K}$.
\end{rem}

By restricting the forgetful functor $\pi:\cF\to f\Mbar^{\bR}_0$, as well as the Floer multi-functor $\Floer^*:\cF\to \Ch_{\Lambda_{\geq 0}}$,  we obtain a diagram of cubically enriched multifunctors

\begin{equation} 
  \begin{tikzcd}
    \cF_K \ar[r,"\Floer^*"] \ar[d,"\pi"] & \Ch_{\Lambda_{\geq 0}} \\
    f\Mbar^{\bR}_0 . & 
  \end{tikzcd}
\end{equation}
\begin{rem}
  In complete generality, the ring $\Lambda_{\geq 0}$ will be the Novikov ring over a ground field of characteristic $0$, and the multicategory $\Ch_{\Lambda_{\geq 0}}  $ denotes the multicategory of $\bZ/2$-graded complexes. As discussed earlier, in the Calabi-Yau case, one can work instead with $\bZ$-graded complexes, and, under the more general hypotheses of Proposition \ref{prop:regular_moduli_for_monotone}, one can assume that the ground ring is given by the integers. None of our arguments are sensitive to this difference, so, in this regard, we shall keep the notation ambiguous in what follows.
\end{rem}

Applying the functor of symmetric normalised cubical chains from Appendix \ref{sec:symm-norm-cubic}, we wish to fill in an arrow $C_*(f\Mbar^{\bR}_0)\to \Ch_{\Lambda_{\geq 0}}$ in a `universal' way. The general framework for doing this is a homotopical version of the \emph{operadic Kan extension}, and $SC^*_{M,f\Mbar^{\bR}_0}(K)$ will be defined as its completion.

We shall give an explicit construction of the operadic Kan extension which is essentially an unwinding of the standard bar construction involving the free-forgetful adjunction, see for example Section 13.3 of \cite{Fresse2009book} in the operad case. It generalizes the left homotopy Kan extension of modules over dg-categories, e.g. Section 5 of \cite{HollenderVogt1997}. We prove Theorem \ref{thm:main_thm} in Section \ref{sec:proof-theorem-main} using only the explicit construction, but then provide the more abstract formulations in Sections \ref{sec:diff-grad-mult}--\ref{sec-ext} in preparation for the proof of Theorem \ref{thm:linear_comparison_from_CF} in Section \ref{sec:proof-theorem-main-2}.

\subsection{The chain complex}
\label{sec:chain-complex}
We start by providing an explicit description of the chain complex $SC_{M,f\Mbar^{\bR}_0}^*(K)$. For this we introduce some definitions. A \emph{leveled tree of height $n$} is a collection $V_0,\dots,V_{n+1}$ of finite sets together with surjections $f_i:V_{i}\to V_{i-1}$. We take $V_0$ to consist of one element, which we call the root. We call $V_{n+1}$ the set of leaves. The associated abstract tree has vertices the union of the sets $V_i$ for $0 \leq i \leq n+1$, and has edges which connect each element $v_i \in V_i$ to its image under $f_i$. An \emph{isomorphism} of leveled trees consists of bijections between the $i$th level vertices commuting with the defining surjections. 

\begin{defin}
A  \emph{decorated leveled tree} is a leveled tree together with the assignment of an object of $\cF_K$ for each edge. 
We write $\cT_n(K)$ for the set of trees with $n$-levels.
\end{defin}

\begin{figure}[h]
  \centering
   \begin{tikzpicture}[node distance=1.5cm]
    \tikzstyle{box} = [circle,text centered, draw=black];
    \node (O) [box] {$\cO(3)$};
    
    \node (M11) [box, above of=O, xshift=-3cm] {$\cM(2)$};
    \draw (O) -- (M11);
    \node (M12) [box, above of=O] {$\cM(1)$};
    \draw (O) -- (M12);
    \node (M13) [box, above of=O, xshift=3cm] {$\cM(3)$};
    \draw (O) -- (M13);

    \node (M21) [box, above of=M11, xshift=-1.5cm] {$\cM(2)$};
    \draw (M11) -- (M21);
    \node (M22) [box, above of=M11] {$\cM(1)$};
    \draw (M11) -- (M22);
    \node (M23) [box, above of=M12] {$\cM(2)$};
    \draw (M12) -- (M23);
    \node (M24) [box, above of=M13, xshift=-1.5cm] {$\cM(1)$};
    \draw (M13) -- (M24);
    \node (M25) [box, above of=M13] {$\cM(2)$};
    \draw (M13) -- (M25);
    \node (M26) [box, above of=M13, xshift=3cm] {$\cM(1)$};
    \draw (M13) -- (M26);

     \node (A1) [box, above of=M21, xshift=-1.5cm] {$A$};
    \draw (A1) -- (M21);
    \node (A2) [box, above of=M21] {$A$};
    \draw (A2) -- (M21);
    \node (A3) [box, above of=M22] {$A$};
    \draw (A3) -- (M22);
    \node (A4) [box, above of=M23, xshift=-1.5cm] {$A$};
    \draw (A4) -- (M23);
    \node (A5) [box, above of=M23] {$A$};
    \draw (A5) -- (M23);
    \node (A6) [box, above of=M24] {$A$};
    \draw (A6) -- (M24);
\node (A7) [box, above of=M25] {$A$};
\draw (A7) -- (M25);
\node (A8) [box, above of=M25, xshift=1.5cm] {$A$};
\draw (A8) -- (M25);
\node (A9) [box, above of=M26] {$A$};
    \draw (A9) -- (M26);
\end{tikzpicture}
  \caption{A tree of height $2$, labelled by an element of the operadic bar construction, where $A$ stands for the algebra $\Floer$, $\cM$ for the multicategory $C_* \cF_K$, and $\cO$ stands for the operad $C_* f\Mbar^{\bR}_0$.}
\end{figure}


Let $T=V_{n+1}\to \dots\to V_0$ be a decorated leveled tree. We associate to each vertex $v$ of $T$ above the root a chain complex  $C_v$ 
defined by downward induction along the levels in the following way. To each leaf labelled by an object $(H,J)$, we associate the chain complex $\Floer^*(H,J)$.  Suppose inductively that we have associated to each vertex $v$ at the level $k+1$ a chain complex $C_v$, with the base case as above.   Given a vertex $w$ at the level $k$ and an ordering $o$ on the incoming edges, let $(v_1,\dots,v_m)$ be the corresponding set of vertices and $\overrightarrow{(H,J)}_o=((H_1,J_1),\ldots,  (H_m, J_m))$ be the corresponding tuple of Hamiltonians. Let $O_w$ be the set of orderings of the incoming edges, 
and define
\begin{equation} 
C_w:= \left(\bigoplus_{o\in O_w} C_{v_1}\otimes\cdots \otimes C_{v_m}  \otimes C_*(\cF_K(\overrightarrow{(H,J)}_o; (H_w,J_w))) \right)_{S_m}.
\end{equation}

Here the subscript $S_m$ means that we are taking coinvariants (the quotient under the subcomplex generated by elements of the form $x- \sigma\cdot x$ for $\sigma \in S_m$),  under the action 
defined by the maps 
\begin{multline}
C_{v_1}\otimes\ldots \otimes C_{v_n} \otimes C_*(\cF_K(\overrightarrow{(H,J)}_o; (H_w,J_w)))  \to  \\
  C_*(\cF_K ( \sigma \cdot \overrightarrow{(H,J)}_o; (H_w,J_w))) \otimes C_{\sigma(v_1)}\otimes\ldots \otimes C_{\sigma(v_n)}.
\end{multline}
We now associated to the tree $T$ the complex
\begin{equation}\label{eqRootSC}
\Floer^*_{T} :=\left(\bigoplus_{o\in O_{root}} C_{v_1}\otimes\ldots \otimes C_{v_k}\otimes  C_* f\Mbar^{\bR}_0(k) \right)_{S_k},
\end{equation}
where we assume that the root has $k$ incoming edges. Taking the direct sum of these complexes over all trees of height $n$, we obtain the direct sum
\begin{equation} \label{eq:n-simplices-simplicial-model-SC}
\Floer^*_{M,f\Mbar^{\bR}_0,n}(K) \equiv     \bigoplus_{T\in\cT_n(K)} \Floer^*_{T},
\end{equation}
which is equipped with the differential $d_{int}$, which is the sum of the differentials $d_T$ for each complex $\Floer^*_T$.

For each integer $i$ between $0$ and $n$, we construct a chain map \begin{equation} \label{eq:face-maps-simplicial-SC}
  d_i:  \Floer^*_{M,f\Mbar^{\bR}_0,n}(K)  \to  \Floer^*_{M,f\Mbar^{\bR}_0,n-1}(K)
\end{equation}
as follows: the map $d_0$ collapses the edges beween the root and the first level  by projecting from the multicategory $\cF$ to the operad $f\Mbar^{\bR}_0 $ and applying operadic composition. For   $0<i<n $, $d_i$ is given by collapsing the edges between the level $i$ and $i+1$, and  applying composition in the indexing category $\cF$. Finally, $d_n$ is given by collapsing the edges between the $n$th level and the leaves, and applying the Floer functor.  
\begin{defin} \label{def:operadic_relative_chains}
  The \emph{operadic symplectic cochains} with support $K$ is the degreewise completion, with respect to the valuation of the Novikov ring, of the direct sum of the shift by $n$ of the complexes associated to trees of level $n$,
  \begin{equation}
SC^*_{M,f\Mbar^{\bR}_0}(K):=\widehat{\bigoplus_n} \Floer^*_{M,f\Mbar^{\bR}_0,n}(K)[-n],
\end{equation}
equipped with the differential \begin{equation} \label{eq:total_differential}
d=d_{int}+\sum_{i=0}^{n}(-1)^id_i.
\end{equation} 
\end{defin}

Having given an explicit definition of our chain complex, we now give a slightly more abstract description, using simplicial methods, which will be useful when describing the algebraic structure.
We start by recalling that a \emph{simplicial chain complex} $A^*_\bullet$ is a contravariant functor from the simplex category to the category of chain complexes. Explicitly, this amounts to the assignment of a chain complex $A^*_n$ to each natural number $n$, together with chain maps corresponding to the face and degeneracy maps, that satisfy the simplicial  identities. The \emph{geometric realization}\footnote{Here we are slightly abusing terminology, as it would be more appropriate to call this the geometric realization of the associated semisimplicial chain complex. We do this is only for convenience. We could have used the normalized chain complex as in \cite{Smirnov2001}[Section 3.1] as well. The equivalence of the constructions can be shown using the discussion in \cite{goerss}[Section III.2]. We will keep abusing terminology the same way in the rest of the paper.} functor takes each simplicial chain complex to an ordinary chain complex, by taking the direct sum of the shift by $n$ of $A^*_n $, equipped with the differential from Equation \eqref{eq:total_differential}.

Now, the operations in Equation \eqref{eq:face-maps-simplicial-SC} are exactly the face maps of a simplicial  complex $\Floer^*_{M,f\Mbar^{\bR}_0,\bullet}(K)$. The degeneracy  maps $s_i$ correspond to replacing each vertex at the $i$th leaf by a pair of vertices labeled by the same Hamiltonian connected by an edge labeled by the identity. From Definition \ref{def:operadic_relative_chains}, we have:
\begin{lem}
  The complex of operadic symplectic cochains with support $K$ is isomorphic to the completion of the  geometric realization of $\Floer^*_{M,f\Mbar^{\bR}_0,\bullet}(K)$. \qed
\end{lem}

\subsection{The algebra structure}
As indicated above, we shall use simplicial methods in order to describe the algebra structure. The main benefit is that we thus avoid getting mired in formulae.

The key definition is that of a simplicial $C_* f\Mbar^{\bR}_0$-algebra, which is a contravariant functor from the simplex category to the category of $C_* f\Mbar^{\bR}_0$-algebras.
\begin{lem}
The simplicial complex $\Floer^*_{M,f\Mbar^{\bR}_0,\bullet}(K)$ is the underlying complex of a simplicial $C_* f\Mbar^{\bR}_0$-algebra.
\end{lem}
\begin{proof}
 
For any $n$ and any $m\geq 1$ we construct a  chain map
\begin{equation}\label{eqSCAlgStr}
\left(\Floer^*_{M,f\Mbar^{\bR}_0,n}(K)  \right)^{\otimes m} \otimes C_* f\Mbar^{\bR}_0(m)\to \Floer^*_{M,f\Mbar^{\bR}_0,n}(K) 
\end{equation}
as follows. Recall from Equation \eqref{eqRootSC} that $\Floer^*_{M,f\Mbar^{\bR}_0,n}(K) $ consists of summands $CF_{T}^*$ for trees of height $n$ obtained by taking coinvariants of the sum
\begin{equation}
\widetilde{CF}^*_T:=\bigoplus_{o\in O_{root}}   C_{v_1}\otimes\ldots \otimes C_{v_k} \otimes C_* f\Mbar^{\bR}_0(k)
\end{equation} 
Taking the direct sum of these chain complexes, we define
\begin{equation} \label{eq:lift_of_n-simplices}
  \widetilde{CF}^*_{M,f\Mbar^{\bR}_0,n}(K) =\bigoplus_{T\in\cT_n(K)}\widetilde{CF}^*_T.
\end{equation}
There is an obvious chain map 
\begin{equation}\label{eqwSCAlgStr}
\left( \widetilde{CF}^*_{M,f\Mbar^{\bR}_0,n}(K) \right)^{\otimes m} \otimes C_* f\Mbar^{\bR}_0(m) \to  \widetilde{CF}^*_{M,f\Mbar^{\bR}_0,n}(K),
\end{equation}
by applying the operadic compositions.  To obtain a map as in Equation \eqref{eqSCAlgStr} we write $\widetilde{CF}^*_{M,f\Mbar^{\bR}_0,n,m}(K) $ for the direct sum in  the right hand side of Equation \eqref{eq:lift_of_n-simplices}, associated to trees with $m$ incoming edges at the root, and obtain a direct sum decomposition
\begin{equation}
 \widetilde{CF}^*_{M,f\Mbar^{\bR}_0,n}(K) =\bigoplus_m  \widetilde{CF}^*_{M,f\Mbar^{\bR}_0,n,m}(K).
\end{equation}
Then the map of Equation \eqref{eqwSCAlgStr} splits as a direct sum of maps
\begin{multline} \label{eq:action_operad-Kan-extension}
\left(\widetilde{CF}^*_{M,f\Mbar^{\bR}_0,n,k_1}(K) \otimes \cdots\otimes \widetilde{CF}^*_{M,f\Mbar^{\bR}_0,n,k_m}(K)  \right)  \otimes  C_* f\Mbar^{\bR}_0(m)\\ \to \widetilde{CF}^*_{M,f\Mbar^{\bR}_0,n,\sum_{i=1}^m k_i }(K),
\end{multline} 
which are equivariant with respect to the action of $S_{k_1}\times\dots \times S_{k_m}$ which on the right is induced by the inclusion into $S_{\sum_{i=1}^mk_i}$. 
We thus take the map in Equation \eqref{eqSCAlgStr} to be the induced one on the quotient.

We leave the verification of the compatibility of this construction with the face structure to the reader with the following indication: for $i=0$, we use the fact that the projection map from $\cF$ to $f\Mbar^{\bR}_0 $ is a map of multicategories, for $0 < i < n$, we use the associativity of the operations on   $\cF$, while for $i =n$, we use the fact that $\Floer$ is an algebra. In all cases, the fact that operations are strictly compatible with the action of the symmetric group follows from the equivariance conditions in the definition of multicategories and algebras over them.
\end{proof}

The geometric realisation functor is (lax) symmetric monoidal, via the Eilenberg-Zilber shuffle maps \cite{EilenbergZilver1953} that combinatorially encode the standard simplicial subdivision of each prism $\Delta^p \times \Delta^q$ into a union of simplices $\Delta^{p+q}$. This map immediately gives an algebra structure on the geometric realisation of each simplicial algebra.   The details are given below in a more general context in Section \ref{sec-geo-real} below.

\subsection{Proof of Theorem \ref{thm:main_thm}}
\label{sec:proof-theorem-main}

The proof of the first main result asserted in the introduction is now a completely straightforward consequence of the construction:
\begin{proof}[Proof of Theorem \ref{thm:main_thm}]
      By construction, the multicategories $\cF_K$ are sub-multicategories of $\cF$, and thus each inclusion $K \subset K'$ gives rise to an inclusion $\cF_{K'} \subset \cF_K$, and hence an inclusion of sets of decorated levelled trees $\cT_n(K') \subset \cT_{n}(K)$, which itself gives an inclusion of complexes in Equation \eqref{eq:n-simplices-simplicial-model-SC}
  \begin{equation} \label{eq:inclusion-simplicial-cochains-K-K'}
        \Floer^*_{M,f\Mbar^{\bR}_0,n}(K') \subset \Floer^*_{M,f\Mbar^{\bR}_0,n}(K).
      \end{equation}
      Since the inclusions of sets of decorated levelled trees are compatible with the face maps associated to collapsing levels, we obtain a map of simplicial chain complexes
   \begin{equation}
        \Floer^*_{M,f\Mbar^{\bR}_0,\bullet}(K') \subset \Floer^*_{M,f\Mbar^{\bR}_0,\bullet}(K),
      \end{equation}
      which defines the desired map of operadic symplectic cochains with support $K$ after completions.

      Given that the inclusion in Equation \eqref{eq:inclusion-simplicial-cochains-K-K'} is given by inclusion of trees, it lifts to an inclusion
      \begin{equation}
\widetilde{CF}^*_{M,f\Mbar^{\bR}_0,n}(K') \to         \widetilde{CF}^*_{M,f\Mbar^{\bR}_0,n}(K)         
\end{equation}
of the complexes whose coinvariants give rise to the operadic symplectic cochains. Inspecting Equation \eqref{eq:action_operad-Kan-extension}, it follows that the restriction map is compatible with the operadic action.

Since the strict compatibility of maps associated to a triple inclusion $K \subset K' \subset K''$ is immediate from the inclusion of multicategories $  \cF_{K''} \subset \cF_{K'} \subset \cF_{K}$, it remains to identify the action of the symplectomorphism group. Every symplectomorphism $\psi$ of $M$ induces a multifunctor
\begin{equation}
  \psi_* \co   \cF \to \cF,
\end{equation}
which at the level of objects and multimorphisms is given by pre-composing the Hamiltonian data with $\psi$, and conjugating the almost complex data by it. This action restricts to a multifunctor from $\cF_K$ to $\cF_{\psi K}$, with inverse the multifunctor associated to $\psi^{-1}$. The induced map yields the isomorphism in Equation \eqref{eq:action_symplectomorphism}, which is compatible with the operadic action by the functoriality of our construction of the structure maps.

Since $(\psi \circ \phi)_* = \psi_* \circ \phi_*$ as multifunctors on $\cF$, the action is eminently compatible with composition, and it is compatible with restriction map because $\psi_*$ preserves inclusions, in the sense of inducing, for each inclusion $K \subset K'$, a commutative diagram
\begin{equation}
  \begin{tikzcd}
    \cF_{K'} \ar[r, "\subset"] \ar[d, "\psi_*"] & \cF_{K} \ar[d, "\psi_*"] \\
     \cF_{\psi K'} \ar[r, "\subset"] & \cF_{\psi K}.
  \end{tikzcd}
\end{equation}
\end{proof}

\subsection{Differential graded multicategories}
\label{sec:diff-grad-mult}

At this stage, the reader can provide an explicit proof of Theorem \ref{thm:linear_comparison_from_CF}, along the lines of the proof of Theorem \ref{thm:main_thm} presented above. However, we prefer to avoid the morass of notation that would appear in such a direct approach, and present an even more abstract construction of the left Kan extension giving rise to the symplectic cochains with support. We shall thus take a small detour to explicitly describe the theory of multicategories enriched over chain complexes, which is a special case of the notion of enriched multicategories discussed in Appendix \ref{sec:multicategories}.

Let $\Ch_R$ be the category of chain complexes over a fixed commutative ring $R$ with its standard symmetric monoidal structure given by tensor product of chain complexes. The reader should keep the case where $R$ is the Novikov ring in mind, since that is the only one that we will use.  
If we do not write anything under a tensor product sign it means that it is the tensor product of chain complexes over $R$.

Throughout this section we use the word multicategory to mean a \emph{differential graded multicategory} (a special case of Definition \ref{def:symmetric-multi-category}): writing $\cX$ for the object set, this means that the multimorphisms $\cC(\vec{x};y)$ with inputs a non-empty sequence $\vec{x}\in\cX^n$, and output an element $y \in \cX$ are given by an object of $\Ch_R$. The multicomposition operations, are given by $R$-linear chain maps
\begin{equation} \label{eq:multicategory-structure-maps-multiple-gluing}
    \circ:\bigotimes_{i=1}^n \cC(\vec{x}_i;y_i)\otimes\cC(\vec{y};z)\to\cC(\vec{x}\circ\vec{y},z)
\end{equation}
for each sequence of objects $\vec{y}$ of length $n$, and each collection of $n$-sequences of objects  $\vec{x}_1,\ldots ,\vec{x}_n$, and where $\vec{x}\circ\vec{y}$ denotes the replacement of the $i$th element in $\vec{y}$ by the sequence $\vec{x}_i$. Note that Equation \eqref{eq:multicategory-structure-maps-multiple-gluing} corresponds to applying the operations $\circ_i$ from Equation \eqref{eq:multicategory-glue-at-one-edge}, simultaneously for all elements of the sequence $\circ_i$. The fact that such an operation is well-defined (i.e. independent of the ordering of the multi-compositions) is a consequence of the axioms. The unital structure is simply an element in $\cC(x,x)$ for each object $x$, and we have the symmetry isomorphisms
\begin{equation}
\sigma^*:\cC(x_1,\dots,x_n;y)\to\cC(x_{\sigma(1)},\dots,x_{\sigma(n)};y) 
\end{equation}
given by $R$-linear chain maps.

The properties satisfied by these data are those given in Section \ref{sec:multicategories}, which the reader can find formulated in terms only of the composition operation $\circ$ at all inputs in Definition 2.2.21 of \cite{LeinsterBook}. Following the standard convention, we call a multicategory with one object an operad.

Note that there is a forgetful $2$-functor from (differential graded) multicategories to (differential graded) categories, which forgets $n-$ary multimorphism spaces with $n>1$. Let us call the category associated to a multicategory its underlying category. 
There is a functor in the other direction, which associates to a category the multicategory with the same morphism ($1$-ary multimorphism) spaces and $n-$ary multimorphism spaces with $n>1$ are set to $0$ (the chain complex with one element). By abuse of language, we will call such a multicategory a category. Therefore, our constructions for multicategories will specialize to constructions for categories. Note that a multifunctor $\cC\to \cD$ where $\cC$ is a category is nothing but a functor between the underlying categories.

The closed symmetric monoidal structure on $\Ch_R$ gives rise to a multicategory which we denote by $\Ch_R^{dg}$, with object set those of $\Ch_R$, and multimorphisms defined by
\begin{equation}
    \Ch_R^{dg}\left(\vec{C},D\right) :=\Hom_*\left(\bigotimes_{i=1}^n C_i;D\right).
\end{equation}

We omit writing down the standard definitions of multicompositions and symmetric group actions.

An algebra over a differential graded multicategory $\cC$
is simply a multifunctor \begin{equation}\cC\to \Ch_R^{dg}.\end{equation} Algebras over $\cC$ form a category that we will denote by \begin{equation}Alg_\cC.\end{equation} When $\cC$ is a category, an algebra over $\cC$ is sometimes called a module over $\cC$ or a $\cC$-space in the literature. We will not make a distinction and continue to think of categories as special cases of multicategories.

\subsection{The free algebra functor}\label{sec-free}

Given a multicategory $\cM$, let us denote by $\Color(\cM)$ the \emph{colors of $\cM$}, which is the smallest multicategory with the same objects as  $\cM$. Precisely, this means that the only multimorphisms in $\Color(\cM)$ are the endomorphism spaces which are all assumed to be the unit object of $\Ch_R$ (which is the ground ring considered as a complex concentrated in degree $0$). Note that the category of algebras over $\Color(\cM)$ has objects consisting of 
a collection of chain complexes over $R$ indexed by the objects of $\cM$, and morphisms given by collections of maps of chain complexes.

We have a forgetful functor ($U$ stands for underlying) \begin{equation}U: Alg_{\cM}\to Alg_{\Color \cM}.\end{equation} More interesting is the fact that $U$ admits a left adjoint: \emph{the free algebra functor} $\mathbb{F}_{\cM}: Alg_{\Color \cM}\to Alg_{\cM},$ which we now construct.

Let $C=\{C_x\}_{x\in Ob(\cM)}$ be an object of $Alg_{\Color \cM}.$ For every $y\in Ob(\cM)$, we define
\begin{equation} \label{eq:formula_free_functor}
  \mathbb{F}_{\cM}(C)(y):= \bigoplus_{n=1}^\infty\left( C_{x_1}\otimes\ldots \otimes C_{x_n} \otimes \bigoplus_{\vec{x}\in Ob(\cM)^n}\cM(\vec{x}; y) \right)_{S_n}.
\end{equation}
Here, the subscript $S_n$ means that we are taking coinvariants of the chain complex $(C_{x_1}\otimes\ldots \otimes C_{x_n}) \otimes \cM(\vec{x}; y) $ under the action of $S_n$ defined by the maps 
\begin{equation}\label{eq:S_n_action}
 C_{x_1}\otimes\ldots \otimes C_{x_n} \otimes \cM(\vec{x}; y)\to  C_{\sigma(x_1)}\otimes\ldots \otimes C_{\sigma(x_n)} \otimes  \cM(\sigma\cdot\vec{x}; y).
 \end{equation}

Using the multicompositions in $\cM$ one can define chain maps \begin{equation}\cM(\vec{x};y)\to Hom_{\Ch_R}(\mathbb{F}_{\cM}(C)(x_1)\otimes \ldots \otimes \mathbb{F}_{\cM}(C)(x_n); \mathbb{F}_{\cM}(C)(y)),\end{equation} which are compatible with multicompositions and the symmetric structure. There is a natural isomorphism \begin{equation}Alg_\cM(\bF_\cM(C),A)\simeq \bigoplus_{x\in Ob(\cM)} Hom_{\Ch_R}(C_x, UA_x)\end{equation}certifying that indeed the free algebra functor is left adjoint to the forgetful functor.

\subsection{Geometric realization preserves algebras}\label{sec-geo-real}

Let us denote the category of simplicial objects in $Alg_\cM$ by $Alg_\cM^\Delta$.  Note that an $\cM$-algebra $A$ can be thought of as a simplicial $\cM$-algebra whose $n$-simplices are given by $A$ for all simplicial degrees $n\geq 0$ and all face and degeneracy maps being the identity. For the next statement, we also recall that the geometric realization $C_*:\Ch_R^\Delta\to \Ch_R$
 maps every simplicial complex to the direct sum of the (shifted) underlying complexes, equipped with the alternating sum of the face maps. The fact that this is a symmetric monoidal functor is well-known, and goes back to Eilenberg and Zilber's idea \cite{EilenbergZilver1953} for comparing the geometric realisation of a product of simplicial sets with the product of the geometric realisations. The reader seeking a reference for the exact statement we use (for simplicial chain complexes and geometric realization as a semisimplicial chain complex) can find the proof in \cite[Proposition 2.16 and 2.17]{RichterSagave2020}.

\begin{lem}
  The Eilenberg-Zilber suffle map determines a lift
  \begin{equation}
       C_* : Alg_\cM^\Delta\to Alg_\cM 
     \end{equation}
     of the geometric realization functor to the category of algebras over $\cM$.

\end{lem}
\begin{proof}[Sketch of proof:]
The structure of a simplicial $\cM$-algebra consists of maps of simplicial chain complexes \begin{equation} A^\Delta_\bullet(x_1)\otimes\ldots \otimes A^\Delta_\bullet(x_n) \otimes \cM(\vec{x};y)_\bullet \to A^\Delta_\bullet(y),\end{equation} where $\cM(\vec{x};y)_\bullet $ denotes the constant simplicial chain complex.

We apply $C_*$ to this map and pre-compose it with the shuffle (Eilenberg-Zilber) map to obtain \begin{equation} C_* A^\Delta_\bullet(x_1)\otimes\ldots \otimes C_* A^\Delta_\bullet(x_n) \otimes C_*(\cM(\vec{x};y)_\bullet) \to C_* A^\Delta_\bullet(y).\end{equation}

Noting that $C_0(\cM(\vec{x};y)_\bullet)=\cM(\vec{x};y)$, we get the desired $\cM$-algebra structure maps. 
\end{proof}

\subsection{The extension functor $\mathbb{L}\pi_*: Alg_\cM\to Alg_\cO$}\label{sec-ext}

Given a map $\cM \to \cO$ of multicategories, our goal in this section is to finally construct the functor $\mathbb{L}\pi_*: Alg_\cM\to Alg_\cO$ in the abstract setup; the algebra $\Floer^*_{M,f\Mbar^{\bR}_0,\bullet}(K)$ which was explicitly constructed in Definition \ref{def:operadic_relative_chains}, is the outcome of applying this functor to the Floer algebra, and to the forgetful map from $    \cF_K $ to $f\Mbar^{\bR}_0  $. This point of view will be important in Section \ref{sec:comp-two-models-1} below, where we compare our model of symplectic cohomology with the one coming from the mapping telescope. 

Let us first define a functor on algebras over multicategories of colors  $\Color(\pi)_*: Alg_{\Color \cM} \to Alg_{\Color \cO}$ which maps an algebra $x\mapsto C_x$ to the algebra 
\begin{equation} y\mapsto \bigoplus_{x\in \pi^{-1}(y)}C_x.\end{equation}

We define a simplicial object $(\Delta\pi)_*A_\bullet$ in $Alg_\cO$ by the formula
\begin{equation} \label{eq:left_Kan_extension_Formula}
  (\Delta\pi)_*A_n= (\bF_\cO\circ \Color(\pi)_*\circ U)\circ (\bF_{\cM}U)^{n}(A).\end{equation}
It is clear how all degeneracy maps and all but one of the face maps (for each $n\geq 1$) are defined by the functoriality of $\bF_\cO\circ \Color(\pi)_*$ using the unit and the counit of the free-forgetful adjunction. For those last maps we construct an algebra map \begin{equation} \bF_\cO\circ \Color(\pi)_*\circ U\circ \bF_{\cM}(D)\to  \bF_\cO\circ \Color(\pi)_*(D),\end{equation}where $D$ is an $Ob(\cM)$ indexed collection, using $\pi$ to turn multimorphisms in $\cM$ to multimorphisms in $\cO$ and then using multicompositions in $\cO$.

\begin{defin}
  The extension of $A$ along $\pi$ is the geometric realisation
  \begin{equation}
        \mathbb{L}\pi_*A : = C_*  (\Delta\pi)_*A_\bullet.
  \end{equation}
\end{defin}

Going back to our geometric context, the definition we gave in Section \ref{sec:chain-complex} can be translated as follows:
\begin{lem} \label{lem:operadic_symmplectic_cochain_extension}
  The operadic symplectic cochains with support a compact subset $K$ are equal to the completion, with respect to the valuation of the Novikov ring, of the extension of the restriction of $\Floer^*$ to $\cF_K$ along the projection $\pi$ to the chains on $f\Mbar^{\bR}_0 $:
  \begin{equation}
      SC_{M,f\Mbar^{\bR}_0}^*(K) = \widehat{ \mathbb{L}\pi_* \left(\Floer^*|\cF_K \right)}.
  \end{equation} \qed
\end{lem}

\subsection{Proof of Theorem \ref{thm:linear_comparison_from_CF}}
\label{sec:proof-theorem-main-2}

We are now ready to prove the second main result asserted in the introduction:
\begin{proof}[Proof of Theorem \ref{thm:linear_comparison_from_CF}]
  We proceed to check the asserted properties as they are listed. First, to construct a map from the Floer cochains of every Hamiltonian $H$ which is negative on $K$, we consider the associated inclusion of the category $\ast_{(H,J)}$ with one object in $\cF_K$. This inclusion induces a map of extensions
  \begin{equation}
       \mathbb{L}\pi_* \left( \Floer^*|_{\ast_{(H,J)} }\right) \to   \mathbb{L}\pi_* \left(\Floer^*|\cF_K \right) .
  \end{equation}
  The left hand side is, essentially by definition, the geometric realisation of the constant simplicial functor on the free $C_*f \Mbar^{\bR}_{0} $ algebra on the Floer complex $\Floer^*(H,J)$, and thus has a canonical inclusion of this free algebra, associated to the $0$-simplices of the simplicial construction. To see this, note that, in the formula corresponding to Equation \eqref{eq:left_Kan_extension_Formula}, we have that the underlying functor on the category of modules over $ \ast_{(H,J)} $ is the identity functor, and hence so is  $\bF_{ \ast_{(H,J)}}$. A map of such algebras thus canonically corresponds to a map of complexes
  \begin{equation}
      \Floer^*(H,J) \to  U \mathbb{L}\pi_*  \left(\Floer^*|\cF_K \right). 
  \end{equation}
  Composing with the completion maps yields the map whose existence is asserted in Equation \eqref{eq:map_Ham_Floer_Cochains-SC}.

  Next, we produce the homotopy in Diagram \eqref{eq:homotopy_commutative_continuation}. Considering the inclusion of a category $\cF_{\kappa}$  with objects $(H_0,J_0)$ and $(H_1,J_1)$, and a unique morphism between them given by a continuation map $\kappa$.   The morphisms from $ \Floer^*(H_0,J_0)$ and $ \Floer^*(H_1,J_1)$  to $ SC^*_{M,f\Mbar^{\bR}_0}(K)$ both factor through the underlying complex of $ \mathbb{L}\pi_*  \left(\Floer^*|\cF_{\kappa} \right)$. Next, we consider the projection map $p$ from $ \cF_{\kappa}$ to the point. The inclusion of the unit in $f\Mbar^{\bR}_0$ induces a map
    \begin{equation}
      U \mathbb{L}p_*  \left(\Floer^*|\cF_{\kappa} \right) \to    U \mathbb{L}\pi_*  \left(\Floer^*|\cF_{K} \right) ,
    \end{equation}
and the maps from both $\Floer^*(H_0,J_0)  $ and $\Floer^*(H_1,J_1)  $ factor through it.  It thus suffices to show that the following diagram is homotopy commutative:
     \begin{equation} \label{eq:homotopy_commutative_continuation}
      \begin{tikzcd}
        \Floer^*(H_0,J_0)  \ar[r] \ar[dr] &    \Floer^*(H_1,J_1) \ar[d] \\
        & U \mathbb{L}p_*  \left(\Floer^*|\cF_{\kappa} \right).
      \end{tikzcd}
    \end{equation}
    We now analyse the target of these maps. Since the target category of $p$ is the point, the map denoted by $\bF_{\cO}$ in   Equation \eqref{eq:left_Kan_extension_Formula} is the identity map. The $0$-simplices of the associated simplicial chain complex are given by
    \begin{equation} \label{eq:0-simplices-two-object-model}
       (\Color(p)_*\circ U)\circ \Floer^*|_{ \cF_{\kappa}}   \cong  \Floer^*(H_0,J_0)  \oplus \Floer^*(H_1,J_1),
     \end{equation}
     and the two maps in Diagram \eqref{eq:homotopy_commutative_continuation} are given by the inclusions of these factors.

     We claim that there is a map from $\Floer^*(H_0,J_0) $ to the $1$-simplices, whose composition with the face map for $n=0$ gives the inclusion of the first factor in the right of Equation \eqref{eq:0-simplices-two-object-model}, and whose composition with the face map for $n=1$ gives the composition of $\kappa$ with the inclusion of the second factor. To see this, we compute from Equation \eqref{eq:formula_free_functor} (which simplifies significantly in the context of categories) that the value on the object $(H_1, J_1) $ of the algebra obtained by the composing $\Floer^*$ with the forgetful functor and then the free functor is 
     \begin{equation}
               \bF_{\cF_{\kappa}}U \left( \Floer^*|_{ \cF_{\kappa}} \right)(H_1, J_1) \cong \Floer^*(H_0,J_0)  \oplus \Floer^*(H_1,J_1),
             \end{equation}
     where the first summand corresponds to the identity of the object $(H_0,J_0)$, and the second to the unique morphism from this object to $(H_1,J_1)$. Tracing through the construction, we find that the inclusion of this first summand provides the desired homotopy. 

\end{proof}

\section{Comparison of the two models}
\label{sec:comp-two-models-1}

In this section, we prove Theorem \ref{thm:comparison}, which establishes the non-triviality of the operadic symplectic cochains with support by proving that this cochain complex is homotopy equivalent to the existing model for symplectic cochains with support.

While this section is fairly long, the basic ideas underlying the comparison can be explained quite succintly: as in Section \ref{sec:strictly-funct-coch} denote by $\cF_{K,\star}$ the subcategory of the $1$-categorical part of the indexing multi-category $\cF_K$ whose morphisms project to the unit in $f\Mbar^{\bR}_0(1)$. The most significant difference between the telescope construction and the model $SC_{M,f\Mbar^{\bR}_0}^*(K)$ is that the former is a model for the homotopy colimit of the Floer functor over $\cF_{K,\star}$, while the latter involves the entire multicategory. We will use the notation $SC_M^*(K)$ to denote the completion of some unspecified model for the homotopy colimit of the Floer functor $\Floer^*$  over $\cF_{K,\star}$, which is well-defined (up to homotopy equivalence) because the homotopy colimit is itself well defined up to contractible choice.

We are thus led to consider the following diagram of differential graded multicategories
\begin{equation}\label{eq:acceleration_diagram-1}
\begin{tikzcd}
C_* \cF_{K,\star} \arrow[r] \arrow[d]& C_* \cF_K \arrow[r,"\Floer^*"] \arrow[d,"\pi"] & \Ch_{\Lambda_{\geq 0}}^{dg}\\ 
\Lambda_{\geq 0}  \arrow[r] \ar[urr,dashed,bend left=10] & C_* f\Mbar^{\bR}_0 \ar[ur,dashed] & \end{tikzcd}
\end{equation} 
in which the dashed arrows are the left Kan extensions whose completions give rise to the two models of symplectic Floer cochains that we are trying to compare.

The statement that the two models are homotopy equivalent then follows from a general argument asserting this conclusion for any such diagram where the left square is a (homotopy) pullback square. The fact that we do obtain a pullback square in our geometric situation is ultimately a consequence of the results of Section \ref{sec:hamilt-index-categ}, specifically of Proposition \ref{lmFrgtHtpyEq}. Since the statement that a square of topological spaces is a homotopy pullback square is equivalent to the statement that the induced map of fibres is a homotopy equivalence, it should not be surprising that a proposition asserting that a projection map has contractible fibres implies the desired result.

If we had formulated our operadic constructions in a fully homotopically invariant way, e.g. using a theory of operads internal to quasi-categories, the above paragraph would be an outline of the proof, and its implementation would be shorter than the argument we provide, by relying on multiple citations of results in Section 3 of the manuscript \cite{Lurie2017}. Since the statements in the introduction are about operads in the ordinary differential graded sense, such an approach would require us to then further use a comparison between these two different notions.

We prefer to give a longer, explicit argument, relying on classical results of homotopical algebra, by comparing operadic left Kan extensions to categorical Kan extensions.

\subsection{Overview of the proof of Theorem \ref{thm:comparison}}
\label{sec:overv-proof-theor}

It is clear from the phrasing of the problem in the paragraphs surrounding Diagram \eqref{eq:acceleration_diagram-1} that the missing ingredient is a general statement about multi-categories. For this reason, until the end of the proof, everything will be stated in general algebraic terms.

As in Sections \ref{sec:diff-grad-mult}-\ref{sec-ext}, we fix a commutative ring $R$ and the following data: \begin{itemize} \item a differential graded multicategory $\cM$
\item a differential graded multicategory $\cO$, which we assume to only have one object 
\item a differential graded multi-functor $\pi: \cM\to \cO$
\item an $\cM$-algebra $A: \cM\to \Ch_R^{dg} $
\end{itemize} 
We call this the \emph{abstract setup} in this section. The construction of Section \ref{sec-ext} can be formulated as the construction of such data by applying the normalized chains of symmetric cubical sets (and the resulting functor from $\Ch_R^{\square} $ to $ \Ch_R^{dg} $) to the Floer theoretic constructions from Sections \ref{sec:hamilt-index-categ} and \ref{sec:floer-functor}. To be specific $R=\Lambda_{\geq 0}$, $C_* \cF_K $ plays the role of $\cM$,  $C_* f\Mbar^{\bR}_0 $ of $\cO$, $C_*$ applied to the multifunctor $\cF_K\to f\Mbar^{\bR}_0$ plays the role of $\pi$, and $\Floer^*$ that of $A$. In Section \ref{sec-ext}, we explained the construction of an $\cO$-algebra which is a model for the homotopical operadic left Kan extension of $A$ via $\pi$. We formulated the construction as a functor \begin{equation}\mathbb{L}\pi_*:Alg_\cM\to Alg_\cO.\end{equation}

On the other hand, by taking the fibre of $\pi$ over the unit of $\cO$, we obtain a category which we denote by $\cM_\star$.  The homotopy colimit of $A$ over $\cM_\star$, is equipped with a natural map of chain complexes
\begin{equation}
  \hocolim_{\cM_\star} A \to   \mathbb{L}\pi_* A. 
\end{equation}

Our goal is to prove that this map is a homotopy equivalence. We will need a key intermediate step where we use a functor from the category of multi-categories to the category of symmetric monoidal categories, called the PROP functor (see Section 5.4.1 in \cite{loday2012algebraic} and also Section \ref{sec-prop} below), and which we denote by $P$. From a sufficiently high-level perspective, the functor $P$ encodes multicategorical data in terms of equivalent categorical data, and the main difference between the two approaches is how the notion of symmetry is encoded.

\begin{rem}
  While not completely standard, the version of the PROP functor that we will consider will rely on an auxiliary choice of ordering of the set of objects of the multicategory which we consider. We stress that this ordering is not going to be arising from geometric considerations. Rather, we choose it because it simplifies the comparison between the constructions that we will associate to the PROP and the operadic construction.
\end{rem}

We will ignore part of the structure and consider the target of $P$ as simply the category of (differential graded) categories. As a first step, we note that if $\cM$ is a multi-category (with an ordered set of objects), then the objects of $P\cM$ are ordered sequences of objects of $\cM$.  We can apply the PROP functor to our abstract setup and obtain the following data:
\begin{itemize} \item a differential graded category $P\cM$,
\item a differential graded category $P\cO$,
\item a differential graded functor $P\pi: P\cM\to P\cO$, and 
\item a differential graded module  $PA: P\cM\to \Ch_R^{dg} $ (i.e. a $P\cM$-algebra).
\end{itemize}

We are thus at the position where we have to compare three different chain complexes:

\begin{enumerate}
\item (Operadic) The left Kan extension associated to the multifunctors $\pi$ and $A$: $\mathbb{L}\pi_*A,$
\item (Categorical) The left Kan extension associated to the functors $P \pi$ and $PA$: $\mathbb{L}P\pi_*PA,$
\item The homotopy colimit of $A$ over $ \cM_\star$.
\end{enumerate}

The inclusion of categories $ \cM_\star \to  P\cM$ induces a comparison map between the second and the third complex, and we shall discuss, in Section \ref{sec-hv} a criterion for when it is quasi-isomorphism.  We shall also show that there is always a map comparing the first and the second of these chain complexes, which is analogous to the map from the Borel construction to the quotient by a group action. In order for the comparison map to be a quasi-isomorphism, we need additional assumptions. Before we state them let us make an ad hoc definition extending, to the $\mathbb{Z}/2$-graded case, the notion of a bounded below and free $\mathbb{Z}$-graded right module (i.e. one which is a free $\mathcal{A}$-module in each degree, and vanishes below some dimension).
\begin{defin} \label{def:Z2_graded_bounded_below}
A $\mathbb{Z}/2$-graded right chain complex over $\mathcal{A}$ is called \emph{bounded below and free} if it is obtained by collapsing the grading of a bounded below and free $\mathbb{Z}$-graded right $\mathcal{A}$-chain complex.
\end{defin}

\begin{defin} \label{def:freeness_assumption}
  The abstract setup is said to satisfy the \emph{freeness hypothesis} if the following conditions hold:
   \begin{itemize}
\item[] \emph{(Identity)} The endomorphism algebra of every object  $\vec{x}$ of $Ob(P\cM)$ is given by the group ring on the subgroup $\Aut(\vec{x})$ of permutations of $[n]$ that preserve the map $\vec{x}:[n]\to Ob(\cM)$  \begin{equation}P\cM(\vec{x};\vec{x})=R[\Aut(\vec{x})]. \end{equation}
\item[] \emph{(Freeness 1)} For every pair of objects $\vec{x}$ and $\vec{y}$ of $P(\cM)$, the morphism complex  $\cM(\vec{x};\vec{y})$ is a bounded below and free right $R[\Aut(\vec{x})]$-chain complex.
\item[] \emph{(Freeness 2)} For every $n\geq 1$, the complex $\cO(n)$ is a bounded below and free right $R[S_n]$-chain complex.  \end{itemize}
\end{defin}

Note that Freeness 2 property implies that $\cO(n)$ is a bounded below and free right $R[G]$-chain complex for all subgroups $G\subset S_n$ by an elementary argument

The comparison of Operadic (1) and Categorical (2) Kan extensions as above is the subject of the following result, whose proof takes up the bulk of this section:
\begin{prop}\label{prop-cat-op-comp} Assuming that the freeness hypothesis holds, the value of the module $\mathbb{L}P\pi_*PA $ on the unit object of $P \cO$ is chain homotopy equivalent to the complex underlying $\mathbb{L}\pi_*A$.
\end{prop} 

\begin{rem}
A similar result was obtained in \cite[Proposition 1.15]{Horel2017factorization}, using more abstract model categorical arguments, but it does not appear that his result directly implies ours. 
\end{rem}

The above result reduces the desired equivalence from Theorem \ref{thm:comparison} to a comparison between the homotopy colimit and the categorical left Kan extension, which follows whenever the diagram
\begin{equation}\label{eq:acceleration_diagram-2}
\begin{tikzcd}
\cM_{\star} \arrow[r] \ar[d] &P \cM  \arrow[d,"\pi"]\\ 
\star  \arrow[r]  & P \cO  \end{tikzcd}
\end{equation}
is a homotopy pushout square. This is exactly the point that we alluded to in the discussion surrounding Diagram \eqref{eq:acceleration_diagram-1}, but, having passed from multicategories to categories, both the formulation and the proof become more classical, as we discuss in Section \ref{sec-hv} below.

\subsection{PROP functor}\label{sec-prop}
We now recall the construction of the PROP functor $P$ which associates to a multi-category with an ordering of the set of objects a symmetric monoidal category. We will not use the symmetric monoidal structure so we only explain the underlying categorical structure below. Throughout, all constructions are enriched over $\Ch_R$. 
\begin{defin}
Given a multicategory $\mathcal{M}$, define the associated \emph{PROP category} $P(\mathcal{M})$ to have
\begin{enumerate} \item Objects $Ob(P(\cM))$given by (finite) non-decreasing sequences of objects of $\cM$. We write $Os^k(\cM)$ for those sequences of length $k$.

\item Morphism complexes between objects $\vec{a}\in Os^n(\cM) $ and $\vec{b}\in Os^m(\cM) $ of $P(\mathcal{M})$, given by the direct sum, over all maps $f: [n]\to [m]$ of the sets labelling the sequences, of the tensor product over all elements $b_j$ of the output sequence $ \vec{b}$ of the complex of multimorphisms with (i) inputs the objects of $\vec{a}$ labelled by $f^{-1}(j)$ and (ii) output $b_j$: \begin{equation}P(\mathcal{M})(\vec{a},\vec{b}):= \bigoplus_{f:[n]\to [m]}\bigotimes_{j=1}^m \mathcal{M}(\vec{a}|_{f^{-1}(j)}; b_j).\end{equation}
\item The composition structure is defined in \cite[Construction 4.1]{may1978uniqueness}, using the composition in the multicategory along with its symmetric structure.
\end{enumerate}
\end{defin}
We clarify that the choice of ordering does not enter in the description of the morphisms; in this way, it should be clear that $P \cM$ is a full subcategory of a category whose objects are all (finite) sequences of objects.

Note that for $\sigma\in S_n$ whose action on $\vec{a}$ is trivial, the symmetric structure of the operad $\cM$ induces permutation morphisms $\sigma^*\in P(\mathcal{M})(\vec{a},\vec{a})$.   These will play a key role in the comparison argument.
\begin{rem}
The construction of composition in  \cite{may1978uniqueness} 
is much more clear when the colors of all objects in question are different. One first composes only using the composition structure in $\cM$ and then uses the symmetric structure so that the domain of the composition comes out correctly. In the general case where not all colors are different, we use the same formula - even though it is less obvious to see why this should be the formula if all colors are the same for example. Associativity is proved by a relatively painful explicit check. There is a third definition of multicategories, called fat multicategories in Appendix A.2 of \cite{LeinsterBook}, which makes the check of associativity trivial. Yet now the work is shifted to proving that a multicategory gives rise to a canonical fat symmetric multicategory.
\end{rem}

\begin{rem}
When $\cM$ is an operad there is a single object, and no auxiliary choice is needed. In this case, B. Fresse pointed out to us that there is a canonical isomorphism
\begin{equation}P\cM(n;m)\simeq \bigoplus_{\substack{\sum_{j=1}^m n_j=n \\  n_j\in \mathbb{N}}}\left(  R[S_n] \otimes_{R[S_{n_1}]\otimes \dots \otimes R[S_{n_m}]} \bigotimes_{j=1}^m\mathcal{M}(n_j)\right).\end{equation}
This is sometimes used directly as the definition of the PROP functor in the operad case, e.g. \cite{markl2008operads}, Example 59. The isomorphism can be constructed using the same idea as that of Proposition \ref{prop-untangle} below.
\end{rem}

The assignment $P$ lifts to a functor from multicategories (with orderings on the choice of objects, and multifunctors preserving this ordering) to categories in a straightforward way: given a multi-functor $F:\mathcal{M}\to \mathcal{M}'$,  
$P(F)$ acts on objects by assigning to a map $[n]\to Ob(\mathcal{M})$ its post-composition with $Ob(F): Ob(\mathcal{M})\to Ob(\mathcal{M}').$ From this description the action of $P(F)$  on morphisms  is obvious.

Given an algebra $A: \cC\to \Ch_R^{dg},$ we obtain a functor $P\cC\to P\Ch_R^{dg}$. Note also that there exists a canonical (symmetric monoidal) functor  $P\Ch_R^{dg}\to \Ch_R^{dg}$, which sends a sequence of vector spaces to their tensor product. By $PA$, we thus mean the $P\cC$-algebra obtained by the composition \begin{equation}P\cC\to P\Ch_R\to \Ch_R.\end{equation} 
We conclude:
\begin{lem}
  The PROP construction induces a functor \begin{equation}P: Alg_{\cM}\to Alg_{P\cM}.\end{equation} \qed
\end{lem}

As discussed in Section \ref{sec:diff-grad-mult}, the embedding of categories into multicategories implies that the constructions of Section \ref{sec-free}--\ref{sec-ext} directly apply to algebras over $P \cM$. In practice the formulas simplify significantly. For example, there is no need to pass to coinvariants in the construction of the free algebra because the group action is trivial.  In fact, Equation \eqref{eq:formula_free_functor}, which is our formula for the free algebra functor, simplifies even further for the specific algebras over $P \cM$ that we will consider:
\begin{lem}\label{lmUPA}
  If $D=\{D_{\vec{x}}\}_{\vec{x}\in Ob(P\cM)}$ is a collection of chain complexes satisfying the condition \begin{equation}D_{\vec{x}}=D_{x_1}\otimes\ldots \otimes D_{x_n},\end{equation}
  then the free algebra on $D$ is \begin{multline}\bF_{P\cM}(D)(\vec{y})= \\
    \bigoplus_{n=1}^\infty\bigoplus_{\vec{x}\in Ob(\cM)^n}\left(\bigoplus_{f:[n]\to [m]} D_{x_1}\otimes\ldots \otimes D_{x_n}  \otimes \left(\bigotimes_{j=1}^m \mathcal{M}(x|_{f^{-1}(j)}; y_j)\right)\right).\end{multline}  \qed
\end{lem}
 
We note that the hypothesis of the above result is satisfied whenever $D$ is obtained by applying the forgetful functor to the image of an $\cM$-algebra $A$ under $P$.

\subsection{From the  bar construction to the operadic extension}
We now construct a chain map 
\begin{equation} \label{eq:projection_map_ordered_bar_extension}
 \mathbb{L}P\pi_*PA(\star) \to  \mathbb{L}\pi_*A(\star).
\end{equation}
which we will show in the next two sections to be a quasi-isomorphism under the assumptions of Proposition \ref{prop-cat-op-comp}. This is where the choice of ordering becomes essential, and the starting point is the following formula for the (operadic) free algebra functor in terms of data involving the associated PROP:
 
\begin{lem}
  Denoting by  $\Aut(\vec{x})$ the permutations of $[n]$ that preserve a map $\vec{x}:[n]\to Ob(\cM)$, we have a natural isomorphism of chain complexes
  \begin{align}\label{eqFreeOrdered}
    \mathbb{F}_{\cM}(C)(y)&= \bigoplus_{n=1}^\infty\left(\bigoplus_{\vec{x}\in Ob(\cM)^n} C_{x_1}\otimes\ldots \otimes C_{x_n} \otimes \cM(\vec{x}; y)\right)_{S_n}\\&=\bigoplus_{n=1}^\infty\bigoplus_{\vec{x}\in Os^n(\cM)} (C_{x_1}\otimes\ldots \otimes C_{x_n}) \otimes_{R[\Aut(\vec{x})]}   \cM(\vec{x}; y), \end{align}
  for each object $y$. \qed
\end{lem}

We shall be specifically interested in a formula for the composition of the PROP functor with the free and forgetful functors applied to operadic algebras.  The next result provides such an expression; Figure \ref{fig:untangling_forest-corollas} justifies why we refer to it as \emph{untangling}.
\begin{figure}[h]
  \centering
  \begin{tikzpicture}
    \node[label=below:{$y_1$}] (y1) at (-2,0) {$\bullet$};
    \node[label=below:{$y_2$}] (y2) at (0,0) {$\bullet$};
    \node[label=below:{$y_3$}] (y3) at (2,0) {$\bullet$};
    \node[label=below:{$y_4$}] (y4) at (4,0) {$\bullet$};

\node[label=above:{$x^1_1$}] (x11) at (2,2) {$\bullet$};
    \node[label=above:{$x^1_2$}] (x12) at (-2,2) {$\bullet$};
    \node[label=above:{$x^1_3$}] (x13) at (0,2) {$\bullet$};
    \node[label=above:{$x^1_4$}] (x14) at (3,2) {$\bullet$};

\node[label=above:{$x^2_1$}] (x21) at (-3,2) {$\bullet$};
\node[label=above:{$x^2_2$}] (x22) at (5,2) {$\bullet$};

\node[label=above:{$x^3_1$}] (x31) at (1,2) {$\bullet$};
    \node[label=above:{$x^3_2$}] (x32) at (-1,2) {$\bullet$};
    \node[label=above:{$x^3_3$}] (x33) at (4,2) {$\bullet$};

    \node[label=above:{$x^4_1$}] (x41) at (-4,2) {$\bullet$};
    
    \draw[thick] (y1.center)--(x11.center);
    \draw[thick] (y1.center)--(x12.center);
    \draw[thick] (y1.center)--(x13.center);
    \draw[thick] (y1.center)--(x14.center);

    \draw[thick] (y2.center)--(x21.center);
    \draw[thick] (y2.center)--(x22.center);
   
    \draw[thick] (y3.center)--(x31.center);
    \draw[thick] (y3.center)--(x32.center);
    \draw[thick] (y3.center)--(x33.center);
    
    \draw[thick] (y4.center)--(x41.center);
        
    \begin{scope}[shift = {(0,-4)}]
    \node[label=below:{$y_1$}] (y1) at (-2,0) {$\bullet$};
    \node[label=below:{$y_2$}] (y2) at (0,0) {$\bullet$};
    \node[label=below:{$y_3$}] (y3) at (2,0) {$\bullet$};
    \node[label=below:{$y_4$}] (y4) at (4,0) {$\bullet$};

\node[label=above:{$x^1_1$}] (x11) at (-4,2) {$\bullet$};
    \node[label=above:{$x^1_2$}] (x12) at (-3,2) {$\bullet$};
    \node[label=above:{$x^1_3$}] (x13) at (-2,2) {$\bullet$};
    \node[label=above:{$x^1_4$}] (x14) at (-1,2) {$\bullet$};

\node[label=above:{$x^2_1$}] (x21) at (0,2) {$\bullet$};
\node[label=above:{$x^2_2$}] (x22) at (1,2) {$\bullet$};

\node[label=above:{$x^3_1$}] (x31) at (2,2) {$\bullet$};
    \node[label=above:{$x^3_2$}] (x32) at (3,2) {$\bullet$};
    \node[label=above:{$x^3_3$}] (x33) at (4,2) {$\bullet$};

    \node[label=above:{$x^4_1$}] (x41) at (5,2) {$\bullet$};
    
    \draw[thick] (y1.center)--(x11.center);
    \draw[thick] (y1.center)--(x12.center);
    \draw[thick] (y1.center)--(x13.center);
    \draw[thick] (y1.center)--(x14.center);

    \draw[thick] (y2.center)--(x21.center);
    \draw[thick] (y2.center)--(x22.center);
   
    \draw[thick] (y3.center)--(x31.center);
    \draw[thick] (y3.center)--(x32.center);
    \draw[thick] (y3.center)--(x33.center);
    
    \draw[thick] (y4.center)--(x41.center);
      
    \end{scope}
  \end{tikzpicture}
  \caption{Untangling a forest of $4$ corollas}
  \label{fig:untangling_forest-corollas}
\end{figure}
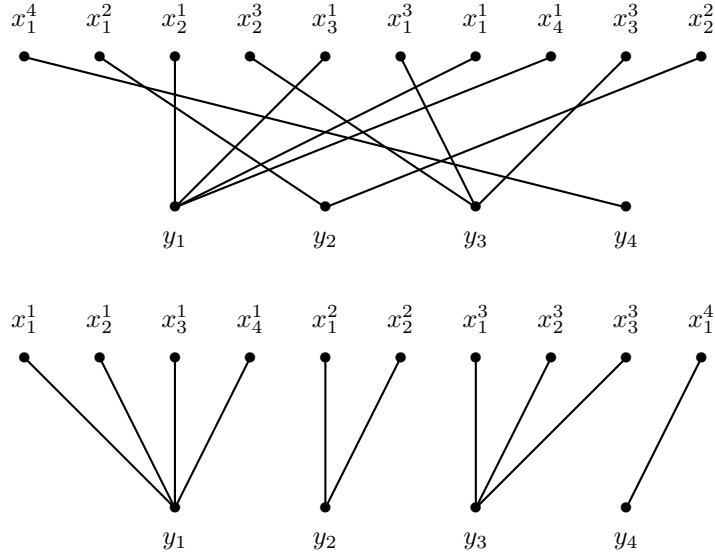

In the statement of the next result, the symbol $U$ refers to chain complex obtained by forgetting the operadic operations on an algebra, as in Section \ref{sec-free}, so that the composite $\mathbb{F}_\cM U D$ is the free algebra on the complex underlying $D$.
\begin{prop}[Untangling trick]\label{prop-untangle} For every $\cM$ algebra $D$ and an ordered sequence $\vec{y}$ 
  we have a canonical isomorphism 
\begin{equation}\label{eqUntangle} P\mathbb{F}_\cM UD(\vec{y})\cong \bigoplus_{\vec{x}\in  Os^n(\cM)}P\cM(\vec{x};\vec{y})\otimes_{R[\Aut(\vec{x})]}PD(\vec{x}).\end{equation}
\end{prop}
\begin{proof}
By definition $P\mathbb{F}_\cM UD(\vec{y})$ is equal to \begin{equation}\bigotimes_{j=1}^m\left( \bigoplus_{\vec{x}_j\in Ob  P\cM} (D(x_{j_1})\otimes\ldots\otimes D(x_{j_{n_j}}) \otimes_{R[\Aut(\vec{x}_j)]}  \mathcal{M}(\vec{x}_j; y_j))\right).\end{equation} We can rewrite this as 
\begin{equation}\label{eq-untangle-intermediate}
\bigoplus_{n=1}^\infty\bigoplus_{\substack{\sum_{j=1}^m n_j=n \\  n_j\in \mathbb{N}}}\left( \bigoplus_{\vec{x}_j\in Os^{n_j}(\cM)}\left((D(x_{j_1})\otimes\ldots\otimes D(x_{j_{n_j}})) \otimes_{R[\Aut(\vec{x}_j)]} \bigotimes_{j=1}^m\mathcal{M}(\vec{x}_j; y_j) \right)\right).
\end{equation}

We can put $\vec{x}_1,\ldots ,\vec{x}_m$ next to each other and obtain a sequence $\vec{x}$. There are $\Aut(\vec{x})$ many permutations of $[n]$ which make $\vec{x}$ into an ordered sequence. By making a choice for each of these, we obtain a canonical map from 

\begin{equation}\bigoplus_{n=1}^\infty\bigoplus_{\substack{\sum_{j=1}^m n_j=n \\  n_j\in \mathbb{N}}}\left( \bigoplus_{\vec{x}_j\in Os^{n_j}(\cM)}\left(\bigotimes_{j=1}^m(D(x_{j_1})\otimes\ldots\otimes D(x_{j_{n_j}}))\otimes  \mathcal{M}(\vec{x}_j; y_j)\right)\right) \end{equation} to the right hand side of Equation \eqref{eqUntangle} from the statement:

\begin{equation}\bigoplus_{\vec{x}\in Ob(P\cM)}PD(\vec{x})  \otimes_{R[\Aut(\vec{x})]} P\cM(\vec{x};\vec{y}).\end{equation} Clearly the map is independent of the choices. Moreover, it also descends to a map from the desired chain complex displayed in Equation \eqref{eq-untangle-intermediate}. It is easy to see that this map is an isomorphism.

\end{proof}

Using the fact that we defined the operadic Kan extension in terms of the free forgetful adjunction, the above result yields:
\begin{cor}\label{CorOperadicOrdered}
  An ordering of the objects of the multicategory $\cM$ determines an isomorphism between the operadic Kan extension $ \mathbb{L}\pi_*A(\star)$, and the simplicial chain complex whose $m$-simplices are given by the direct sum

\begin{multline} \label{eq:tensor_product_over_group_ring}
\bigoplus_{\vec{x}_1,\dots,\vec{x}_{m}\in Ob(P\cM)} PA(\vec{x}_0) 
\otimes_{R[\Aut(\vec{x}_{{0}})]}
P\cM(\vec{x}_{0};\vec{x}_{{1}})\otimes_{R[\Aut(\vec{x}_{1})]}\dots \otimes_{R[\Aut(\vec{x}_{m})]} \cO(k).
\end{multline}
\end{cor}

\begin{proof} Iteratively applying the canonical isomorphism \eqref{eqUntangle}, starting with $D=(\mathbb{F}_\cM U)^{m-1}A$, we obtain the isomorphism \begin{multline}(\mathbb{F}_\cM U)^mA(y)=P(\mathbb{F}_\cM U)^mA(y)\to \bigoplus_{\vec{x}_1,\dots,\vec{x}_{m}\in Ob(P\cM)} \\  PA(\vec{x}_0)
\otimes_{R[\Aut(\vec{x}_{{0}})]}
 P\cM(\vec{x}_{0};\vec{x}_{1})\otimes_{R[\Aut(\vec{x}_{1})]}\dots \otimes_{R[\Aut(\vec{x}_{m})]} \cM(\vec{x}_{{m}};y) .\end{multline}

For an $\cM$ algebra $D$, we also have that  \begin{align*}\bF_\cO\circ \Color(\pi)_*\circ U(D)(\star)&= \bigoplus_k\left(\left(\bigoplus_{y_1\in Ob(\cM)}D(y_1)\right)\otimes\dots\otimes\left(\bigoplus_{y_k\in Ob(\cM)}D(y_k) \right)\right) \otimes_{R[S_k]}  \cO(k)\\&=\bigoplus_k\bigoplus_{\vec{y}\in Os^k(\cM)}PD(\vec{y}) \otimes_{R[\Aut(\vec{y})]} \cO(k)\end{align*} which finishes the proof by applying it to $D=(\mathbb{F}_\cM U)^mA$.
\end{proof}

Since the bar construction $\mathbb{L}P  \pi_*PA(\star)$ is the geometric realisation of a simplicial complex whose $n$-simplices are given by
\begin{multline}\label{eq:tensor_product_ground_ring}\bigoplus_{[k_0]\to \ldots \to [k_n]}\bigoplus_{\substack{\vec{x}_i\in Os^{k_i}(\cM)\\i\in [n]}}PA(\vec{x}_0) \otimes P\cM(\vec{x}_{0};\vec{x}_{1})\otimes \ldots\otimes P\cM(\vec{x}_{n-1};\vec{x}_{n})\otimes \cO(k_n) ,\end{multline}
we conclude:
\begin{lem}
  The projection map from the tensor product over the ground ring to the tensor product over the group ring of $\Aut(\vec{x})$ 
  defines a map from Equation \eqref{eq:tensor_product_ground_ring} to Equation \eqref{eq:tensor_product_over_group_ring}, which induces a map from the simplicial chain complex giving rise to $\mathbb{L}P  \pi_*PA(\star)$ to the simplicial chain complex giving rise to $ \mathbb{L}\pi_*A(\star) $. \qed 
\end{lem}

This defines the map of Equation \eqref{eq:projection_map_ordered_bar_extension}. We now proceed to show that it is a quasi-isomorphism under the freeness hypothesis.

\subsection{Dg-flatness}

As alluded to at the beginning of this section, the heuristic for the proof that the comparison map in Equation \eqref{eq:projection_map_ordered_bar_extension} is an equivalence is the statement that the projection map from the Borel construction to the ordinary quotient is a homotopy equivalence for spaces with free actions.

We need a purely algebraic analogue of this idea, so we start by fixing  an associative algebra  $\mathcal{A}$ over $R$, where the key example we have in mind is the group ring $R[G]$ of a finite group.

\begin{defin}
A left $\mathcal{A}$-chain complex $M$ is \emph{dg-flat} if the functor $N \mapsto N\otimes_{\mathcal{A}} M$ sends quasi-isomorphisms of right $\mathcal{A}$-chain complexes to quasi-isomorphisms of $R$-chain complexes.
\end{defin}

For the next statement, recall that we provided an ad-hoc formulation for $\bZ_2$-graded chain complexes in Definition \ref{def:Z2_graded_bounded_below}, in terms of the existence of a bounded below graded lift:
\begin{lem}\label{lem-dg-flat-free}
Let $M$ be a bounded below and free left $\mathcal{A}$-chain complex. Then $M$ is dg-flat.
\end{lem}

\begin{proof}
 This is classical in the $\mathbb{Z}$-graded case, see 1.1.F - 1.2.F of \cite{avramov1991homological}.

Assume that we have a quasi-isomorphism of $\mathbb{Z}_2$ graded left $\mathcal{A}$-chain complexes $N\to N'$. Our assumption that $M$ is bounded below asserts the existence of a bounded below dg-flat $\mathbb{Z}$-graded $\mathcal{A}$-chain complex $\tilde{M}$ whose $\mathbb{Z}_2$ reduction is $M$. Then, it follows that \begin{equation}{N}\otimes _{\mathcal{A}} M\to {N'}\otimes _{\mathcal{A}} M\end{equation} is a quasi-isomorphism. This is because we can first unroll the $\mathbb{Z}_2$ complexes $N$ and $N'$ into $\mathbb{Z}$-graded ones, take tensor product as such with $\tilde{M}$ and then only consider what happens in degrees $0$ and $1$ to recover the desired map. This finishes the proof in the $\mathbb{Z}_2$ graded case.
\end{proof}

\begin{rem}
If we assume that $\mathcal{A}$ has finite global dimension, then we can drop the bounded below assumption from Lemma \ref{lem-dg-flat-free} (see the proof of Proposition \ref{prop-app} and 1.5 of \cite{avramov1991homological}). However, in our application below, we will have $\mathcal{A}=R[G]$ for a finite group $G$. Even if we assume that $R$ has finite global dimension and $|G|$ is a unit in $\mathcal{A}$, we could not determine if one can conclude that $R$ has finite global dimension. Of course, when $R$ is a field, then this is true: $\mathcal{A}$ has global dimension $0$ by Maschke's theorem.
\end{rem}

Let $M_r$ be a right $\mathcal{A}$-chain complex and $M_l$ be a left $\mathcal{A}$-chain complex. We define $B(M_r,\mathcal{A},M_l)$ to be the bar complex \begin{equation}\label{eq-bimod-bar}\bigoplus_{n=0}^\infty M_r\otimes \mathcal{A}^{\otimes n}\otimes M_l=M_r\otimes T\mathcal{A}\otimes M_l\end{equation} with its standard differential. Note that in our terminology, $B(M_r,\mathcal{A},M_l)$ is the geometric realization of the simplicial $R$-chain complex whose $n$-simplices are \begin{equation}B^n(M_r,\mathcal{A},M_l):= M_r\otimes \mathcal{A}^{\otimes n}\otimes M_l.\end{equation} When we set $M_l = \mathcal{A}$, we have a canonical a chain map of right $\mathcal{A}$-chain complexes: \begin{equation}B(M_r,\mathcal{A},\mathcal{A})\to M_r,\end{equation}  which is a homotopy equivalence of $R$-chain complexes, and hence in particular a quasi-isomorphism. This last conclusion holds more generally:
\begin{prop} If $M_l$ is dg-flat leftt $\mathcal{A}$-chain complex, then the canonical map
of $R$-chain complexes \begin{equation}p: B(M_r,\mathcal{A},M_l)= B(M_r,\mathcal{A},\mathcal{A} )\otimes_{\mathcal{A}} M_r \to M_r\otimes_{\mathcal{A}} M_l,\end{equation}  is a quasi-isomorphism.  \qed
\end{prop}

It will be convenient for our arguments to factor this map through a map of geometric realisations. We thus define the $R$-chain complex \begin{equation}M_r\otimes^s_{\mathcal{A}}M_l:=\bigoplus_{n=0}^{\infty}M_r\otimes_{\mathcal{A}}M_l[n]\end{equation} which is the geometric realization of the constant simplicial $R$-chain complex $M_r\otimes_{\mathcal{A}}M_l$ (the differentials alternate between vanishing and agreeing with the identity map). Note that we have a canonical quasi-isomorphism \begin{equation}
  q:M_r\otimes^s_{\mathcal{A}}M_l\to M_r\otimes_{\mathcal{A}}M_l.
\end{equation}

Let us note that there is always a simplicial map \begin{equation}B_\bullet(M_r,\mathcal{A},M_l)\to M_r\otimes_{\mathcal{A}}M_l,\end{equation} where the target is thought of as the constant simplicial $R$-chain complex. It is defined for any $n\geq 0$ by taking the map $M_r\otimes \mathcal{A}^{\otimes n}\otimes M_l\to M_r\otimes_{\mathcal{A}}M_l$ to be \begin{equation}m_r\otimes a_1\otimes\ldots \otimes a_n\otimes m_l\mapsto m_ra_1\ldots a_n\otimes m_l.\end{equation} This is in fact a bi-natural transformation.

Taking geometric realizations, we obtain a chain map 
\begin{equation}\label{eq:comparison_map_iterated_bar_tensor-chain}
  f: B(M_r,\mathcal{A},M_l)\to M_r\otimes^s_{\mathcal{A}}M_l.
\end{equation}

\begin{rem}
Assuming that $M_l$ is an $\mathcal{A}-\mathcal{A}'$ bimodule, this map is a map of $\mathcal{A}'$ modules. 
\end{rem}We conclude:
\begin{cor}\label{prop-free-quasi}We have a commutative diagram \begin{equation}
  \begin{tikzcd}
    B(M_r,\mathcal{A},M_l) \ar[r, "f"] \ar[dr, "p"] & M_r\otimes^s_{\mathcal{A}}M_l  \ar[d,"q"] \\
     & M_r\otimes_{\mathcal{A}}M_l &
  \end{tikzcd}
\end{equation}

Assuming that $M_l$ is dg-flat, the horizontal map is a quasi-isomorphism. \qed
\end{cor}

\subsection{Proof that the map from the  bar construction to the operadic extension is a quasi-isomorphism}
\label{sec:proof-that-map}

In this section, we show that the comparison map from Equation \eqref{eq:projection_map_ordered_bar_extension} is a quasi isomorphism by equipping to two sides with compatible filtration, and showing that the map induces an isomorphism on associated graded homology groups. To simplify the notation, we write
\begin{align}
  L:= &\mathbb{L}P\pi_*PA(\star) \\
  L':=& \mathbb{L}\pi_*A(\star).
\end{align}
We begin by recalling that 
\begin{multline*}L =\bigoplus_{n\geq 0}\bigoplus_{[k_0]\to [k_1]\to\ldots \to [k_n]\to 1}\\ \bigoplus_{\substack{\vec{x}_i\in Os^{k_i}(\cM)\\i\in [n]}}  PA(\vec{x}_0) \otimes P\cM(\vec{x}_{0};\vec{x}_{1})\otimes \cdots\otimes P\cM(\vec{x}_{n-1};\vec{x}_{n})  \otimes \cO(k_n) [n].\end{multline*}

For an integer $m\geq 0$ let $O_m$ be the set of $(m+1)$-tuples $(\vec{x}_0,\dots,\vec{x}_m)$ of ordered sequences of objects of $\cM$ with the property that $\vec{x}_i\neq\vec{x}_{i+1}$. 

Under the identity assumption, and using $T$ for the tensor algebra construction, we can rewrite $L$ as 
\begin{multline*}L= \bigoplus_{m\geq 0}\bigoplus_{\quad(\vec{x}_0,\dots,\vec{x}_{m})\in O_{m}}  PA(\vec{x}_0)
\otimes T (R[\Aut(\vec{x}_{{0}})])
\otimes P\cM(\vec{x}_{0};\vec{x}_{1})\otimes \\ T (R[\Aut(\vec{x}_{1})])\otimes
 \cdots \otimes P\cM(\vec{x}_{m-1};\vec{x}_{m})\otimes T(R[\Aut(\vec{x}_{0})])\otimes  \cO(|\vec{x}_{{m}}|).
\end{multline*}

Note that the summands here are isomorphic as a $k$-module to an iterated bar construction as in Equation \eqref{eq-bimod-bar} with a left parenthesis inserted after every occurrence of $T (R[\cdot])$ and right parentheses at the end. We remind the reader that this splitting of $L$ is only as a $k$-module and not as a chain complex. The differential has some terms which do preserve this decomposition and this agrees with the differential of the iterated bar complex from Equation \eqref{eq-bimod-bar}. On the other hand, part of the differential decreases $m$ by composing $\vec{x}_i$ and $\vec{x}_{i+1}.$

\begin{lem}
Under the Identity assumption, there is a canonical isomorphism of modules 
\begin{multline*}
L'= \bigoplus_{m\geq 0}\bigoplus_{\quad(\vec{x}_0,\dots,\vec{x}_{m})\in O_{m}}  PA(\vec{x}_0)
\otimes^s_{R[\Aut(\vec{x}_{{0}})]}
P\cM(\vec{x}_{0};\vec{x}_{1})\otimes^s_{R[\Aut(\vec{x}_{1})]}\\
\cdots \otimes^s_{R[\Aut(\vec{x}_{m-1})]} P\cM(\vec{x}_{m-1};\vec{x}_{m})\otimes^s_{R[\Aut(\vec{x}_{m})]} \cO(|\vec{x}_{{m}}|)
\end{multline*}
Under this identification, the map $L\to L'$ of Equation \eqref{eq:projection_map_ordered_bar_extension} is given by  applying Equation \eqref{eq:comparison_map_iterated_bar_tensor-chain} iteratively. The map respects the module grading by $m$.
\end{lem}
\begin{proof}

The isomorphism part of the claim follows from Corollary \ref{CorOperadicOrdered} and the rest of the computation is straightforward.
  
\end{proof}

Let us call the map of Equation \eqref{eq:projection_map_ordered_bar_extension} after this identification \begin{equation}\Phi: L\to L'.\end{equation}

\begin{rem}
Here is a toy version of the next argument to aid the reader. Let us use the notation of Corollary \ref{prop-free-quasi} and in addition assume that we are given an $R$-module $M$ and an $R$-linear map\begin{equation}M_r\otimes M_l\to M, \end{equation} which factors through $M_r\otimes M_l\to M_r\otimes_{\mathcal{A}} M_l.$ Then, we can construct the cones of the canonical chain maps\begin{equation}B(M_r,\mathcal{A},M_l)\to M\end{equation} and $M_r\otimes_{\mathcal{A}}^s M_l\to M$. The map $f$ induces a chain map on these cones, which is a quasi-isomorphism under the same assumption that $M_l$ is dg-flat.
\end{rem}

Let $L_m$ and $L_m'$ be the submodules with exactly $m$ essential levels. We also define $F_mL:=\bigoplus_{n=0}^m L_m$ and  $F_mL':=\bigoplus_{n=0}^m L_m'$. Noting \begin{equation}F_mL=F_{m-1}L\oplus L_m\end{equation} we can write the differential in block form. Considering $L_m$ and $L_m'$ as a chain complex with their canonical differential, we can express $F_mL$ and $F_mL'$ as the cones of the canonical maps \begin{align}f_m: L_m & \to F_{m-1}L \\
f_m': L_m'& \to F_{m-1}L'.\end{align}

\begin{prop}\label{prop-cone-L}
The diagram \begin{equation}
  \begin{tikzcd}
    L_m \ar[r, "f_m"] \ar[d, "\Phi_m"] & F_{m-1}L  \ar[d,"F_{m-1}\Phi"]  \\
    L_m' \ar[r, "f_m'"] & F_{m-1}L' &
  \end{tikzcd}
\end{equation} commutes. \qed

\end{prop}

\begin{thm} \label{thm:borel-iso-ordinary-quotient}
Under the freeness hypothesis, the map \begin{equation}F_m\Phi: F_mL\to F_mL'\end{equation} is a quasi-isomorphism for all $m\geq 0$.

\end{thm}
\begin{proof}
We will do this by induction on $m$. For $m=0$, this is a direct consequence of Lemma \ref{lem-dg-flat-free} and Proposition \ref{prop-free-quasi}. In fact, similarly, the map \begin{equation}\Phi_m: L_m\to L_m'\end{equation} is a quasi-isomorphism by an iterated use of Proposition \ref{prop-free-quasi}. Then, we use Proposition \ref{prop-cone-L} along with the long exact sequence of a cone to finish the induction.
\end{proof}

 Because our filtrations are exhaustive, taking cones reduces the next result to an acyclicity statement which is easy to see:
 
 \begin{cor} \label{cor:borel-iso-ordinary-quotient}
   Under the freeness hypothesis, $\Phi:L\to L'$ is a quasi-isomorphism. \qed
 \end{cor}

\subsection{Two-sided bar constructions as left Kan extensions}
\label{sec:two-sided-bar}
It is convenient, for the next two sections, to formulate the left Kan extensions in the categorical context as a $2$-sided bar construction (c.f. \cite[Appendix A]{GugenheimMay1974}). Given a differential graded category $\cC$, we use the term \emph{right module} for a differential graded functor from $\cC$ to the category of chain complexes, while a \emph{left module} is a differential graded functor on the opposite category.

\begin{rem}
We recommend Section 3 of \cite{HollenderVogt1997} as a reference for the two-sided bar construction and a list of basic properties. Especially important for us is Section 5 of the same reference, where the two-sided bar construction is used to construct a left homotopy Kan extension along enriched functors. Even though the statements in \cite{HollenderVogt1997} are formulated for  categories enriched over topological spaces, the construction and results that we use extend in a straightforward way to chain complexes.
\end{rem}

We associate to a left module $\cL$ and a right module $\cR$ over $\cC$ a simplicial chain complex $B_{\bullet}(\cR,\cC, \cL)$ whose $n$-simplices are given by
\begin{equation}
  \bigoplus_{(x_0, \ldots, x_n) \in \Ob \cC^{n} } \cR(x_0) \otimes \cC(x_0, x_1) \otimes \cdots \otimes \cC(x_{n-1}, x_n) \otimes \cL(x_n)   
\end{equation}
and whose face maps are given for $i=0$ by the action of $\cC$ on $\cL$, for $0 < i < n$ by composition in $\cC$, and for $i=n$ by the action of $\cC$ on $\cR$. This can be formulated more abstractly in terms of a free-forgetful adjunction. We abuse notation and write $B(\cR,\cC, \cL) $ for the associated chain complex. More generally, if $\cR$ is a differential graded $\cB$-$\cC$ bimodule, and $\cL$ is a $\cC$-$\cD$ bimodule, then the bar construction is a $\cB$-$\cD$ bimodule.

The following result asserts that the left Kan extension that we have been studying, when computed for categories, is given by a $2$-sided bar construction. In its formulation, we shall use the following notation: for a differential graded functor $p:\cA\to \cC$,  denote by ${}_p \cC$ the $\cA-\cC$-bimodule which assigns to a pair of objects $(a,c)$ the morphism space $\cC(f(a),c)$.

\begin{lem} \label{lem:bar=leftkan}
Let $\cC$ be a category and $\cR$ be a right module.  If $p: \cA\to \cC$ is a functor, we have an isomorphism of right $\cA$-modules
\begin{equation}\label{eq:case_of_categories}
  \mathbb{L}p_* \cR \cong 
  B(\cR ,\cA, {}_p \cC)\end{equation}  
\end{lem}

\begin{proof}
Recall that a right module for a category considered as a multicategory is simply what we referred to as an algebra. The only trees that contribute to the left hand side are the linear ones. Moreover, the symmetric group actions in this case do not play any role and we do not need to take coinvariants at any step. By inspection of the constructions, it follows that the two sides of the equation are precisely the same.
\end{proof}

\subsection{Hollender-Vogt cofinality result}\label{sec-hv}

Consider a commutative square of differential graded functors:
\begin{equation}\label{eq:square_dg}
  \begin{tikzcd}
    \cA \ar[r, "f"] \ar[d, "p"] & \cB \ar[d,"q"]  \\
    \cC \ar[r, "g"] & \cD.
  \end{tikzcd}
\end{equation}
We say that such a diagram is a \emph{homotopy pushout square} if the natural comparison map
  \begin{equation}\label{eq:induced_nat_map} 
    B(f^* \cB, \cA, {}_p \cC ) \to g^* {}_{q} \cD
  \end{equation}
  is a quasi-isomorphism of $\cB-\cC$-bimodules.

\begin{lem}[Proposition 5.4 of \cite{HollenderVogt1997}] \label{lem:hollender-vogt-criterion}
  Given a right module $X : \cB \to \Ch_R$, the comparison map of $\cC$ modules
 \begin{equation} \label{eq:comparison_square_hocolim}
  B( f^*X ,  \cA, {}_p \cC ) \to g^*B(X,  \cB, {}_q \cD).
\end{equation}
is a quasi-isomorphism for each homotopy pushout square as in Equation \eqref{eq:square_dg}.

\end{lem}
\begin{proof}
  We consider the following sequence of maps
  \begin{align}
    g^*B(X , \cB, {}_q\cD) & = B(X , \cB, g^* {}_q \cD)\\
    &\leftarrow B(X,  \cB, B( f^* \cB ,  \cA, {}_p \cC) )\notag \\
    & \cong B(  f^*B(\cB,  \cB, X) , \cA, {}_p \cC) \notag\\
    & \simeq    B(f^*X , \cA, {}_p \cC ).\notag
  \end{align}
  The first arrow is a quasi-isomorphism by assumption, the second is an isomorphism of chain complexes, and the third is a homotopy equivalence by the acyclicity of the bar complex.
\end{proof}

Returning to the outline of our strategy provided in Section \ref{sec:overv-proof-theor}, find that we can complete the comparison between the operadic Kan extension and the homotopy colimit, in a setting slightly more general than Diagram \eqref{eq:acceleration_diagram-2}, where we replace the fibre $ \cM_{\star}$ by a potentially different category: 
\begin{cor} \label{cor:comparison_operadic_Kan_htpy_colim}
Given a functor $\cN \to \cM_{\star} $ with the property that the Diagram
\begin{equation}\label{eq:acceleration_diagram-3}
\begin{tikzcd}
\cN  \arrow[r] \ar[d] &P \cM  \arrow[d,"\pi"]\\ 
\star  \arrow[r]  &  P \cO  \end{tikzcd}
\end{equation} is a homotopy pushout square, the comparison map
  \begin{equation}
         \hocolim_{\cN} A \to   \mathbb{L}\pi_* A
  \end{equation}
is a  quasi-isomorphism. \qed
\end{cor}

\subsection{Proof of Theorem \ref{thm:comparison}}
\label{sec:proof-prop-refpr}

We need one final ingredient for the proof of the remaining result stated in the introduction: the techniques of this section in general produce quasi-isomorphisms of complexes obtained from various bar constructions. 
To obtain homotopy equivalences, we will thus need the following result:

\begin{prop}\label{prop-app} Let $\Bbbk$ be a field and $\Lambda_{\geq 0}$ be the Novikov ring over $\Bbbk$.
Let $C$ be a ($\bZ$-graded or $\bZ/2$-graded) chain complex over $\Lambda_{\geq 0}$. If the underlying graded module of $C$ is degree-wise free, then its acyclicity implies its contractibility.
\end{prop}
\begin{proof} We consider the projective model structure on $\Ch_{\Lambda_{\geq 0}}$ as in \cite{hovey2007model}. By Remark 2.3.7 of the same reference the cofibrant objects correspond to the DG-projective
complexes of \cite{avramov1991homological}.

It is not difficult to prove that the global dimension of $\Lambda_{\geq 0}$ is finite using \cite[065T]{stacks-project} under the assumption that $\Bbbk$ is a field. Proposition 3.4 of  \cite{avramov1991homological} shows that degree-wise free chain complexes are cofibrant! Lemma 2.3.8 of  \cite{hovey2007model} finishes the proof. For the last step one can also use the general Whitehead theorem for model categories (Theorem 7.5.10 of \cite{hirschhorn2009model}) as all chain complexes are fibrant in the projective model structure.
The $\mathbb{Z}/2$-graded case is a consequence of  \cite[Proposition 5.9]{keller2010}).
\end{proof}

We now assemble the results we have established in the proof of our comparison statement:
\begin{proof}[Proof of Theorem \ref{thm:comparison}]
  We begin by checking that the conditions of  Definition  \ref{def:freeness_assumption} hold in the geometric context. Identity assumption immediately follows from the fact that the objects of $\cF$ form partially ordered sets in the following sense; if there are non-zero morphisms from $x$ to $y$ and vice versa, then $x=y$. The second Freeness assumption follows from the fact that the symmetric group on $k$-letters acts freely on $f \Mbar^{\bR}_{0,k}$, and this implies the first Freeness assumption given the forgetful map from $\cF$ to $f \Mbar^{\bR}_{0}$ (any fixed points in the multimorphism spaces must map to a fixed point in the operad, but none exist).

  We conclude the existence of a canonical quasi-isomorphism
  \begin{equation}
       \mathbb{L}P\pi_*P\Floer(\star)\to \mathbb{L}\pi_*\Floer(\star) ,
  \end{equation}
 by Corollary \ref{cor:borel-iso-ordinary-quotient}. 

 Next, we establish that the fact that the diagram
  \begin{equation}\label{eq:acceleration_diagram}
\begin{tikzcd}
\mathbb{N}_{\Lambda_{\geq 0}}\arrow[r] \arrow[d]& PC_* \cF_K  \arrow[d,"P\pi"]\\ 
\star_{\Lambda_{\geq 0}} \arrow[r] & PC_*f\Mbar^{\bR}_0 \end{tikzcd}
\end{equation}
associated to a cofinal sequence of Hamiltonians is a homotopy pushout square, in order to apply Corollary \ref{cor:comparison_operadic_Kan_htpy_colim}.   Here,  $\mathbb{N}_{\Lambda_{\geq 0}}$ is the category with objects the natural numbers $0,1,2,\ldots$ and morphisms \begin{equation}Hom_{\mathbb{N}_{\Lambda_{\geq 0}}}(i,j)=\begin{cases}  
 \Lambda_{\geq 0} & i\leq j\\ 0  & i>j,
 \end{cases}\end{equation}
and the cofinal sequence of Hamiltonians determines a functor to $ C_* \cF_K$.

Explicitly, we need to show that for each sequence $((H_1,J_1), \ldots, (H_k,J_k))$ of objects of $\cF_K$ the natural map
\begin{equation}
B\left(C_*\cF_K((H_1,J_1), \ldots, (H_k,J_k); \_) , \mathbb{N}^{op}_{\Lambda_{\geq 0}}, \Lambda_{\geq 0}  \right) \to C_{*}f\Mbar^{\bR}_{0,k+1}
\end{equation}
is a quasi-isomorphism. 

If all the Hamiltonians in the cofinal sequence are strictly larger than $(H_1,\ldots H_k)$ the last statement is an immediate consequence of Proposition \ref{lmFrgtHtpyEq}. Therefore, for each fixed sequence $(H_1, \ldots, H_k)$ the statement holds if we start the cofinal sequence from a large enough natural number. But the homotopy type depends only on the tail of the sequence, so the statement holds for  the original cofinal sequence. 

It is a well-known result that the telescope of a diagram $\mathbb{N}_{\Lambda_{\geq 0}}\to C_* \cF_K$ admits a canonical quasi-isomorphism to its homotopy colimit. The latter is of course nothing but the left Kan extension of $\mathbb{N}_{\Lambda_{\geq 0}}\to C_* \cF_K$ over $\star_{\Lambda_{\geq 0}}$

Combining these results with the fact that completion preserves quasi-isomorphisms of chain complexes with torsion free underlying modules (see \cite[Corollary 2.3.6 (3)]{Varolgunes2018} finishes the proof of the first statement of Theorem \ref{thm:comparison}. The homotopy equivalence statement is a direct consequence of Proposition \ref{prop-app} noting that homotopy equivalences are automatically preserved by completion. The homotopy commutativity of the diagram is similar to the proof of Theorem \ref{thm:linear_comparison_from_CF}.

\end{proof}

\appendix

\section{Trees and Riemann surfaces}
\label{sec:trees-riem-surf}
\subsection{Conventions about trees}

\begin{figure}[h]
  \centering
  \begin{tikzpicture}
    \begin{scope}
    \node[circle, fill=black, inner sep=0, minimum size = 6] (v1) at (2.5,-1) {};
    \foreach \i in {1,...,4}
    {\draw (\i,0) -- (v1);
      \node[label=above:{$e_{\i}$}] (e\i) at (\i,0) {};};

    \end{scope}
      \begin{scope}[shift={(5,0)}]
       \node (e...) at (0,0) {$\cdots$};

       \begin{scope}[shift={(0,-1)}]
         \node[circle, fill=black, inner sep=0, minimum size = 6, label=right:{$v_{root}$}] (v) at (0,-2.5) {};
         \draw (v) --  +(-90:1);
         \node[label=below:{$e_{0}$}] (e0) at (0,-3.5) {};
         \foreach \i in {110, 100, ..., 70} \draw[dashed] (v)-- +(\i:0.5);
        \end{scope}
       \end{scope}

    \begin{scope}[shift={(-1,0)}]
    \node[label=above:{$e_{n-1}$}] (en-1) at (7,0) {};
    \node[label=above:{$e_{n}$}] (en) at (8,0) {};
    \node[circle, fill=black, inner sep=0, minimum size = 6] (v'') at (7.5,-2) {};
    \draw (en)--(v'') -- (en-1);

    \end{scope}

\draw[dashed] (v1)--(v)--(v'');    
    
  \end{tikzpicture}
  \caption{A stable tree}
  \label{fig:stable_tree}
\end{figure}
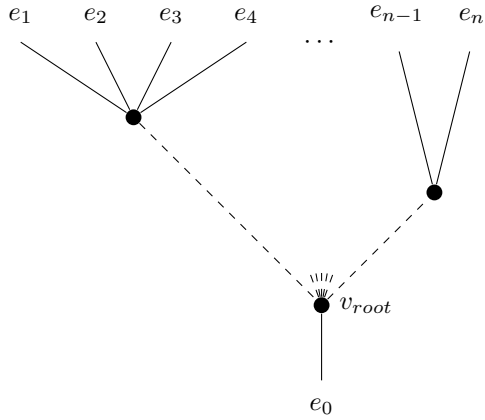

In what follows we will allow trees $T$ to have edges that have one endpoint. We call such edges \emph{external edges} and denote them by $E_{ext}(T)$ and call the edges with two endpoints \emph{internal edges}, which we denote by $E_{int}(T)$. We write $E(v)$ for the set of edges adjacent to a vertex $v$, and denote the \emph{valency} of $v$ by $\upsilon(v):=|E(v)|$. We assume that the valency of each vertex is at least two. Note that our trees (without any other qualification or adjective) are not equipped with a planar structure and we do not allow edges with zero endpoints.

\begin{defin}
  For a natural number $n$, a \emph{pre-stable tree $T$ with $n$ inputs and $1$ output} is a tree equipped with a bijection
\begin{equation}
\{0,\dots,n\}\to E_{ext}(T).
\end{equation}

A pre-stable tree $T$ is called \emph{stable} if $\upsilon(v)\geq 3$ for all $v\in V(T)$.
\end{defin}
We now describe the notation we use, which is illustrated in Figures \ref{fig:stable_tree} and \ref{fig:edges_vertices_labels}. Refer to the external edge labelled by $0$ as the \emph{output}.  The other external edges are the \emph{inputs} of $T$. 
For each vertex $v$, there is a distinguished outgoing edge $e_{out}(v)\in E(v)$ lying on the minimal arc connecting $v$ to the output edge. 
The edges adjacent to $v$ which are not outgoing are called incoming edges:
\begin{equation}
E_{in}(v) \equiv E(v)\setminus\{e_{out}(v)\}.
\end{equation}
Each internal edge is the distinguished outgoing edge of exactly one vertex denoted by $v^-(e)$. The other vertex of $e$ is denoted by $v^+(e)$.

\begin{figure}[h]
  \centering
  \begin{tikzpicture}
                 \node[circle, fill=black, inner sep=0, minimum size = 6, label=right:{$v^-_e$}] (v+e) at (-2,-2) {};
              \foreach \i in {110,100,80,70} \draw[dashed] (v+e)-- +(\i:0.5);

    \node[circle, fill=black, inner sep=0, minimum size = 6, label=right:{$v^+_e$}] (v-e) at (-2,-3) {};
    \draw (v+e)--(v-e);
    \node[label=right:{$e$}] (e) at (-2,-2.5) {};
    \foreach \i in {110, 100, 80,70} \draw[dashed] (v-e)-- +(\i:0.5);
   \draw (v-e)-- +(-90:0.5);
    \draw[dashed] (v-e)-- +(-90:1);

    \begin{scope}[shift={(0,-.5)}]
     \node[label=above:{$E_{in}(v)$}] (Einv) at (2,-1.5) {};
    \node[circle, fill=black, inner sep=0, minimum size = 6, label=right:{$v$}] (v) at (2,-2.5) {};
    \node[label=right:{$e_{out}(v)$}] (eout) at (2,-3) {};
     \draw (v)-- +(-90:0.5);
     \draw[dashed] (v)-- +(-90:1);
     \foreach \i in {110, 100, 80,70}
     {\draw (v)-- +(\i:0.5);
       \draw[dashed] (v)-- +(\i:1);
     };
\end{scope}

  \end{tikzpicture}
  \caption{Conventions for labelling edges and vertices}
  \label{fig:edges_vertices_labels}
\end{figure}
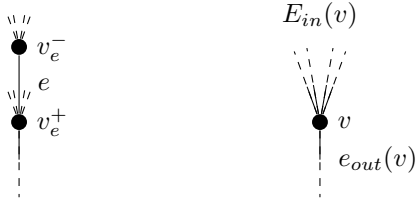

Below, for a pre-stable tree $T$, we abbreviate $v\in T$ to mean $v\in V(T)$.
\subsubsection{Pre-stable Riemann surfaces with cylindrical ends}
We now introduce our conventions for Riemann surfaces equipped with input and output marked points; as discussed in Remark \ref{sec:PROP-structure} in our work, we only consider the case of $1$ output marked point:

\begin{defin}
  A \emph{genus-$0$ punctured Riemann surface} is a Riemann surface $\Sigma$ which is the complement of finitely many points $P$ in a closed genus-$0$ Riemann surface $\overline{\Sigma}$. The elements of $P$ are referred to as the punctures of $\Sigma$.

  A \emph{choice of cylindrical ends on $\Sigma$} is a decomposition $P=P_{in}\coprod P_{out}$ into inputs and outputs and a choice for each $p\in P$ of a map
\begin{align}
  \epsilon_p^{-}:(-\infty,0]\times\bRZ & \to\Sigma  \qquad \textrm{if } p \in P_{in}\\
  \epsilon_p^{+}:[0,\infty)\times\bRZ& \to\Sigma \qquad \textrm{if } p \in P_{out}
\end{align}
which is a biholomorphism onto a punctured neighborhood of $p$ in $\overline{\Sigma}$. We moreover, make the assumption that the images of the cylindrical ends are pairwise disjoint and, in the case of one input and one output, we assume  that $\epsilon^{\pm}$ extend to biholomorphisms of $\bR\times\bRZ$ onto $\Sigma$.

A \emph{genus-$0$ Riemann surface with $n_{in}$ inputs and $n_{out}$ outputs,} is a genus-$0$ punctured Riemann surface, with $|P_{in}| = n_{in}$ and $|P_{out}| = n_{out}$, together with a choice of cylindrical ends.
\end{defin}
Changing coordinates from the half-cylinder to the disc, we see that a cylindrical end of $\Sigma$ at $p$ gives rise to a tangent vector in $T_p \Sigma$. A \emph{framed Riemann surface} is a Riemann surface equipped with the datum of a tangent vector at each marked point, up to positive real dilation. We thus have a forgetful map from the set of  genus-$0$ Riemann surfaces with cylindrical ends to framed Riemann surfaces.

Next, we discuss the notion of pre-stability, which can be formulated abstractly in terms of nodal Riemann surfaces, but which we prefer to describe prosaically in terms of trees labelled by Riemann surfaces:

\begin{defin}\label{def-pre-stable-rational-curve}
  A \emph{pre-stable rational curve $\Sigma$ with $n$ inputs and one output} consists of a pre-stable tree $T$ with $n$ inputs, and, for each vertex $v\in T$, a genus-$0$ Riemann surface $\Sigma_v$ with $|E_{in}(v)|$ inputs and $1$ output, together with a bijection between the edges adjacent to each vertex $v$ and the punctures of the corresponding Riemann surfaces:
  \begin{align}
    E(v)\to & P_{\Sigma_v} \\
    e\mapsto & p_e.
  \end{align}
We require this bijection to map the outgoing edge maps to the output, and denote the inverse bijection by $p\mapsto e_p$.

A \emph{stable rational curve with $n$ inputs and 1 output} is a pre-stable rational curve whose underlying tree $T$ is stable.
\end{defin}

We refer to the data $\left(T,\{\Sigma_v\}_{v\in T}\right)$ as \emph{the underlying pre-stable curve}, and the surfaces $\Sigma_v$ as the \emph{components} of $\Sigma$. For an internal edge $e\in E_{int}(T)$ we denote by $p^{\pm}_e$ the corresponding punctures in $\Sigma_{v^{\pm}(e)}$ respectively.

\begin{defin}\label{dfGluing0}
Consider a pre-stable rational curve $\Sigma$ with $n$ inputs and one output, with underlying tree $T$, and fix an internal edge $e$ in $T$.
  Given a parameter $r\in (0,1]$ we define the corresponding glued Riemann surface as the pre-stable curve $\Gamma_{r,e}(\Sigma)$ given by the following data:
  \begin{enumerate}
  \item The tree $T/e$ obtained from $T$ by removing $e$ and quotienting the vertex set by $v^{-}(e)\sim v^{+}(e)$. We denote by $v_{new}\in T'$ the vertex corresponding to the equivalence class of $v^{-}(e)\sim v^{+}(e)$.
    \item The surface $\Sigma_{v_{new}}$ which is defined by gluing
\begin{equation}
\left(\Sigma_{v_{+}}\setminus \epsilon^-_{p_e^+}((-\infty,\ln r)\times\bRZ)\right)\sqcup_\sim \left(\Sigma_{v_{-}}\setminus \epsilon^+_{p_e^-}((-\ln r,\infty)\times\bRZ)\right),
\end{equation}
along their boundaries by the relation
\begin{equation} \label{eq:gluing_along_boundary}
{\epsilon^-_{p_e^+}(\ln r,t)\sim \epsilon^+_{p_e^-}(-\ln r,t)}.
\end{equation}
$\Sigma_{v_{new}}$ has cylindrical ends which are induced from those on $\Sigma_{v_{in}}$ and $\Sigma_{v_{out}}$. 
\item The surfaces $\Sigma_v$ for $v \neq v_{new} \in V(T/e)$ are given by the corresponding component of $\Sigma$.
  \end{enumerate}
Let us also declare that $\Gamma_{0,e}(\Sigma)=\Sigma$, i.e. gluing with $0$ gluing parameter does not change the curve.
\end{defin}

\begin{defin}\label{def-gluing-general} Given $\vec{r}\in [0,1]^{E_{int}(T)}$, we denote by $\Gamma_{\vec{r}}(\Sigma)$ the gluing of $\Sigma$ with gluing parameter given by $\vec{r}$. That is, we pick an order on the edges of $T$ and we inductively glue each edge $e$ whose associated parameter $r_e$ is non-zero. The result does not depend on the order in which gluing is performed. For a real number $r\in [0,1]$ we write
\begin{equation}
\Gamma_r(\Sigma):=\Gamma_{(r,\dots,r)}(\Sigma).
\end{equation}
That is, we glue all the edges with the same gluing parameter $r$.\end{defin}

A \emph{biholomorphism} $\psi=(\phi,\{\psi_v\})$ of pre-stable rational curves
$(T_i,\{\Sigma_v\}_{v\in T_i}),i=1,2$ consists of an isomorphism $\phi:T_1\to T_2$ of pre-stable trees with $n$ inputs and $1$ output and for each $v\in V(T_1)$ a biholomorphism $\psi_v:\Sigma_v\to\Sigma_{\phi(v)}$ such that the map $P_{\Sigma_v}\to P_{\Sigma_{\phi(v)}}$ is the one given by $E(v)\to E(\phi(v))$. Two pre-stable rational curves with $n$ inputs and one output are said to be \emph{isomorphic} if there is a biholomorphism of the underlying pre-stable curves which intertwines the cylindrical ends at all the vertices.

Denote by $\cD_{0,{k+1}}$ the  set of isomorphism classes of pre-stable rational curves with $k$ inputs and one output. For a pre-stable tree $T$ with $k$ inputs denote by $\cD^T_{0,{k+1}}\subset \cD_{0,{k+1}}$ the subset consisting of curves whose underlying tree is $T$.

\begin{rem}Note that we consider $\cD_{0,{k+1}}$ merely as a set. We are not aware of natural Hausforff topology on $\cD_{0,{k+1}}$. On top of the difficulties caused by unstable components, even if we consider a sequence of stable curves converging (in a reasonable topology) to a stable curve, there will be an ambiguity in the choice of cylindrical ends on the limiting curve whenever breaking is involved.

For a \emph{stable} tree $T$ we do however consider the set $\cD^T_{0,{k+1}}$ of curves modeled on $T$ with its natural topology. \end{rem}

\subsection{Stabilization and the KSV moduli space} \label{sec:stab-ksv-moduli}
Denote by $\overline{\cM}_{0,k+1}$  the Deligne-Mumford moduli space of biholomorphism classes of stable genus $0$ curves with $k+1$ marked points labeled by $\{0,\dots,k\}$. We recommend the reader \cite{McDuffSalamon2012}[Appendix D] for basic definitions and properties. We call the $0^{th}$ marked point, i.e. the marked point labeled with $0$, the output marked point and the others input marked points. Note that $\overline{\cM}_{0,k+1}$ is canonically a smooth manifold (which would not be true for $g>0$) \cite{McDuffSalamon2012}[Theorem D.5.1]. It is a compactification of the space  $\cM_{0,k+1}$ of biholomorphim classes of smooth curves with $k+1$ marked points. Denote by $\cE_k$ the divisor $\overline{\cM}_{0,k+1}\setminus\cM_{0,k+1}$. It is easy to see that $\cE_k$ is a normal crossings divisor \cite{McDuffSalamon2012}[Proposition D.5.4]. For $k\geq 2$ the \emph{Kimura-Stasheff-Voronov moduli space $\underline{\cM}_{0,k+1}$} is defined as the manifold with corners obtained by doing an oriented real blow-up of $\overline{\cM}_{0,k+1}$ at $\cE_k$. Note that the interior of $\underline{\cM}_{0,k+1}$ is canonically identified with ${\cM}_{0,k+1}.$ We define the \textit{framed KSV moduli space $f\overline{\cM}_{0,k+1}^{\bR}$ }as the total space of the $(S^1)^{k+1}$-bundle over $\underline{\cM}_{0,k+1}$ given by choices of a tangent ray at each marked point. A tangent ray at a marked point $x\in \Sigma$ simply means an element of $T_x\Sigma\setminus \{0\}/\mathbb{R}_{>0}$. Note that the group $\bR/2\pi\bZ$ canonically acts on the set of tangent rays at a point of a Riemann surface.

We define the underlying topological spaces of a topological operad $f\Mbar^{\bR}_0$ by taking $f\Mbar^{\bR}_0(1):=\bRZ$ and $f\Mbar^{\bR}_0(k):=f\overline{\cM}_{0,k+1}^{\bR}$ for $k>1$.

Before we define the operations on $f\Mbar^{\bR}_0$ let us describe the points of framed KSV moduli spaces a bit more concretely. Points of $f\overline{\cM}_{0,k+1}^{\bR}$ can be described canonically as the set of equivalence classes of stable curves equipped with a choice of a tangent ray at each marked point and a pair of tangent rays  modulo simultaneous rotation at each node. That is, for any nodal point $p_1\sim p_2$ contained in irreducible components $\Sigma_1$ and $\Sigma_2$, we have a pair of tangent rays $v_i$ in $T_{p_i}\Sigma_i\setminus \{0\}/\mathbb{R}_{>0}$ modulo the equivalence relation
\begin{equation}\label{eq-tangent-ray-id}
(v_1,v_2)\sim \left(e^{i\theta}v_1,e^{-i\theta}v_2\right),\quad\theta\in\bRZ.
\end{equation}

Using this description it is straightforward to define the operations of the operad $f\Mbar^{\bR}_0$:
\begin{equation}f\Mbar^{\bR}_0(n)\times f\Mbar^{\bR}_0(k_1)\times \dots \times f\Mbar^{\bR}_0(k_n)\to f\Mbar^{\bR}_0(k_1+ \dots +k_n).\end{equation}

Namely, \begin{itemize} 
\item if $n=1$, we use the action of $f\Mbar^{\bR}_0(n)=\bRZ$ on the tangent ray of the output marked point of the element of $f\Mbar^{\bR}_0(k_1)$.
\item if $k_i=1$, we use the action of $f\Mbar^{\bR}_0(k_i)=\bRZ$ on the tangent ray of the $k_i^{th}$ input marked point of the element of $f\Mbar^{\bR}_0(n)$. \item if $k_i>1$, we identify the output marked point of $f\Mbar^{\bR}_0(k_i)$ with the $k_i^{th}$ input marked point of $f\Mbar^{\bR}_0(n)$ to obtain a stable curve and use the projection to the equivalence classes described in Equation \eqref{eq-tangent-ray-id}.\end{itemize}
The unit of $f\Mbar^{\bR}_0$ is $[0]\in \bRZ$. We omit the straightforward checking of the axioms. We call $f\Mbar^{\bR}_0$ \emph{the framed KSV operad}.

\begin{rem}
We could in fact define $f\Mbar^{\bR}_0$ as an operad over the symmetric monoidal category of manifolds with corners. We will use the fact each $f\Mbar^{\bR}_0(k)$ is by construction equipped with a manifold with corners structure below.
\end{rem}

Denoting as before by $\cD_{0,k+1}$ the set of pre-stable genus-$0$ Riemann surfaces with $k$ inputs and one output, there is a natural forgetful map
\begin{equation}
\fF:\cD_{0,k+1}\to f\overline{\cM}_{0,k+1}^{\bR},\end{equation}
which we describe now. 

Consider $\Sigma\in\cD_{0,k+1}$ with underlying tree $T$. To each bivalent vertex $v$ of $T$, we can associate a twist parameter in $\bRZ$ as follows. By assumption we have biholomorphisms \begin{align*}
  \epsilon^{-}:\bR\times\bRZ & \to\Sigma_v  \\
  \epsilon^{+}:\bR\times\bRZ& \to\Sigma_v .
\end{align*}Therefore, we obtain a biholomorphism \begin{align*} 
  (\epsilon^{-})^{-1}\circ \epsilon^{+}:\bR\times\bRZ& \to\bR\times\bRZ,
\end{align*} which is a composition of a translation and a rotation by an angle in $\bRZ$. This angle is our twist parameter.

If $k=1$, all the vertices of $T$ have to be bivalent, and we define $\fF(\Sigma)\in\bRZ$ to be the sum of the twist parameters of all the vertices in $T.$ Before we handle the $k>1$ case, we need a couple of definitions.

If we have a tree $T$ with a bivalent vertex $v$ and at least one other vertex, we define the flattening of $T$ at $v$ to be the tree obtained by removing $v$ from the vertex set and identifying the two edges adjacent to $v$ to define the new edge set.

Given a Riemann surface $\Pi$ with $k\geq 2$ inputs and $1$ output, we construct canonically tangent rays at each marked point of $\bar{\Pi}$ as follows. Let $\epsilon$ be the cylindrical end at a puncture. Let us call $s$ and $t\in\bRZ $ the coordinates in the domain of $\epsilon.$ Consider the tangent rays containing $\epsilon_*\partial_s$ obtained along $\epsilon(t=0)$. The limit of these at the marked point gives the desired tangent ray. Let us call $\bar{\Pi}$ with these tangent rays at its marked points the fine compactification of $\Pi$.

Now we go back to defining $\fF(\Sigma)$ for $k>1$. There is at least one non-bivalent vertex of $T.$ Flattening all the bivalent vertices we obtain a stable tree $T'$ whose vertices are in one-to-one correspondence with the non-bivalent vertices of $T$. The correspondence preserves valencies. To each vertex $v$ of $T'$ we can canonically associate the same Riemann surface $\Sigma_v'$ with $\upsilon(v)-1$ inputs and $1$ output. The stable curve $\Sigma'$ underlying $\fF(\Sigma)$ is obtained by taking the fine compactification of $\Sigma_v'$ at each $v\in V(T')$ and identifying the added points to form double points as prescribed by $T'$. At this point $\Sigma'$ is also equipped with the data of a tangent ray at each marked point and double point of its irreducible components.

Moreover, the edges of $T'$ correspond to maximal subtrees of $T$ with only bivalent vertices. These maximal subtrees include edges all of whose endpoint vertices have valency more than $2$. We can add up the twist parameters in each such maximal subtree to obtain an element of $\bRZ$ associated to the edges of $T'$. If there is no vertex of the associated maximal subtree we associate $[0]$ to the edge. Let us call this the total twist of an edge of $T'$.

The existing tangent rays of $\Sigma'$ and the total twists at the edges of $T'$ are then used in the following way to upgrade the stable curve $\Sigma'$ to $\fF(\Sigma)\in f\overline{\cM}_{0,k+1}^{\bR}$. If $q$ is the unique output of $\Sigma'$ with tangent ray $v_q$ and $\tau$ the total twist of its outgoing edge, we let $v_{q}'=e^{i\tau_i}v_{q}$ to be the final tangent ray at $q$. If $q$ is an input, we let $v_{q}':=e^{-i\tau}v_{q}$, where $\tau$ is now the total twist of its incoming edge. If $q$  is a double point $p^+\sim p^-$ of $\Sigma'$ with tangent rays $v_+,v_-$, we take the equivalence class 
\begin{equation}
  (e^{i\tau_i}v_+,v_-)\sim  (v_+,e^{i\tau_i}v_-),
\end{equation}where $\tau$ is the total twist of the associated internal edge.

\begin{defin}
The \emph{stabilization} of a pre-stable rational curve $\Sigma$ with $k$ inputs and $1$ output is the element $\fF(\Sigma)$ of the framed KSV moduli space. 
\end{defin}

\begin{rem}\label{rem-non-haus}
The surjective map $\fF:\cD_{0,k+1}\to f\overline{\cM}_{0,k+1}^{\bR}$ has contractible fibres, as shown in Lemma \ref{lmForSmHoEq}. In order to state a more global result, we could define a cubical set, which is roughly the singular cubical chains on $\cD_{0,k+1}$ (with the non-Hausdorff topology that we omitted defining) equipped with a map of cubical sets to the singular cubical chains of $f\overline{\cM}_{0,k+1}^{\bR}$. This map will be a homotopy equivalence. 

We give a toy example to guide the reader. Let $X$ be an arbitrary topological space and consider what one might call a non-Hausdorff boundary blow-up of $(0,1]$: the quotient space of $(0,1]\times X$ by the equivalence relation \begin{equation}(r,x)\sim (r',x')\text{ if } r=r'<1.\end{equation} The point is that the blow-down map is always a homotopy equivalence. This can be proved by first showing that the blow-up space deformation retracts to $(0,1)$. There is not a homotopy equivalence which preserves boundaries.
\end{rem}

\section{Categorical and algebraic background}
\label{sec:categ-algebr-backgr}

\subsection{Symmetric cubical sets}
\label{sec:cubical-sets}

As discussed in the introduction, we shall use the category of symmetric cubical sets as a model for the homotopy theory of spaces. Underlying this category is the category of symmetric cubes, which can be considered as a subcategory of the category of topological spaces, with objects the cubes $[0,1]^n$, and morphisms given by those maps
\begin{equation}
  [0,1]^n \to [0,1]^m,
  \end{equation}
which can be expressed as a composition of (i) projections, (ii) permutation of coordinates, and (iii) inclusion of faces corresponding to setting some coordinates equal to $0$, and others to $1$. This is not the standard definition, because it is possible to express all generators and relations completely combinatorially, but the geometrically minded reader will hopefully find this definition more amenable to their intuition.

The category of symmetric cubical sets is the category of contravariant functors from symmetric cubes to sets. While we shall use this perspective in explaining our constructions, for proofs and formal definitions, we shall often use the purely combinatorial perspective. A reference for this approach is \cite{Grandis2009}:
\begin{defin}
A \emph{symmetric cubical set} $K_*$ is a sequence of sets $\{K_n\}_{n\geq 0}$   together with a collection of \emph{face maps}
\begin{equation}
d^{\pm}_{n,i}:K_n\to K_{n-1}, \quad n\geq 1,\quad 1\leq i\leq n,\end{equation}
\emph{degeneracy maps}
\begin{equation}
  s_{n-1,i}:K_{n-1}\to K_{n}, \quad n\geq 1,\quad 1\leq i\leq n,
\end{equation}
and  \emph{transposition maps}
\begin{equation}
p_{n,i}:K_n\to K_n\quad n\geq 2,\quad 1\leq i\leq n-1.\end{equation}
These are required to satisfy the following relations for $\mu,\nu\in\{+,-\}:$
\begin{align}
d^\mu_{n-1,i}\circ d^\nu_{n,j}&=d^\nu_{n-1,j-1}\circ d^\mu_{n,i},& i<j\label{eqCubSet1},\\
s_{n,i}\circ s_{n-1,j}&=s_{n,j+1}\circ s_{n-1,i},& i\leq j\label{eqCubSet2},\\
p_{n,i}^2&=\id,\quad (p_{n,i}\circ p_{n,i+1})^3=\id,\label{eqCubSet4}\\
 p_{n,i}\circ p_{n,j}&=p_{n,j}\circ p_{n,i},& i+1<j,\label{eqCubSet5}\\
d^{\mu}_{n,i}\circ s_{n-1,j}&=\begin{cases} 
						s_{n-2,j-1}\circ d^{\mu}_{n,i},&\quad i<j,\\
						s_{n-2,j}\circ d^\mu_{n,i-1},&\quad i>j,\\
						\id,&\quad i=j,\label{eqCubSet3}
					\end{cases}\\
d^\mu_{n,j}\circ p_{n,i}&=\begin{cases}
						p_{n-1,i-1}\circ d^\mu_{n,j},&\quad j<i,\\
						d^{\mu}_{n,i+1},&\quad j=i,\\
						d^{\mu}_{n,i},&\quad j=i+1,\\
						p_{n-1,i}\circ d^\mu_{n,j},&\quad j>i+1,
					\end{cases}\\
p_{n,i}\circ s_{n-1,j}&=\begin{cases}
					s_{n-1,j}\circ p_{n,i-1},&\quad j<i,\\
					s_{n-1,i+1},&\quad j=1,\\
					s_{n-1,i},&\quad j=i+1,\\
					s_{n-1,j}\circ p_{n-1,i},&\quad j>i+1. \label{eqCubSet6}
				\end{cases}
\end{align}
A morphism of cubical sets $f:K^1_{\bullet}\to K^2_{\bullet}$ (of degree $0$) is a map of graded sets which commutes with all face, degeneracy and transposition maps. We denote the category of cubical sets with morphisms of degree $0$ by $\sqset$.  
\end{defin}

\begin{example}\label{example-symmetric-singular}
Let $X$ be a topological space and define $\Box_n(X)$ to be the set of continuous maps from the standard $n$-cube $[0,1]^n$ into $X$. For $\sigma\in \Box_n(X)$ define the face map $d^\pm_{n,i}(\sigma)$ to be the restriction $\sigma|_{x_i=(1/2\pm1/2)}\in \Box_{n-1}(X)$,  the degeneracy map  $s_i(\sigma)$ to be the composition $\sigma\circ\pi_i$ where $\pi_i:[0,1]^{n+1}\to[0,1]^n$ is the projection forgetting the $i$th component, and the transposition map $p_{n,i}(\sigma)$ to be the composition  $\sigma\circ\tau_{n,i}$ for $\tau_{n,i}:[0,1]^n\to [0,1]^n$ the map which transposes the $i$th and $(i+1)$th coordinates. Then $\Box_{\bullet}(X)$ is a symmetric cubical set. We refer to it as the \emph{set of singular (symmetric) cubes in $X$.} The reader may verify that a continuous map $f:X\to Y$ induces a morphism $f_{\bullet}:\Box_{\bullet}(X)\to \Box_{\bullet}(Y)$ of degree $0$ by mapping $\sigma\mapsto f\circ\sigma$. Thus the assignment $X\mapsto \Box_{\bullet}(X)$ is a functor $\Top\to\sqset$. 
\end{example}
We can construct a functor from symmetric cubical sets to spaces by taking a colimit weighted by the tautological standard cube functor, which assigns to a natural number $n$ the associated cube $[0,1]^n$. From the combinatorial perspective, this is given by the following expression:
\begin{defin}
Let $K_{\bullet}$ be a symmetric cubical set. The \emph{geometric realization} $|K_{\bullet}|$ is the topological space obtained by the quotient
\begin{equation}\label{eqGeoRea}
\coprod_{n}K_n\times[0,1]^n/\sim,\end{equation} 
where the relation $\sim$ is generated by 
\begin{equation}\label{eqGeoRea2}
(d^{\pm}_{n,i}(\sigma),t)\sim (\sigma,\iota^{\pm}_{n,i}(t)),\quad (s_{n,i}(\sigma),t)\sim (\sigma,\pi_{n+1,i}(t)),\quad (p_{n,i}(\sigma),t)\sim(\sigma,\tau_{n,i}(t)).\end{equation}
Here $\iota^{\pm}_{n,i}(t):[0,1]^n\to [0,1]^{n+1}$
is the standard embedding as the face $x_i=1/2\pm1/2$,  $\pi_{n+1,i}:[0,1]^{n+1}\to[0,1]^n$ is the projection forgetting the $i$th coordinate, and $\tau_{n,i}:[0,1]^n\to [0,1]^n$ is the map which transposes the $i$th and $(i+1)$th coordinates. 
\end{defin}

Note that the geometric realization is automatically a Hausdorff space. The geometric realization functor $|-|: \sqset\to\Top$ is left adjoint to the singular cubes functor $\Box_{\bullet}:\Top\to\sqset$. 
\subsubsection{The monoidal structure} \label{sec:monoidal-structure}
The category $\sqset$ carries a symmetric monoidal structure $\otimes$ which arises as a left Kan extension with respect to the functor that maps a pair of cubes to their product. Combinatorially, this may be expressed as follows: first observe that by Equations \eqref{eqCubSet4} and \eqref{eqCubSet5} the maps $p_{n,i}$ generate an action of the symmetric group $S_n$ on $K_n$ for any symmetric cubical set $K_{\bullet}$. Using this we define
\begin{equation}\label{dfSymMonStr}
(K^1\otimes K^2)_n:=\coprod_{n_1,n_2\geq 0,n_1+n_2=n}S_n\times_{S_{n_1}\times S_{n_2}}(K^1_{n_1}\times  K^2_{n_2}/\sim).\end{equation}
The equivalence relation $\sim$ is given by
\begin{equation}
(s_{n_1-1,n_1}(\sigma^1),\sigma^2)\sim (\sigma^1,s_{n_2-1,1}(\sigma^2)). \end{equation}
For $u\in S_n,x\in K^1_{n_1},y\in K^2_{n_2}$ we denote by $[u\cdot(x\otimes y)]$ the corresponding element in $K^1\otimes K^2$.

The transposition maps on $K^1\otimes K^2$ are defined by $p_i$ acting as the $i$th transposition on the $S_n$ factor. To define the face and degeneracy maps it suffices to describe the maps $d^{\pm}_{n,1}$ and $s_{n,1}$ as all the others are determined by these together with the transposition maps. For this we identify $S_n$ with the set of bijections $\underline{n}\to\underline{n}$. Define a map $\eta_n:S_n\to S_{n+1}$ by 
\begin{equation}
\eta_n(u)= \id\times u,\end{equation} 
and $\zeta_n:S_n\to S_{n-1}$ by 
\begin{equation}
\zeta_n(u)(j)=\begin{cases} 
			u(j)-1,&\quad j<u^{-1}(1)\\
			u(j+1)-1,&\quad j\geq u^{-1}(1).
		\end{cases}\end{equation}
Finally, writing $i=u^{-1}(1)$, define
\begin{align}
d_{n,1}^{\mu}\left([u\cdot(x\otimes y)]\right)&=\begin{cases}
							[\zeta_n(u)\cdot (d^\mu_{n_1,i}(x)\otimes y)],\quad i\leq n_1\\
							[\zeta_n(u)\cdot(x\otimes d^\mu_{n_2,i-n_1}(y)],\quad i>n_1.
						\end{cases}\\
s_{n,1}\left([u\cdot(x\otimes y)]\right)&=[\eta_n(u)\cdot(s_{n_1,1}(x)\otimes y].
\end{align}
The functorial perspective is useful in establishing the following result, which is analogous to the characterisation of maps with domain a tensor product of vector spaces in terms of bilinear maps:
\begin{lem}
  \label{lem:universal_property_cubical}
  There is a natural isomorphism between the set of maps of cubical sets with domain $K^1\otimes K^2 $ and target $K$, and the data of maps $K^1_{n_1} \times K^2_{n_2}$ to $K_{n_1+n_2}$ for all pairs $n_1$ and $n_2$ of integers which
  \begin{enumerate}
  \item are equivariant with respect to the $S_{n_1} \times S_{n_2}$,
  \item intertwine the face maps $d^{\pm}_{i,1} $ with $d^{\pm}_{i}$ (for $1 \leq i \leq n_1$), and $d^{\pm}_{j,2} $ with $d^{\pm}_{n_1+j}$ (for $1 \leq j \leq n_2$), and
  \item  intertwine the degeneracy maps $s_{i,1} $ with $s_{i}$ (for $1 \leq i \leq n_1+1$), and $s_{j,2} $ with $s_{n_1+j}$ (for $1 \leq j \leq n_1+1$).
  \end{enumerate} \qed
\end{lem}

\subsubsection{Homotopy between morphisms of cubical sets}
The \emph{path functor} $P:\sqset\to\sqset$ is given by setting $PK_{\bullet}=K_{\bullet+1}$  and discarding all the operations $d_{n,1}^{\pm},s_{n,1}$ and $\tau_{n,1}$. The discarded face and degeneracy maps give rise to natural transformations between the identity functor and the path functor $d^{\pm}=d_{\bullet,1}:PK_{\bullet}\to K_{\bullet}$ and $s=s_{\bullet,1}:K_{\bullet}\to PK_{\bullet}$.

\begin{defin}
  A \emph{cubical homotopy} between maps $f^{\pm} \co K^1 \to K^2$ of cubical sets is a morphism
  \begin{equation}
    H \co K^1\to P K^2    
  \end{equation}
  whose composition with $d^{\pm} $ agrees with $f^{\pm}$. We say that $f^\pm$ are homotopic if there exists a homotopy between them.  
 \end{defin} 
 
We can  describe a cubical homotopy as a map  $H:K^1_{*}\to K^2_{*+1}$ of sets satisfying 
\begin{equation}\label{eqDeg1Mor}
H\circ d^\pm_{n,i}=d^{\pm}_{n+1,i+1}\circ H,\quad H\circ s_{n,i}=s_{n+1,i+1}\circ H,\quad i=1,\dots,n,\end{equation} 
and
\begin{equation}\label{eqDeg1Mor2}
H\circ\tau_{n,i}=\tau_{n+1,i+1}\circ H\end{equation}
The morphisms $f^\pm$ corresponding to the endpoints of the interval can be recovered from a homotopy $H$ via the formula
\begin{equation}
f^{\pm}  =  d_1^{\pm}\circ H.\end{equation}

\begin{lemma}\label{lmCubHtpy}
A homotopy $H$ between morphisms $f,g:K^1_{\bullet}\to K^2_{\bullet}$ of cubical sets induces a homotopy 
\begin{equation}
|H|:[0,1]\times |K^1_{\bullet}|\to |K^2_{\bullet}|\end{equation}
between the geometric realization $|f|$ and $|g|$. 
\end{lemma}
\begin{proof}
We show that $H$ induces a continuous map $[0,1]\times |K^1|\to |K^2|$. This is defined for $s\in[0,1], (\sigma,t)\in |K^1|$ by 
\begin{equation}
(s,(\sigma,t))\mapsto(f(\sigma),(s,t)).\end{equation}
Then $|H|$ is well defined on equivalence classes under the relations given in \eqref{eqGeoRea2}. It follows that $|H|$ is continuous. Moreover $|d^{\pm}\circ H|$ is precisely the restriction of $|H|$ to $s=1/2\pm1/2$. The claim follows.
\end{proof}

\subsubsection{Symmetric normalised cubical chains}
\label{sec:symm-norm-cubic}

While most of the constructions of this paper take place at the level of cubical sets, the objects that we are ultimately interested in are formulated as algebraic structures in category of chain complexes. The key construction in passing from one category to the other is the notion of \emph{symmetric normalised cubical chains}, which is the chain complex $C_*(K)$, which in degree $n$ is the quotient of the complex freely generated by the $n$-cubes, modulo those which are in the image of a degeneracy map, and the relation which identifies a cube with the negative of its image under a transposition:
\begin{equation}
  C_n(K) \equiv \frac{\bZ[K_n]}{ \sum_{i=1}^{n} \mathrm{Im}(s_{n-1,i}:K_{n-1} \to K_n) +  \sum_{i=1}^{n-1} \mathrm{Im}(1 + p_{n,i}: K_{n} \to K_n) }.
\end{equation}
The alternating sum of the face maps define a map
\begin{equation}
  d = \sum_{i=1}^{n}  (-1)^{i}\left( d^{+}_{i} - d^{-}_i \right) :   C_n(K) \to C_{n-1}(K)
\end{equation}
which is easily seen to square to $0$. Since a map of symmetric cubical sets respects the three operations that enter in the definition of the complex $C_*(K)$, this construction yields a functor from the category of symmetric cubical sets to the category of chain complexes of abelian groups. We shall repeatedly use the fact that this functor is compatible with the monoidal structure on the two sides:

\begin{prop}[Lemma B.4 of \cite{Abouzaid2022}]
  The normalised symmetric cubical chain functor is (lax) symmetric monoidal. \qed
\end{prop}

\subsection{Multicategories}
\label{sec:multicategories}
We choose to encode multiplicative structures in Floer theory in the language of multicategories (these are also referred to as coloured operads in the literature). We introduce the basic definitions here (c.f. \cite{LeinsterBook})

\begin{defin}
  A \emph{multicategory} enriched in a symmetric monoidal category $(V,\otimes,I)$ consists of the following data
  \begin{enumerate}
  \item an object set $\cX$,
  \item an object $\cC(\vec{x},y)$ of $V$ referred to as $n$-ary multimorphisms for each object $y\in\cX$ and $n$-tuple $\vec{x}\in\cX^n$,
    \item a distinguished morphism $\id_x: I \to \cC(x,x)$, and 
    \item multicomposition  maps
\begin{equation} \label{eq:multicategory-glue-at-one-edge}
\circ_i:\cC(\vec{x},y_i)\otimes\cC(\vec{y},z)\to\cC(\vec{x}\circ_i\vec{y},z)\end{equation}
where $y_i$ is the $i$th element of $\vec{y}$, and $\vec{x}\circ_i\vec{y}$ denotes the replacement of the $i$th element in $\vec{y}$ by the sequence $\vec{x}$.
  \end{enumerate}

\begin{figure}[h]
  \centering
  \begin{tikzpicture}[baseline=(current  bounding  box.center), xscale=0.75]
    \begin{scope}[shift={(-1,0)}]
      \node at (1.5,0) {$\cC(\vec{y}, z) $};
      \draw (0,-2) -- (0,2);
    \draw (0,-2) -- (3,0);
    \draw (0,2) -- (3,0);
    \draw (3,0) -- (3.5,0) node[right]{$z$} ;
    \draw (-.25,-1.75)  node[left]{$y_{k_1}$}  -- (0,-1.75);
    \draw (-.25,-1.25)   -- (0,-1.25);
    \foreach \i in {1,...,6} {\draw (-.25,-1.25 + 0.5* \i)   -- (0,-1.25 + 0.5 *\i) ;} ;
    \draw (-.25,1.25)  -- (0,1.25) ;
    \draw (-.25,1.75) node[left]{$y_1$} -- (0,1.75);
    \end{scope}

    \begin{scope}[shift={(-4.25,0.25)}]
       \node at (1,0) {$\cC(\vec{x},y_i)$};
      \draw (0,-1) -- (0,1);
    \draw (0,-1) -- (2,0);
    \draw (0,1) -- (2,0);
    \draw (2,0) -- (2.25,0) node[right]{$y_i$};
    \draw (-.25,-.75)  node[left]{$x_{k_2}$}  -- (0,-.75);
    \draw (-.25,-.5)   -- (0,-.5);
    \draw (-.25,-.25)   -- (0,-.25);
    \draw (-.25,0)   -- (0,0);
    \draw (-.25,.25)   -- (0,.25);
    \draw (-.25,.5)  -- (0,.5) ;
    \draw (-.25,.75) node[left]{$x_1$} -- (0,.75);
  \end{scope}
  \begin{scope}[shift={(5,0)}]
     \node at (1.75,0) {$\cC(\vec{x} \circ_i \vec{y}, z)$};
      \draw (0,-3) -- (0,3);
    \draw (0,-3) -- (3.5,0);
    \draw (0,3) -- (3.5,0);
    \node[left] at (-.25,1.25) {$y_{i-1}$};
    \node[left] at (-.25,.75) {$x_{1}$};
    \node[left] at (-.25,-.25) {$x_{k_2}$};
    \node[left] at (-.25,-.75) {$y_{i+1}$};

    \draw (3.5,0) -- (4,0) node[right]{$z$} ;
    \draw (-.25,-2.75)  node[left]{$y_{k_1}$}  -- (0,-2.75);
    \foreach \i in {0,...,9} {\draw (-.25,-2.25 + 0.5* \i)   -- (0,-2.25 + 0.5 *\i) ;} ;
    \draw (-.25,2.75) node[left]{$y_1$} -- (0,2.75);
    \end{scope}
  \end{tikzpicture}
  
  \caption{Multicomposition maps}
\end{figure}
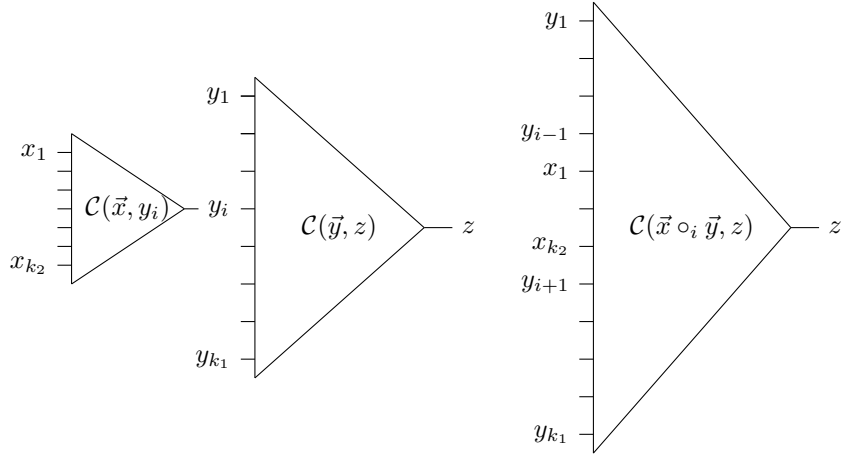

 The composition maps are required to satisfy the associativity relations
\begin{equation}\label{eqMultComp1}
a \circ_{i+j-1} (b \circ_{j} c)=(a \circ_{i} b) \circ_{j} c,\end{equation}
whenever $a\in\cC(\vec{x},y_{i}),b\in\cC(\vec{y},z_{j}),c\in\cC(\vec{z},w)$, and
\begin{equation}\label{eqMultComp2}
a \circ_{i_1} (b \circ_{i_2} c)=b \circ_{i_2+i_1-1} (a \circ_{i_1} c),\end{equation}
whenever $a\in\cC(\vec{x}_1,y_{i_1}),b\in\cC(\vec{x}_2,y_{i_2}),c\in\cC(\vec{y},z)$ and $i_1<i_2$, and,
\begin{equation}\label{eqMultComp3}
a \circ_1 \id=id \circ_i a =a.\end{equation}

\begin{figure}[h]
  \centering
  \begin{tikzpicture}[baseline=(current  bounding  box.center)]
    \begin{scope}[shift={(2,0)}]
      \node at (1.5,0) {$\cC(\vec{z},w)$};
      \draw (0,-2) -- (0,2);
    \draw (0,-2) -- (3,0);
    \draw (0,2) -- (3,0);
    \draw (3,0) -- (3.5,0) node[right]{$w$} ;
    \draw (-.25,-1.75)  -- (0,-1.75);
    \draw (-.25,-1.25)   -- (0,-1.25);
    \foreach \i in {1,...,6} {\draw (-.25,-1.25 + 0.5* \i)   -- (0,-1.25 + 0.5 *\i) ;} ;
    \draw (-.25,1.25)  -- (0,1.25) ;
    \draw (-.25,1.75) -- (0,1.75);
    \end{scope}
    \begin{scope}[shift={(-1.25,0.25)}]
       \node at (1,0) {$\cC(\vec{y},z_{j})$};
      \draw (0,-1) -- (0,1);
    \draw (0,-1) -- (2,0);
    \draw (0,1) -- (2,0);
    \draw (2,0) -- (2.5,0) node[right]{$z_{j}$};
    \draw (-.25,-.75) -- (0,-.75);
    \draw (-.25,-.5)   -- (0,-.5);
    \draw (-.25,-.25)   -- (0,-.25);
    \draw (-.25,0)   -- (0,0);
    \draw (-.25,.25)   -- (0,.25);
    \draw (-.25,.5)  -- (0,.5) ;
    \draw (-.25,.75) -- (0,.75);
    \end{scope}
\begin{scope}[shift={(-5.75,-0.25)}]
      \node at (1.5,0) {$\cC(\vec{x},y_i)$};
      \draw (0,-2) -- (0,2);
    \draw (0,-2) -- (3,0);
    \draw (0,2) -- (3,0);
    \draw (3,0) -- (3.5,0) node[right]{$y_i$} ;
    \draw (-.25,-1.75)  -- (0,-1.75);
    \draw (-.25,-1.25)   -- (0,-1.25);
    \foreach \i in {1,...,6} {\draw (-.25,-1.25 + 0.5* \i)   -- (0,-1.25 + 0.5 *\i) ;} ;
    \draw (-.25,1.25)  -- (0,1.25) ;
    \draw (-.25,1.75) -- (0,1.75);
  \end{scope}
\end{tikzpicture}
\caption{The first associativity relation}
\end{figure}
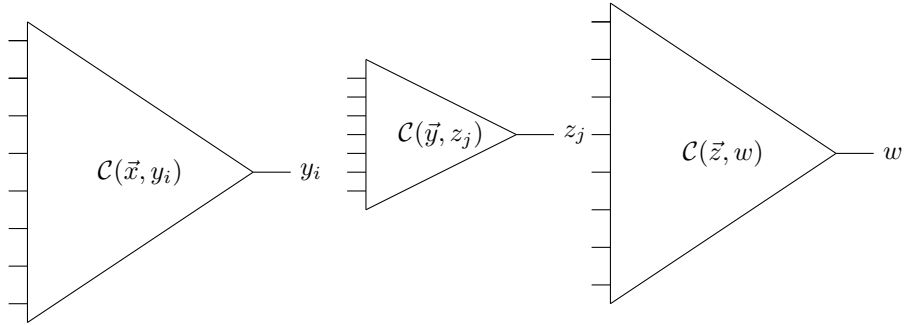

\begin{figure}[h]
  \centering
  \begin{tikzpicture}[baseline=(current  bounding  box.center)]
    \begin{scope}[shift={(2,0)}]
      \node at (1.5,0) {$\cC(\vec{y}, z)$};
      \draw (0,-2) -- (0,2);
    \draw (0,-2) -- (3,0);
    \draw (0,2) -- (3,0);
    \draw (3,0) -- (3.5,0) node[right]{$z$} ;
    \draw (-.25,-1.75)  -- (0,-1.75);
    \draw (-.25,-1.25)   -- (0,-1.25);
    \foreach \i in {1,...,6} {\draw (-.25,-1.25 + 0.5* \i)   -- (0,-1.25 + 0.5 *\i) ;} ;
    \draw (-.25,1.25)  -- (0,1.25) ;
    \draw (-.25,1.75) -- (0,1.75);
    \end{scope}
    \begin{scope}[shift={(-1.25,1.25)}]
       \node at (1,0) {$\cC(\vec{x}_1,y_{i_1})$};
      \draw (0,-1) -- (0,1);
    \draw (0,-1) -- (2,0);
    \draw (0,1) -- (2,0);
    \draw (2,0) -- (2.5,0) node[right]{$y_{i_1}$};
    \draw (-.25,-.75) -- (0,-.75);
    \draw (-.25,-.5)   -- (0,-.5);
    \draw (-.25,-.25)   -- (0,-.25);
    \draw (-.25,0)   -- (0,0);
    \draw (-.25,.25)   -- (0,.25);
    \draw (-.25,.5)  -- (0,.5) ;
    \draw (-.25,.75) -- (0,.75);
  \end{scope}
      \begin{scope}[shift={(-1.25,-1.25)}]
       \node at (1,0) {$\cC(\vec{x}_2,y_{i_2})$};
      \draw (0,-1) -- (0,1);
    \draw (0,-1) -- (2,0);
    \draw (0,1) -- (2,0);
    \draw (2,0) -- (2.5,0) node[right]{$y_{i_2}$};
    \draw (-.25,-.75) -- (0,-.75);
    \draw (-.25,-.5)   -- (0,-.5);
    \draw (-.25,-.25)   -- (0,-.25);
    \draw (-.25,0)   -- (0,0);
    \draw (-.25,.25)   -- (0,.25);
    \draw (-.25,.5)  -- (0,.5) ;
    \draw (-.25,.75) -- (0,.75);
    \end{scope}
  \end{tikzpicture}
  \caption{The second associativity relation}
\end{figure}
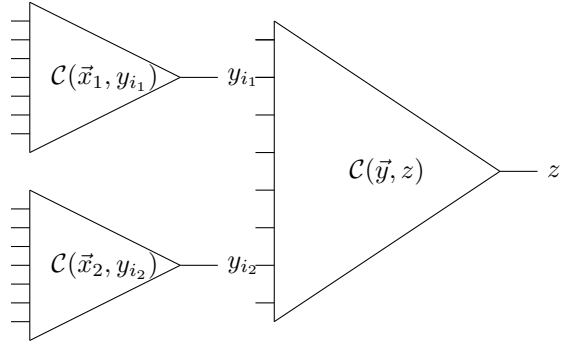

\end{defin}

\begin{defin}\label{def:symmetric-multi-category}
A \emph{symmetric multi-category} is a multi-category together with an action of the symmetric group $S_k$ on the set of $k$-morphisms, given by a map
\begin{equation}
\sigma^*:\cC(x_1,\dots,x_k;y)\to\cC(x_{\sigma(1)},\dots,x_{\sigma(k)};y)\end{equation}
for each $\sigma\in S_k$. The action is required to be compatible with multi-compositions. To formulate this consider ordered sequences  $\vec{x},\vec{y}$ of length $k,j$ respectively. The composition $\vec{y}\circ_i\vec{x}$ induces embeddings $\zeta_i:S_k\to S_{k+j-1}$ and $\eta_i:S_j\to S_{k+j-1}$. $\zeta_i(\sigma)$ acts by applying $\sigma$ to the subsequence of length $k$ starting at the $i$th place, while $\eta_i(\sigma)$ acts by treating the same sub-sequence as a single element. We require that
\begin{align}\label{eqSymAc1}
\circ_i(\vec{x},\sigma^*(\vec{y}))& =\eta_i(\sigma)^*(\circ_i(\vec{x},\vec{y})),\quad\forall\sigma\in S_j \\
\label{eqSymAc2}
\circ_i(\sigma^*(\vec{x}),\vec{y})& =\zeta_i(\sigma)^*(\circ_i(\vec{x},\vec{y})),\quad\forall\sigma\in S_k.
\end{align}
\end{defin}

The collection of all multicategories forms a category, with morphisms defined as follows: a multifunctor between multicategories $G: \cC\to \cD$ consists of\begin{itemize}
\item a map of object sets $G: Ob(\cC)\to Ob(\cD)$
\item a chain map $\cC(\vec{x};y)\to \cD(\vec{Gx};Gy)$ for each $\vec{x}\in Ob(\cC)^n$ and $y\in Ob(\cC)$.
\end{itemize}
We again omit the standard list of properties, but note that this includes a compatibility with the symmetric structure. 

A natural transformation of multifunctors $G_0,G_1: \cC\to \cD$ consists of a morphism $\cD(G_0(x); G_1(x))$ for every $x\in Ob(\cC)$ compatible with all multimorphisms.

\section{Dissipative cubes}
\label{sec:app-diss}
In this appendix we indicate the adjustments required for the geometrically bounded case. Symplectic cohomology and all its associated structures are defined and are invariant in this more general setting. However the proof of invariance is more involved and requires the Floer multi-functor to be indexed by a multi-category consisting of \emph{dissipative Floer data}. In this case, we need to involve the almost complex structures at the outset. The category of dissipative Floer data will be defined as a proper submulticategory $\cF^d_{\bullet}\subset\cF_{\bullet}$ of the multicategory $\cF_{\bullet}$ of Floer data as defined in Section \S\ref{sec:floer-functor}.

\subsection{Dissipative Floer data on cylinders}

We begin by constraining the class of almost complex structures that we shall consider; a (compatible) almost complex structure on a symplectic manifold determines a Riemannian metric, and the constraints that we impose will depend only on this data.

For a Riemannian metric $g$ on a manifold $M$ and a point $p\in M$ we denote by $\inj_g(p)$ the radius of injectivity and by $\Sec_g(p)$ the maximal sectional curvature at $p$. We drop $g$ from the notation when it is clear from the context.

\begin{defin}\label{dfIntBounded}
Let $(M,g)$ be a complete Riemannian manifold. For $a>0$, the metric $g$ is said to be \emph{$a$-bounded} at a point $p\in M$ if $\inj(x)\geq\frac1a$ and $|\Sec(x)|\leq a^2$ for all $x\in B_{1/a}(p)$.

We say that $g$ is \emph{strictly intermittently bounded}  if there is an exhaustion $K_1\subset K_2\subset \dots$ of $M$ by precompact sets and a sequence $\{a_i\}_{i\geq 1}$ of positive numbers such that the following holds.
\begin{enumerate}
    \item $d( K_i,\partial K_{i+1})> \frac1{a_i}+\frac1{a_{i+1}}.$
    \item $g$ is $a_i$-bounded on $\partial K_i$.
    \item the series obtained by adding the squares of the inverses of $a_i$ diverges:
        \begin{equation}\label{Eqtame}
            \sum_{i=1}^{\infty}\frac1{{a_i}^2}=\infty.
        \end{equation}
\end{enumerate}
The data $\{K_i,a_i\}_{i\geq 1}$ is called taming data for $(M,g)$. The open neighborhood $V=\cup_iB_{1/a_i}(K_i)$ is said to \emph{support  taming data for $g$}.
\end{defin}

More generally we allow a slight weakening in the definition:
\begin{defin}
  A Riemannian metric $g$ is \emph{intermittently bounded}, abbreviated \emph{i-bounded}, if there exists a metric $g'$ that is strictly intermittently bounded with taming data $(K_i,a_i)$, and a sequence of constants $C_i$ such that         $g$ is $C_i$-quasi-isometric to $g'$ on $B\left(\partial K_i,\frac1{a_i}\right)$ and \begin{equation}\label{Eqtame}
            \sum_{i=1}^{\infty}\frac1{{(C_ia_i)}^2}=\infty.
        \end{equation}
        In this case we will refer to the sequence $(K_i,a_i,C_i)$ as the taming data of $g$.
\end{defin}
We remind the reader that the quasi-isometry condition in the definition asserts that, on $B_{1/a_i}(\partial K_i)$, the lengths of any tangent vector $X$ with respect to $g$ and $g'$ are mutually bounded with respect to each other as follows:
        \begin{equation}
            \frac1{C_i}\|X\|_g \leq \|X\|_{g'}\leq C_i\|X\|_g
        \end{equation}

For a symplectic manifold $(M,\omega)$, an $\omega$-compatible almost complex structure $J$ is called \emph{i-bounded} if the associated metric $g_J$ is i-bounded.

We now turn our attention to the class of Hamiltonians that we will consider: given a Floer datum $(H,J)$ on a symplectic manifold $M$, we refer to a Floer solution with domain of the form  $[a,b]\times \bR/\bZ$, with $a<b\in\bR$ as a \emph{partial Floer trajectory}.

We then associate to each proper function $F:M\to\bR$, a map
\begin{equation}
   \Gamma^F_{H,J} \co \bR^2 \to \bR 
\end{equation}
defined as the infimum over all $E$ for which there is a partial Floer trajectory   $u$ of geometric energy $E$  with one end of $u$ contained in $F^{-1}([-r_1,r_1])$ and the other end in $F^{-1}(\bR\setminus (-r_2,r_2))$. Note that $\Gamma^F_{H,J}(r_1,r_2)$ may take the value of infinity.

\begin{defin}

A Floer datum $(H,J)$ is \emph{loopwise dissipative (LD)} if for some $F$ (hence any $F$) and any  $r_1$ we have $\Gamma^F_{H,J}(r_1,r)\to\infty$ as $r\to\infty$. We say that $(H,J)$ is \emph{robustly loopwise dissipative (RLD)} if  in the uniform $C^1\times C^0$ topology determined by $g_J$ there is an open neighborhood of the datum $(H,J)$ such that all elements are loopwise dissipative.
\end{defin}

We now formulate the compatibility condition between boundedness of the almost complex structure and dissipativity of the Hamiltonian:
\begin{defin}
A Floer datum $(H,J)$ is called \emph{dissipative} if
\begin{enumerate}
\item The datum $(H,J)$ is robustly loopwise dissipative.
\item 
for each $t$ the almost complex structure $J_t$ is intermittently bounded and there exist taming data which are independent of $t$ and are supported on some set $V$ such that the following properties hold:
\begin{enumerate}
\item For any $t_0,t_1$, the associated metrics $g_{J_{t_0}}$ and $g_{J_{t_1}}$ restricted to  $V$ are quasi-isometric. 
\item The restriction of $H$ to $V\times S^1$ is uniformly Lipschitz with respect to the induced metric.

\item The function $\min_{t\in \bR/\bZ}H_t(x)$ can be approximated uniformly from below by a function $\tilde{H}:M\to\bR$ whose restriction to $V$ is uniformly Lipschitz . 
\end{enumerate}
\end{enumerate}
\end{defin}
\begin{rem}
The definition of dissipativity given here is a little more restrictive than the one given in \cite{Groman2015} which instead of the Lipschitz condition on $H$ requires only that the Gromov metric determined by $J$ and $H$ be intermittently bounded. 
\end{rem}

\begin{rem}
Observe that dissipativity is an open condition. 
\end{rem}

\begin{lemma}
Suppose that $(M,\omega)$ is geometrically bounded in the sense that it carries a geometrically bounded almost complex structure $J$. If $K\subset M$ is a compact subset, then the set of Hamiltonian $H\in\cH_K$ for which there exists a $J$ so that $(H,J)$ is dissipative is cofinal in $\cH_K$. 
\end{lemma}
\begin{proof}
See \cite[Theorem 6.6, Theorem 6.10, Lemma 8.11]{Groman2015}. 
\end{proof}
\begin{rem}
It is another matter to ask for a cofinal \emph{sequence}. If $M$ is non-compact, no such sequence exists. However,  one can show that given a monotone sequence $(H_i,J_i)$ of dissipative data converging pointwise to the function $H_K$, the map from the homotopy colimit of the truncated Floer complexes over the sequence to the homotopy colimit over all of $\cH_K$ induces an isomorphism in each truncation level. 
\end{rem}

\subsection{Dissipative Floer data on general Riemann surfaces}
\label{sec:diss-floer-data}

We proceed to discuss the notion of dissipativity for multimorphisms. Let $\Sigma$ be a Riemann surface with complex structure $j_{\Sigma}$, equipped with an area form $\omega_{\Sigma}$. Let $J$ be a $\Sigma$-parametrized family of $\omega$-compatible almost complex structures on $M$  and let $H:\Sigma\times M\to \bR$ be a smooth function. We assume that for each $z\in\Sigma$ the function $H_z$ is proper and bounded below. Finally we fix a $1$-form $\alpha$ on $\Sigma$ and assume that for each  $z \in \Sigma$, we have $dH_z\wedge\alpha\geq 0$. To this data one associates an almost complex structure  on $\Sigma\times M$ defined by
\begin{equation}\label{eqGrTrick}
J_{H}:=J_M+j_{\Sigma}+X_H\otimes \alpha+JX_H\otimes \alpha\circ j_{\Sigma}.
\end{equation}
This almost complex structure is compatible with the symplectic form
\begin{equation}\label{eqGrTrick2}
\omega_{H}:=\pi_1^*\omega_{\Sigma}+\pi_2^*\omega+ d(H\alpha)
\end{equation}
on $\Sigma\times M$. We denote the induced metric on $\Sigma\times M$ by $g_{J_{H}}$. We refer to the metric $g_{J_H}$ as the \emph{Gromov metric}. We stress that \emph{the Gromov metric depends on the choice of area form on $\Sigma$}.

In order to define the condition of dissipativity for multimorphisms  we shall need to consider the notion of intermittent boundedness relative to a projection. Let $U$ be a possibly open Riemann surface equipped with an area form $\omega_U$ and a complex structure $j_U$. Let $\pi:M\times U\to U$ be a Hamiltonian fibration over $U$. Let $J$ be an almost complex structure on $M\times U$ preserving the fibers of $\pi$.
\begin{defin}
The almost complex structure $J$ is \emph{uniformly strictly intermittently bounded rel $\pi$} if there is 
\begin{enumerate}
\item an exhaustion  of $M\times U$ by subsets $K_i$ for which $\pi|_{K_i}$ is proper, and
\item positive real numbers $a_i>0$
\end{enumerate}
so that 
\begin{enumerate}
\item $d(\partial K_{i+1},K_i)>\frac1{a_i}+\frac1{a_{i+1}},$
\item  the Gromov metric is $a_i$-bounded on $\partial K_i$,  and
\item $\sum \frac1{a_i}^2=\infty$.
\end{enumerate}
We say that $J$ is \emph{uniformly intermittently bounded rel $\pi$}  if the above inequalities hold only up to constants $C_i$ as in Definition \ref {dfIntBounded}.  
\end{defin}
\begin{rem}
Note that when $U$ is a point the last definition reduces to the definition of intermittent boundedness.
\end{rem}

\begin{lemma}\label{lmDisjointSupport}
Let $J_1,J_2$ be a pair of almost complex structures on $M\times U$ which are uniformly intermittently bounded rel $\pi$ with taming data supported on sets $V_1,V_2$ respectively. Then there are subsets $W_i\subset V_i$ such that $W_1\cap W_2=\emptyset$ and such that $W_1,W_2$ still support taming data for $J_1,J_2$ respectively. 
\end{lemma}
\begin{proof}
See \cite[Theorem 4.3]{Groman2015}. 
\end{proof}

We call an area form on  a Riemann surface  $\Sigma$ with cylindrical ends \emph{admissible} if it is compatible with $j_{\Sigma}$ and it is of the form $\omega_{\Sigma}=ds\wedge dt$ in the ends. 

\begin{defin}\label{dfDisptvMultm}
Let $\{F^i=(H^i,J^i)\}_{i=0}^{n}$  be a sequence of dissipative Floer data.  
Let $\fd$ be a pre-multimorphism from $(F^1,\dots F^n)$ to $F^0$ as in Definition \ref{def:pre-multimorphism-F}, which is modeled on a tree $T$.   We say that $\fd$ is \emph{dissipative} if
\begin{enumerate}
\item
 the Floer datum $F_e$ at each edge $e$ of $T$ is dissipative,
 \item 
for each vertex $v$ of $T$  there is a finite open cover $\cU=\{U_1,\dots,U_N\}$ of $\Sigma_v$ so that fixing any admissible area form on $\omega_{\Sigma_v}$,  the induced Gromov metric on $M\times U_i$ is intermittently bounded rel the projection $\pi:M\times U\to U$ for each $i$. We assume the interior of each cylindrical end is an element of the cover $\cU$.
 \end{enumerate}
\end{defin}

Clearly, the property of being dissipative is preserved under equivalence of pre-multimorphisms. Accordingly, we refer to dissipativity as a property of multi-morphisms. 
\begin{rem}
Note that changing the admissible area form induces a Gromov metric that is quasi-isometric to the original one. The property of dissipativity thus depends only on the data of the pre-multimorphism and not on the additional datum of the area form.
\end{rem}

For each $i$ indexing an element of the cover $\cU$, let $V_i\subset M$ be the support of the taming data for the Gromov metric on $M\times U_i$. In light of Lemma \ref{lmDisjointSupport} we will always assume that the $V_i$ are pairwise disjoint. We shall refer to the union $V=\cup_iV_i$ as the support of the taming data for the Gromov metric on $\Sigma\times M$. 

The following is a criterion that we will repeatedly use:
\begin{lemma}\label{lmRelDissCriterion}
Suppose $U$ is a Riemann surface, $J$ is intermittently bounded rel $\pi:M\times U\to U$ with taming data supported on some set $V$, $\alpha$ is a closed $1$-form on $U$ and $H:M\to\bR$ is uniformly Lipschitz on $V$. Then the corresponding Gromov metric is intermittently bounded. 
\end{lemma}
\begin{proof}
See \cite[Lemma A.3]{GromanVarolgunes2021}  

\end{proof}

We now arrive at the main definition of this appendix:
\begin{defin}\label{dfDisptvCube}
A $k$-dimensional cube $\fd\in\cF_{\bullet}$ of unbroken multi-morphisms is dissipative if there is 
\begin{itemize}
\item a partition of the $k$-cube into subsets $A_1,\dots, A_N$,
\item  for each $1 \leq i \leq N$, a smooth trivialization of the family of  Riemann surfaces underlying $\fd$ which identifies each fiber with a fixed  Riemann surface $\Sigma$ with $n$ inputs and $1$ output preserving the cylindrical ends, and
\item  data taming $\fd_\sigma$ for all $\sigma\in A_i$, with respect to this identification. 
\end{itemize}

An \emph{arbitrary} cube $\fd\in\cF$ is called dissipative if  
the image of $\fd$ under the deformation retraction 
to smooth cubes as in Lemma \ref{lmsmincHoEq}\footnote{Strictly speaking, Lemma \ref{lmsmincHoEq} refers to only multimorphisms in $\cH$ without the data of almost complex structures. But for cubes of Floer data with strict gluing, the same construction goes through seamlessly.} is dissipative. 
\end{defin}

\begin{rem}
Note the difference between Definition \ref{dfDisptvMultm} where we required a finite open cover on which there is fixed taming data and Definition \ref{dfDisptvCube} where we only require a partition. The reason for this difference is that Floer's equation contains derivatives with respect to coordinates on the underlying Riemann surface but not in the direction of the parameter $\sigma$ of the cube. Thus for $C_0$ estimates we need every point on $\Sigma$ to be contained in a disc of radius $r$ bounded away from $0$ for which there is taming data. No such disc is required for the neighborhood of a point in the cube. 

\end{rem}

With these definitions in place we have the following. Denote by $\cF^d_{\bullet}\subset \cF_{\bullet}$ the subset of objects which are dissipative Hamiltonians, and the subset of morphisms which are dissipative multi-morphisms. 
\begin{prop}
$\cF^d_{\bullet}\subset \cF_{\bullet}$ forms a submulticatgory. 
\end{prop}
\begin{proof}
We need to prove that given a pair $\fd_1,\fd_2$ of dissipative multimorphisms such that the output of one matches the $i$th input of the other, the multi-composition is still dissipative. The retraction of Lemma  \ref{lmsmincHoEq} applied to the multicomposition of $\fd_1,\fd_2$ is readily seen to be the same as the one applied to the multicomposition of the retractions. This can be deduced from Equations \eqref{eqTildG}, 
\eqref{eqGlRetF} and the observation that the order of gluing at distinct nodes is immaterial. It thus does no harm to assume that $\fd_1,\fd_2$ are smooth cubes. We further point out that if a broken multimorphism is dissipative, then so is its gluing with any gluing parameter. Indeed, we obtain an open cover witnessing dissipativity  of the glued curve from the open covers of each component by keeping in mind that the ends are elements in such a cover. Observe in particular that the taming data for each element in the cover remains the same under gluing. 

 Let $k_1$ and $k_2$ respectively denote the dimensions  of $\fd_1$ and $\fd_2$. Let $\{A_1,\dots,A_{N_1}\}$ and $\{B_1,\dots,B_{N_2}\}$ be the corresponding partitions on which there are fixed taming data for $\fd_1,\fd_2$ respectively\footnote{Note that trivializations of the pair of families of smooth Riemann surfaces underlying $\fd_1,\fd_2$ induce a trivialization of the family obtained by gluing them.}.  Then by the above  observation, we have fixed taming data on $A_i\times B_j$ for the image of $\fd_1\times\fd_2$ under the retraction of Lemma  \ref{lmsmincHoEq}.  
\end{proof}
\begin{prop}
  For any dissipative cube $\fd$ and any compact set $K$ there exists a function $R=R(\fd, K,E)$ so that any pseudo-holomorphic curve, solving an equation belonging to the family  parametrised by $\fd$,  which intersect $K$ and which has geometric energy at most $E$ is contained in $B_R(K)$.  
\end{prop}
\begin{proof}
Observe that a dissipative cube  conforms to Definition 5.10 of \cite{Groman2015}, so that we can apply Theorem 6.3 of said paper. 
\end{proof}

\begin{defin}
Let $F^0=(H^0,J^0), F^1=(H^1,J^1)\dots,F^n=(H^n,J^n)$  be dissipative Floer data.  We say that  $F^0>\vec{F}$ if  $H^0>\vec{H}$ as per Definition \ref{defHamOrdering}. 
\end{defin}

\begin{prop}
Suppose $F^0>\vec{F}$. The forgetful map $\fF:\cF^d_{\bullet}(\vec{F},F)\to f\Mbar^{\bR}_{0}(k)_{\bullet}$ is a homotopy equivalence of cubical sets.
\end{prop}

\begin{rem}
The last proposition is all we need if we are willing to use virtual techniques. If we rather want to use regular data we need to show in addition that the  inclusion of $\cF^{d,reg}_{\bullet}:=\cF^d_{\bullet}\cap \cF^{reg}$ into $\cF^d_{\bullet}$ is a homotopy equivalence. The proof is the same as that of Proposition \ref{prop:regular_moduli_for_monotone}. The only thing to note is that dissipativity is an open condition on multimorphisms. In particular,  in an arbitrarily small neighborhood of a dissipative cube one can find a regular dissipative one. 
\end{rem}

By definition, the set of all dissipative cubes deformation retracts to the set of smooth dissipative cubes, so we only need to prove:
\begin{prop}\label{prpDissSmHotpyEquiv}
The restriction of the forgetful map $\fF$ to the set of smooth dissipative cubes is a homotopy equivalence.
\end{prop}

We will prove this result at the end of this Appendix, after some preliminary results. Denote by $\tilde{\cD}^{sm}_{\bullet}$ the cubical set of  curves with $n$ inputs and one output, equipped with a $1$-form $\alpha$ and whose underlying curve is smooth. The forgetful map $\tilde{\cD}^{sm}_{\bullet}\to f\Mbar^{\bR}_{0}(k)_{\bullet}$ is a homotopy equivalence as was shown during the proof of Lemma \ref{lmForSmHoEq}. Thus it remains to show that the forgetful map $\fF':\cF^d_{\bullet}(\vec{F},F)\to \tilde{\cD}^{sm}_{\bullet}$ is a homotopy equivalence. The first order of business is to construct a homotopy inverse. Unlike before, we work at the outset with cubical sets since the dissipative cubes are not the (smooth) singular cubes of a topological space.

\begin{lemma}
There is a morphism $\mathfrak{G}: \tilde{\cD}^{sm}_{\bullet}\to \cF^d_{\bullet}(\vec{F},F^0)$  such that $\fF'\circ\mathfrak{G}=\id$. 
\end{lemma}
\begin{proof}

As in the proof of Lemma \ref{lmExMonH}, we can find a  time independent function $\tilde{H}$ satisfying
\begin{equation}
\min_{t\in \bRZ}H^0_t> \tilde{H}>2^{k-1} \max_{t\in\bRZ} H^i_t
\end{equation}  
which by the definition of dissipativity can  in addition be taken to be Lipschitz with respect to $\tilde{J}=J_t$ for some $t\in S^1$ on an open set $V$ supporting taming data for $\tilde{J}$. 

Fix a cube $\sigma\mapsto (\Sigma_{\sigma},\alpha_{\sigma})$ in $\tilde{\cD}^{sm}$ where for each $\sigma$ the $1$-form $\alpha_{\sigma}$ has weights $w_{i,\sigma}\geq 2^{1-k}$. We first 
observe that by Lemma \ref{lmRelDissCriterion}, for each $\sigma$ the datum $(\tilde{H},\alpha_{\sigma},\tilde{J})$ determines an intermittently bounded Gromov metric on $\Sigma_{\sigma}\times M$. 

It remains to interpolate between $\tilde{F}=(\tilde{H},\tilde{J})$ and the $F^i$ near the ends for each $\sigma$ in a way which maintains the monotonicity and is smooth in $\sigma$. This is done in a marginally different setting in Lemma 7.6 of \cite{Groman2015}. In the interest of self containment and to prepare the ground for the proof of the main proposition we spell out the proof. 

For each end, we will carry out the interpolation within the cylindrical end. That is, at the $i$th input we need to piece together the datum $\tilde{F}=(w^i\tilde{H}dt,\tilde{J})$ with the datum $F^i=(H^idt, J^i)$ in a dissipative manner while maintaining monotonicity. It suffices to produce  a Floer datum $F_{s,t}=(H_{s,t},J_{s,t})$ on  $[0,1]\times \bRZ\times M $  such that the following conditions are satisfied
\begin{itemize}
\item $\partial_s F_{s,t}$ vanishes identically near the boundary of $[0,1]\times \bRZ\times M $, and thus extends to a Floer datum on $\bR\times \bRZ\times M $ interpolating between $(H^i,J^i)$ and $(w^i\tilde{H},J)$.
\item Denoting by $\pi: \bR\times \bRZ\times M $ the projection to $\bR\times \bRZ$, we have that the restriction of the Gromov almost complex structure $J_{H}$ to each of $\pi^{-1}((1/3,\infty)\times \bRZ)$ and to $\pi^{-1}((-\infty,2/3)\times \bRZ)$ is intermittently bounded relative to $\pi$. 
\item $\partial_sH_s\geq 0$ everywhere. 
\end{itemize}
Other than the last condition, the construction would be the same as in the proof of \cite[Theorem 4.3]{Groman2015}. We show that the monotonicity requirement  does not affect the proof. Fix two disjoint open sets $V_1$ and $V_2$ of $M$ such that there are taming data for $\tilde{J}$ and $J^i$ which are respectively supported in $[0,1]\times \bRZ\times V_j$ for $j=1,2$. We may assume that each of the $V_j$ is a disjoint union of pre-compact sets. Let $\chi: M\to[0,1]$ be a function which equals $0$ on $V_0$ and $1$ on $V_1$. Let $f:[0,1]\to[0,1]$ be a monotone function which is identically $0$ near $0$ and identically $1$ on $[1/3,1]$. Let $g:M\times [0,1]\to[0,1]$ be defined by
\begin{equation}
g(x,s)= f(1-s)(f(s)\chi(x)-1)+1
\end{equation}
Then $g$ is
\begin{itemize}
\item monotone increasing in $s$, 
\item identically $0$ for all $x$ when $s$ is near $0$,
\item identically $1$ for all $x$ when $s$ is near $1$,
\item identically $0$ on $[0,2/3]\times V_1$, and,
\item identically $1$ on $[1/3,1]\times V_2$. 
\end{itemize}
Take $H_{s,t}=g(x,s)w^i\tilde{H}_t+(1-g(x,s))H^i_t$. Then $H_{s,t}$ is also monotone increasing in $s$. Moreover, $H$ is fixed and equal to $H^i$ on $[0,2/3]\times V_1$ and to $w^i\tilde{H}$ on $[1/3,1]\times V_2$. A similar interpolation can be defined for $\tilde{J}$. To describe this by a similar formula we may use Gromov's deformation retraction from the space of metrics to the space of almost complex structures. 

All the functions $\tilde{H},\tilde{J},\chi,f,g$ can be fixed once and for all to depend only on the inputs and output. Applying the above procedure to any cube in  $\tilde{\cD}^{sm}_{\bullet}$ produces a cube in $\cF^d_{\bullet}(\vec{F},F^0)$. Note that the family produced is smooth in $\sigma$. Indeed the dependence in $\sigma$ is implicit in the choice of the Riemann surfaces $\Sigma_\sigma$ with their cylindrical end, which are assumed to be smooth in $\sigma$. The procedure clearly commutes with face, degeneracy and symmetry maps. This concludes the proof.

\end{proof}

\begin{lemma}
Let $h:[0,1]^n\to [0,1]$ be a smooth surjective function. Suppose that we are given dissipative families of unbroken multimorphisms $\fd_0,\fd_1$ in $\cF(\vec{F},F^0)$ parameterized respectively by $U_0:=h^{-1}([0,1))$ and $U_1:=h^{-1}(0,1])$, and suppose in addition that for each $\sigma$ in the intersection $U_0\cap U_1$ the projection under $\fF'$ of $\fd_0(\sigma)$ agrees with that of $\fd_1(\sigma)$. Then here exists a partition of unity $\{g,1-g\}$ on $[0,1]^n\times M$ subordinate to the cover $\{U_0\times M,U_1\times M\}$  such that $g\fd_1+(1-g)\fd_2$ is dissipative. 
\end{lemma}
\begin{proof}
Let $V_0$ and $V_1$ be disjoint be open subsets of $M$, respectively supporting taming data for $\fd_0$ and $\fd_1$. This means that $V_i$ contains the support of some local taming data for each $\sigma\in U_i$.  Let $\chi:M\to[0,1]$ be  a function which is identically $0$ on $V_0$ and identically  $1$ on $V_1$. Let $f:[0,1]\to[0,1]$ be a monotone function which is identically $0$ near $0$ and identically $1$ on $[1/3,1]$. Define a function $g:M\times [0,1]^n\to[0,1]$ by
\begin{equation}
g(x,\sigma)= f(1-h(\sigma))(f(h(\sigma))\chi(x)-1)+1.
\end{equation}
We point out that since the function $g$ is constant on $\Sigma$ for fixed $(x,\sigma)$,  $F^{\sigma}$ satisfies the monotonicity condition \eqref{eq:monotonicity_multimorphism} for each $\sigma$ so long as $F^1$ and $F^2$ do. As in the previous Lemma, $\{g,1-g\}$ is a partition of unity as desired.

\end{proof}

We have now arranged all the necessary pieces to prove our main result:

\begin{proof}[Proof of Proposition \ref{prpDissSmHotpyEquiv}]
It remains to construct a homotopy from the identity to $\mathfrak{G}\circ\fF'$. Such a homotopy involves constructing a map of cubical sets $\eta:\cF^d\to P\cF^d$ such that $d^-\circ\eta=\id$ and $d^+\circ\eta=\mathfrak{G}\circ\fF'$. 

We will proceed inductively. We first define $\eta$ on $0$-cubes. Let $\fd_0$ be any dissipative multimorphism and let $\fd_1= \mathfrak{G}\circ\fF'(\fd_1)$. Extend $\fd_0$ and $\fd_1$ trivially to $[0,1)$ and to $(0,1]$ and continue to denote this extension by $\fd_i$. Applying the previous Lemma we get a dissipative $1$-cube which is as required.

Inductively, suppose that we have defined $\eta$ on all $n-1$-cells so that it respects face, degeneracy and symmetry maps and that it is a homotopy preserving the fiber of $\fF'$. Given a non-degenerate dissipative $n$-cube $\fd$ we need to define $\eta(\fd)$ so that it agrees with $\eta$ as already defined on the faces and so that it preserves fibers. Consider the $n+1$-cube $[0,1]\times \Box^n$. We have dissipative data on its boundary defined on $[0,1]\times \partial \Box^n$ by applying $\eta$ to the boundary of $\fd_1$ and on $\{0\}\times\Box^n,\{1\}\times\Box^n$ respectively by $\fd$ and $\mathfrak{G}\circ\fF'(\fd)$. We extend this to $[0,1]^{n+1}\setminus [2\epsilon,1-2\epsilon]^{n+1}$ as follows. First we define the underlying curve and $1$-form in the unique way so that it is independent of the first coordinate. Then we fix a flow on $[0,1]^{n+1}\setminus [2\epsilon,1-2\epsilon]^{n+1}$   going from the inner to the outer boundary, for example, the Euler flow. We also fix a diffeomorphism between the family of Riemmann surfaces underlying $\fd$ and the trivial family in such a way that cylindrical ends go to cylindrical ends. We then extend the data so that it is constant along flow lines. The resulting family of multimorphisms is dissipative except that the dependence on the parameters of the cube is not smooth. To fix this, observe that dissipativity is an open condition and we can therefore replace our extension to $[0,1]^{n+1}\setminus [2\epsilon,1-2\epsilon]^{n+1}$ by an arbitrarily close one which is smooth. We also define a dissipative family on the cube $[\epsilon,1-\epsilon]^{n+1}$ by pulling back $\fd$ via the projection forgetting the first coordinate. We then interpolate the two families using the previous Lemma by taking $h$ to be a function such that $h\equiv 0$ on $[0,1]^{n+1}\setminus [\epsilon,1-\epsilon]^{n+1}$ and $h\equiv 1$ on $[2\epsilon,1-2\epsilon]^{n+1}$. It is clear that the inductive hypothesis is now satisfied for all $n+1$-cubes.

\end{proof}

\bibliographystyle{alpha}
\bibliography{large-bib}

\def\cprime{$'$}
\begin{thebibliography}{FOOO09}

\bibitem[Abo10]{Abouzaid2010}
Mohammed Abouzaid.
\newblock A geometric criterion for generating the {F}ukaya category.
\newblock {\em Publ. Math. Inst. Hautes \'Etudes Sci.}, (112):191--240, 2010.

\bibitem[Abo15]{Abouzaid2013}
Mohammed Abouzaid.
\newblock {\em Free Loop Spaces in Geometry and Topology: Including the Monograph Symplectic Cohomology and Viterbo's Theorem by Mohammed Abouzaid}, volume~24.
\newblock European Mathematical Society, 2015.

\bibitem[Abo22]{Abouzaid2022}
Mohammed Abouzaid.
\newblock An axiomatic approach to virtual chains, 2022.

\bibitem[AF91]{avramov1991homological}
Luchezar~L Avramov and Hans-Bj{\o}rn Foxby.
\newblock Homological dimensions of unbounded complexes.
\newblock {\em Journal of Pure and Applied Algebra}, 71(2-3):129--155, 1991.

\bibitem[AFT17]{AyalaFrancisTanaka2017}
David Ayala, John Francis, and Hiro~Lee Tanaka.
\newblock Factorization homology of stratified spaces.
\newblock {\em Selecta Math. (N.S.)}, 23(1):293--362, 2017.

\bibitem[AS10]{AbouzaidSeidel2010}
M.~Abouzaid and P.~Seidel.
\newblock An open string analogue of {V}iterbo functoriality.
\newblock {\em Geom. Topol.}, 14:627--718, 2010.

\bibitem[BM04]{Bourgeois-Mohnke2004}
Fr\'{e}d\'{e}ric Bourgeois and Klaus Mohnke.
\newblock Coherent orientations in symplectic field theory.
\newblock {\em Math. Z.}, 248(1):123--146, 2004.

\bibitem[BX22]{BaiXu2022}
Shaoyun Bai and Guangbo Xu.
\newblock Arnold conjecture over integers, 2022.

\bibitem[CH22]{CiriciHorel2018}
Joana Cirici and Geoffroy Horel.
\newblock {\'E}tale cohomology, purity and formality with torsion coefficients.
\newblock {\em Journal of Topology}, 15(4):2270--2297, 2022.

\bibitem[DGPZ21]{Dicksteinetal2021}
Adi Dickstein, Yaniv Ganor, Leonid Polterovich, and Frol Zapolsky.
\newblock Symplectic topology and ideal-valued measures, 2021.

\bibitem[Eil47]{Eilenberg1947}
Samuel Eilenberg.
\newblock Singular homology in differentiable manifolds.
\newblock {\em Ann. of Math. (2)}, 48:670--681, 1947.

\bibitem[EZ53]{EilenbergZilver1953}
Samuel Eilenberg and J.~A. Zilber.
\newblock On products of complexes.
\newblock {\em Amer. J. Math.}, 75:200--204, 1953.

\bibitem[FH93]{FloerHofer1993}
A.~Floer and H.~Hofer.
\newblock Coherent orientations for periodic orbit problems in symplectic geometry.
\newblock {\em Math. Z.}, 212(1):13--38, 1993.

\bibitem[FH94]{FloerHofer1994}
A.~Floer and H.~Hofer.
\newblock Symplectic homology. {I}. {O}pen sets in {${\bf C}^n$}.
\newblock {\em Math. Z.}, 215(1):37--88, 1994.

\bibitem[Flo89]{Floer1989b}
Andreas Floer.
\newblock Symplectic fixed points and holomorphic spheres.
\newblock {\em Comm. Math. Phys.}, 120(4):575--611, 1989.

\bibitem[FO01]{FukayaOno2001}
Kenji Fukaya and Kaoru Ono.
\newblock Floer homology and {G}romov-{W}itten invariant over integer of general symplectic manifolds---summary.
\newblock In {\em Taniguchi {C}onference on {M}athematics {N}ara '98}, volume~31 of {\em Adv. Stud. Pure Math.}, pages 75--91. Math. Soc. Japan, Tokyo, 2001.

\bibitem[FOOO09]{FukayaOhOhtaOno2009}
Kenji Fukaya, Yong-Geun Oh, Hiroshi Ohta, and Kaoru Ono.
\newblock {\em Lagrangian intersection {F}loer theory: anomaly and obstruction. {P}art {I}}, volume~46 of {\em AMS/IP Studies in Advanced Mathematics}.
\newblock American Mathematical Society, Providence, RI; International Press, Somerville, MA, 2009.

\bibitem[Fre09]{Fresse2009book}
Benoit Fresse.
\newblock {\em Modules over operads and functors}.
\newblock Springer, 2009.

\bibitem[GCTV12]{Galvez-Carrilloetal2012}
Imma G\'{a}lvez-Carrillo, Andrew Tonks, and Bruno Vallette.
\newblock Homotopy {B}atalin-{V}ilkovisky algebras.
\newblock {\em J. Noncommut. Geom.}, 6(3):539--602, 2012.

\bibitem[Get94]{Getzler1994}
E.~Getzler.
\newblock Batalin-{V}ilkovisky algebras and two-dimensional topological field theories.
\newblock {\em Comm. Math. Phys.}, 159(2):265--285, 1994.

\bibitem[GJ09]{goerss}
Paul~G Goerss and John~F Jardine.
\newblock {\em Simplicial homotopy theory}.
\newblock Springer Science \& Business Media, 2009.

\bibitem[GM74]{GugenheimMay1974}
Victor~KAM Gugenheim and J~Peter May.
\newblock {\em On the theory and applications of differential torsion products}, volume 142.
\newblock American Mathematical Soc., 1974.

\bibitem[Gra09]{Grandis2009}
Marco Grandis.
\newblock The role of symmetries in cubical sets and cubical categories.
\newblock {\em Cahiers de topologie et g{\'e}om{\'e}trie diff{\'e}rentielle cat{\'e}goriques}, 50(2):102--143, 2009.

\bibitem[Gro23]{Groman2015}
Yoel Groman.
\newblock Floer theory and reduced cohomology on open manifolds.
\newblock {\em Geometry \& Topology}, 27(4):1273--1390, 2023.

\bibitem[GS10]{GiansiracusaSalvatore2010}
Jeffrey Giansiracusa and Paolo Salvatore.
\newblock Formality of the framed little 2-discs operad and semidirect products.
\newblock In {\em Homotopy theory of function spaces and related topics}, volume 519 of {\em Contemp. Math.}, pages 115--121. Amer. Math. Soc., Providence, RI, 2010.

\bibitem[GV23]{GromanVarolgunes2021}
Yoel Groman and Umut Varolgunes.
\newblock Locality of relative symplectic cohomology for complete embeddings.
\newblock {\em Compositio Mathematica}, 159(12):2551--2637, 2023.

\bibitem[Hir09]{hirschhorn2009model}
Philip~S Hirschhorn.
\newblock {\em Model categories and their localizations}.
\newblock Number~99. American Mathematical Soc., 2009.

\bibitem[Hor17]{Horel2017factorization}
Geoffroy Horel.
\newblock Factorization homology and calculus {\`a} la kontsevich soibelman.
\newblock {\em Journal of Noncommutative Geometry}, 11(2):703--740, 2017.

\bibitem[Hov07]{hovey2007model}
Mark Hovey.
\newblock {\em Model categories}.
\newblock Number~63. American Mathematical Soc., 2007.

\bibitem[HV92]{HollenderVogt1997}
J.~Hollender and R.~M. Vogt.
\newblock Modules of topological spaces, applications to homotopy limits and {$E_\infty$} structures.
\newblock {\em Arch. Math. (Basel)}, 59(2):115--129, 1992.

\bibitem[KLN10]{keller2010}
Bernhard Keller, Wendy Lowen, and Pedro Nicol{\'a}s.
\newblock On the (non) vanishing of some “derived” categories of curved dg algebras.
\newblock {\em Journal of Pure and Applied Algebra}, 214(7):1271--1284, 2010.

\bibitem[KSV95]{KimuraStasheffVoronov1995}
Takashi Kimura, Jim Stasheff, and Alexander~A. Voronov.
\newblock On operad structures of moduli spaces and string theory.
\newblock {\em Comm. Math. Phys.}, 171(1):1--25, 1995.

\bibitem[Lei04]{LeinsterBook}
Tom Leinster.
\newblock {\em Higher operads, higher categories}, volume 298.
\newblock Cambridge University Press, 2004.

\bibitem[Lur17]{Lurie2017}
Jacob Lurie.
\newblock Higher algebra.
\newblock 2017.

\bibitem[LV12]{loday2012algebraic}
Jean-Louis Loday and Bruno Vallette.
\newblock {\em Algebraic operads}, volume 346.
\newblock Springer, 2012.

\bibitem[Mar08]{markl2008operads}
Martin Markl.
\newblock Operads and props.
\newblock {\em Handbook of algebra}, 5:87--140, 2008.

\bibitem[MS12]{McDuffSalamon2012}
Dusa McDuff and Dietmar Salamon.
\newblock {\em {$J$}-holomorphic curves and symplectic topology}, volume~52 of {\em American Mathematical Society Colloquium Publications}.
\newblock American Mathematical Society, Providence, RI, second edition, 2012.

\bibitem[MT78]{may1978uniqueness}
J~Peter May and Robert Thomason.
\newblock The uniqueness of infinite loop space machines.
\newblock {\em Topology}, 17(3):205--224, 1978.

\bibitem[PSS96]{PiunikhinSalamonSchwarz1996}
S.~Piunikhin, D.~Salamon, and M.~Schwarz.
\newblock Symplectic {F}loer-{D}onaldson theory and quantum cohomology.
\newblock In {\em Contact and symplectic geometry ({C}ambridge, 1994)}, volume~8 of {\em Publ. Newton Inst.}, pages 171--200. Cambridge Univ. Press, Cambridge, 1996.

\bibitem[Rez22]{Rezchikov2022}
Semon Rezchikov.
\newblock Integral arnol'd conjecture, 2022.

\bibitem[Rit13]{Ritter2013}
Alexander~F. Ritter.
\newblock Topological quantum field theory structure on symplectic cohomology.
\newblock {\em J. Topol.}, 6(2):391--489, 2013.

\bibitem[RS20]{RichterSagave2020}
Birgit Richter and Steffen Sagave.
\newblock A strictly commutative model for the cochain algebra of a space.
\newblock {\em Compos. Math.}, 156(8):1718--1743, 2020.

\bibitem[Sal19]{Salvatory2019}
Paolo Salvatore.
\newblock Planar non-formality of the little discs operad in characteristic two.
\newblock {\em Q. J. Math.}, 70(2):689--701, 2019.

\bibitem[Sei08a]{Seidel2008b}
Paul Seidel.
\newblock A biased view of symplectic cohomology.
\newblock In {\em Current developments in mathematics, 2006}, pages 211--253. Int. Press, Somerville, MA, 2008.

\bibitem[Sei08b]{Seidel2008a}
Paul Seidel.
\newblock {\em Fukaya categories and {P}icard-{L}efschetz theory}.
\newblock Zurich Lectures in Advanced Mathematics. European Mathematical Society (EMS), Z\"urich, 2008.

\bibitem[Sei12]{Seidel2012b}
Paul Seidel.
\newblock Some speculations on pairs-of-pants decompositions and {F}ukaya categories.
\newblock In {\em Surveys in differential geometry. {V}ol. {XVII}}, volume~17 of {\em Surv. Differ. Geom.}, pages 411--425. Int. Press, Boston, MA, 2012.

\bibitem[Sie19]{Siegel2019}
Kyler Siegel.
\newblock Higher symplectic capacities, 2019.

\bibitem[Smi01]{Smirnov2001}
V.~A. Smirnov.
\newblock {\em Simplicial and operad methods in algebraic topology}, volume 198 of {\em Translations of Mathematical Monographs}.
\newblock American Mathematical Society, Providence, RI, 2001.
\newblock Translated from the Russian manuscript by G. L. Rybnikov.

\bibitem[{Sta}18]{stacks-project}
The {Stacks Project Authors}.
\newblock \textit{Stacks Project}.
\newblock \url{https://stacks.math.columbia.edu}, 2018.

\bibitem[Tam03]{Tamarkin2003}
Dmitry~E. Tamarkin.
\newblock Formality of chain operad of little discs.
\newblock {\em Lett. Math. Phys.}, 66(1-2):65--72, 2003.

\bibitem[TV23]{TonkonogVarolgunes2020}
Dmitry Tonkonog and Umut Varolgunes.
\newblock Super-rigidity of certain skeleta using relative symplectic cohomology.
\newblock {\em Journal of Topology and Analysis}, 15(01):57--105, 2023.

\bibitem[Var21]{Varolgunes2018}
Umut Varolgunes.
\newblock Mayer--vietoris property for relative symplectic cohomology.
\newblock {\em Geometry \& Topology}, 25(2):547--642, 2021.

\bibitem[Ven18]{Venkatesh2018}
Sara Venkatesh.
\newblock Rabinowitz {F}loer homology and mirror symmetry.
\newblock {\em J. Topol.}, 11(1):144--179, 2018.

\bibitem[Vit99]{Viterbo1999}
C.~Viterbo.
\newblock Functors and computations in {F}loer homology with applications. {I}.
\newblock {\em Geom. Funct. Anal.}, 9(5):985--1033, 1999.

\bibitem[\v{S}10]{Severa2010}
Pavol \v{S}evera.
\newblock Formality of the chain operad of framed little disks.
\newblock {\em Lett. Math. Phys.}, 93(1):29--35, 2010.

\end{thebibliography}

\end{document}